\documentclass[10pt]{amsart}
\usepackage{amsfonts,amscd,amsthm,amsgen,amsmath,amssymb}
\usepackage{epstopdf}
\usepackage{epsfig,color}
\usepackage[all]{xy}
\usepackage[vcentermath]{youngtab}
\usepackage[letterpaper,hmargin=1.3in,vmargin=1.1in]{geometry}
\newtheorem{theorem}{Theorem}[section]
\newtheorem{lemma}[theorem]{Lemma}
\newtheorem{corollary}[theorem]{Corollary}
\newtheorem{proposition}[theorem]{Proposition}
\newtheorem{notation}[theorem]{Notation}

\theoremstyle{definition}
\newtheorem{definition}[theorem]{Definition}

\theoremstyle{remark}
\newtheorem{remark}[theorem]{Remark}

\numberwithin{equation}{section}

\begin{document}

\title[Monodromy of rank 2 twisted Hitchin systems]{Monodromy of rank 2 twisted Hitchin systems and real character varieties}
\author{David Baraglia and Laura P. Schaposnik}

\address{School of Mathematical Sciences, The University of Adelaide, Adelaide SA 5005, Australia}
\email{david.baraglia@adelaide.edu.au}
\address{Department of Mathematics, University of Illinois at Urbana-Champaign, IL 61801, USA}
\email{schapos@illinois.edu}

\begin{abstract}
We introduce a new approach for computing the monodromy of the Hitchin map and use this to completely determine the monodromy for the moduli spaces of $L$-twisted $G$-Higgs bundles, for the groups $G = GL(2,\mathbb{C}),SL(2,\mathbb{C})$ and $PSL(2,\mathbb{C})$. We also determine the twisted Chern class of the regular locus, which obstructs the existence of a section of the moduli space of $L$-twisted Higgs bundles of rank $2$ and degree $deg(L)+1$. By counting orbits of the monodromy action with $\mathbb{Z}_2$-coefficients, we obtain in a unified manner the number of components of the character varieties for the real groups $G = GL(2,\mathbb{R}), SL(2,\mathbb{R}), PGL(2,\mathbb{R}), PSL(2,\mathbb{R})$, as well as the number of components of the $Sp(4,\mathbb{R})$-character variety with maximal Toledo invariant. We also use our results for $GL(2,\mathbb{R})$ to compute the monodromy of the $SO(2,2)$ Hitchin map and determine the components of the $SO(2,2)$ character variety.
\end{abstract}
\thanks{This work is supported by the Australian Research Council Discovery Project DP110103745.}


% You can replace \today by manually entering the date (also you can manually enter to put date in format Day Month Year)
\date{\today}

% \keywords{Differential geometry, algebraic geometry}

%%%%%%%%%%%%%%%%%%%%%%%%%%%%%%%%%%%%%%%%%%%%%%%%%%%%%%%%%%%%%%%%%%%%%%%%%%%%%%%%
%%%%%%%%%%%%%%%%%%%%%%%%%%%%%%%%%%%%%%%%%%%%%%%%%%%%%%%%%%%%%%%%%%%%%%%%%%%%%%%%
%\begin{abstract}
%\end{abstract}
%%%%%%%%%%%%%%%%%%%%%%%%%%%%%%%%%%%%%%%%%%%%%%%%%%%%%%%%%%%%%%%%%%%%%%%%%%%%%%%%
%%%%%%%%%%%%%%%%%%%%%%%%%%%%%%%%%%%%%%%%%%%%%%%%%%%%%%%%%%%%%%%%%%%%%%%%%%%%%%%%

\maketitle

%%%%%%%%%%%%%%%%%%%%%%%%%%%%%%%%%%%%%%%%%%%%%%%%%%%%%%%%%%%%%%%%%%%%%%%%%%%%%%%%
%%%%%%%%%%%%%%%%%%%%%%%%%%%%%%%%%%%%%%%%%%%%%%%%%%%%%%%%%%%%%%%%%%%%%%%%%%%%%%%%
%%%%%%%%%%%%%%%%%%%%%%%%%%%%%%%%%%%%%%%%%%%%%%%%%%%%%%%%%%%%%%%%%%%%%%%%%%%%%%%%
%%%%%%%%%%%%%%%%%%%%%%%%%%%%%%%%%%%%%%%%%%%%%%%%%%%%%%%%%%%%%%%%%%%%%%%%%%%%%%%%

%%%%%%%%%%%%%%%%%%%%%%%%%%%%%%%%%%%%%%%%%%%
%%%%%%%%%%%%%%%%%%%%%%%%%%%%%%%%%%%%%%%%%%%
% SECTION
%%%%%%%%%%%%%%%%%%%%%%%%%%%%%%%%%%%%%%%%%%%
%%%%%%%%%%%%%%%%%%%%%%%%%%%%%%%%%%%%%%%%%%%
\section{Introduction}
In this paper, we introduce a new approach for computing the monodromy of the Hitchin system. Our results apply to the Hitchin fibrations of the groups $SL(2,\mathbb{C}), GL(2,\mathbb{C})$ and $PSL(2,\mathbb{C})$ and for {\em twisted Higgs bundles}, i.e. pairs $(E,\Phi)$ where the Higgs field $\Phi$ is valued in an arbitrary line bundle $L$ instead of the canonical bundle. The methods we develop here yield a number of new results concerning the topology of the regular locus of the Hitchin fibration. We summarise the main ideas of the paper below.\\

Let $\Sigma$ be a compact Riemann surface of genus $g > 1$ and $L$ a line bundle on $\Sigma$ such that either $L$ is the canonical bundle or $deg(L) > 2g-2$. We let $\mathcal{M}(r,d,L)$ be the moduli space of $L$-twisted Higgs bundles of rank $r$ and degree $d$ \cite{nit}. In \textsection \ref{sec:review}, we recall the Hitchin fibration and the construction of spectral data for twisted Higgs bundles. As with untwisted Higgs bundles, the Hitchin fibration is a map $h : \mathcal{M}(r,d,L) \to \mathcal{A}(r,L) = \bigoplus_{i=1}^r H^0(\Sigma , L^i )$ obtained by taking the characteristic polynomial of the Higgs field. We let $\mathcal{A}_{\rm reg}(r,L)$ denote the regular locus, the open subset of the base over which the fibres of the Hitchin fibration are non-singular. As recalled in \textsection \ref{secschf}, the non-singular fibres are abelian varieties. We shall denote by $\mathcal{M}_{\rm reg}(r,d,L)$  the points of $\mathcal{M}(r,d,L)$ lying over $\mathcal{A}_{\rm reg}(r,L)$, so that $\mathcal{M}_{\rm reg}(r,d,L) \to \mathcal{A}_{\rm reg}(r,L)$ is a non-singular torus bundle.\\

In \textsection\ref{sec:affine} we study the regular locus $\mathcal{M}_{\rm reg}(r,d,L)$, and  show in Theorem \ref{thmaffine1} that it has an affine structure, meaning that its transitions functions are composed of linear endomorphisms of the torus together with translations. As a consequence of the affine structure, the topology of the regular locus is completely determined by two invariants; the {\em monodromy}, describing the linear component of the transition functions and the {\em twisted Chern class}, describing the translational component. These invariants are calculated in \textsection \ref{secmths}.\\

In addition to the moduli space $\mathcal{M}(r,d,L)$ of $L$-twisted $GL(r,\mathbb{C})$-Higgs bundles, we also consider the $SL(r,\mathbb{C})$ and $PSL(r,\mathbb{C})$ counterparts, namely the moduli space $\check{\mathcal{M}}(r,D,L)$ of $L$-twisted Higgs bundles of rank $r$ and determinant $D$, and the moduli space $\hat{\mathcal{M}}(r,d,L)$ of $L$-twisted $PSL(r,\mathbb{C})$-Higgs bundles of rank $r$ and degree $d$. We study the associated Hitchin fibrations $\check{h} : \check{\mathcal{M}}(r,D,L) \to \mathcal{A}^0(r,L)$, $\hat{h} : \hat{\mathcal{M}}(r,d,L) \to \mathcal{A}^0(r,L)$, where $\mathcal{A}^0(r,L) = \bigoplus_{i=2}^r H^0(\Sigma , L^i)$, and show that the regular loci of these fibrations are again affine. Theorem \ref{thmaffine2} describes the precise relation between the monodromy and twisted Chern classes of the $GL(r,\mathbb{C})$, $SL(r,\mathbb{C})$ and $PSL(r,\mathbb{C})$-moduli spaces.\\

In \textsection \ref{secmths}, we compute the monodromy and twisted Chern classes of the $GL(2,\mathbb{C})$, $SL(2,\mathbb{C})$ and $PSL(2,\mathbb{C})$-moduli spaces. Henceforth, we restrict attention to the $r=2$ case and omit $r$ from our notation. Further, the trace of the Higgs field plays no part in the monodromy and twisted Chern class, so we may restrict to trace-free Higgs fields without loss of generality. We let $\mathcal{M}^0(d,L)$ denote the moduli space of trace-free $GL(2,\mathbb{C})$-Higgs bundles. Thus we have three moduli spaces $\mathcal{M}^0(d,L)$, $\check{\mathcal{M}}(D,L)$ and $\hat{\mathcal{M}}(d,L)$, all of which fibre over $\mathcal{A}^0(L) = H^0(\Sigma , L^2)$. Fix a basepoint $a_0 \in \mathcal{A}^0_{\rm reg}(L)$ and let $\pi : S \to \Sigma$ be the associated spectral curve (see \textsection \ref{secschf}). The monodromy of the $GL(2,\mathbb{C})$-Hitchin system $h : \mathcal{M}^0_{\rm reg}(d,L) \to \mathcal{A}^0_{\rm reg}(L)$ is the Gauss-Manin local system $R^1 h_* \mathbb{Z}$, describing the cohomology of the non-singular fibres. This is equivalent to a representation $\rho : \pi_1( \mathcal{A}_{\rm reg}(r,L) , a_0 ) \to Aut( \Lambda_S )$, where $\Lambda_S := H^1(S , \mathbb{Z})$. We also have monodromy representations $\check{\rho},\hat{\rho}$ corresponding to the $SL(2,\mathbb{C})$ and $PSL(2,\mathbb{C})$-moduli spaces, but these can be deduced from the $GL(2,\mathbb{C})$ case, so we focus attention on $\rho$.\\

To describe the monodromy representation, we find generators for $\pi_1( \mathcal{A}^0_{\rm reg}(L) , a_0 )$ in \textsection \ref{secfgc}, and compute $\rho$ on these generators in \textsection \ref{sectmr}. The regular locus of $\mathcal{A}^0(L) = H^0(\Sigma , L^2)$ coincides with the sections of $L^2$ having only simple zeros. Set $l = deg(L)$, let $\widetilde{S}^{2l} \Sigma$ be the space of positive divisors of degree $2l$ having only simple zeros and let $\tilde{\alpha} : \widetilde{S}^{2l} \Sigma \to Jac_{2l}(\Sigma)$ be the Abel-Jacobi map. Then $\mathcal{A}^0_{\rm reg}$ is a $\mathbb{C}^*$-bundle over $\tilde{\alpha}^{-1}(L^2)$. This gives a sequence
\begin{equation*}\xymatrix{
\pi_1( \mathbb{C}^* ) \ar[r] & \pi_1( \mathcal{A}^0_{\rm reg}(L) , a_0 ) \ar[r] & Br_{2l}(\Sigma , a_0 ) \ar[r]^-{\tilde{\alpha}_*} & H_1( \Sigma , \mathbb{Z} ) \ar[r] & 1,
}
\end{equation*}
where $Br_{2l}(\Sigma , a_0) = \pi_1( \widetilde{S}^{2l} \Sigma , a_0)$ is the $2l$-th braid group of $\Sigma$ \cite{bir}. It follows from Proposition \ref{propexact1} that this is an exact sequence of groups. Proposition \ref{proploopmono} shows that the monodromy action of the generator of $\pi_1(\mathbb{C}^*) = \mathbb{Z}$ acts as $\sigma^* : \Lambda_S \to \Lambda_S$, where $\sigma : S \to S$ is the sheet-swapping involution of the double cover $S \to \Sigma$. Thus it remains to find generators for $ker( \tilde{\alpha}_* )$, to lift these to $\pi_1( \mathcal{A}^0_{\rm reg}(L) , a_0 )$ and determine their monodromy action.\\

Let $b_1, \dots , b_{2l} \in \Sigma$ be the zeros of $a_0$. The spectral curve $\pi : S \to \Sigma$ associated to $a_0$ is a branched double cover, where $b_1 , \dots , b_{2l}$ are the branch points. Let $\gamma : [0,1] \to \Sigma$ be an embedded path from $b_i$ to $b_j$, $i \neq j,$ such that $\gamma$ does not meet the other branch points. From $\gamma$, we obtain a braid $s_\gamma \in Br_{2l}(\Sigma , a_0)$ by exchanging $b_i$ and $b_j$ around opposite sides of $\gamma$, while keeping all other points fixed. We call such a braid a {\em swap}. We show in Theorem \ref{thmswap} that $ker( \tilde{\alpha}_* )$ is generated by swaps. In \textsection \ref{secfgc}, we describe a lifting procedure which lifts a swap $s_\gamma$ to an element $\tilde{s}_\gamma \in \pi_1( \mathcal{A}^0_{\rm reg}(L) , a_0 )$. The central result of this paper, Theorem \ref{thmswapmono}, is a simple description of the monodromy action of $\rho( \tilde{s}_\gamma )$. Note that since $\gamma$ is an embedded path in $\Sigma$ joining two branch points, we have that the pre-image $l_\gamma = \pi^{-1}(\gamma)$ under $\pi$ is an embedded loop in $S$.
\begin{theorem}
The monodromy action of $\rho( \tilde{s}_\gamma )$ is the automorphism of $\Lambda_S = H^1(S , \mathbb{Z})$ induced by a Dehn twist of $S$ around $l_\gamma$. Let $c_\gamma \in H^1(S , \mathbb{Z})$ be the Poincar\'e dual of the homology class of $l_\gamma$. Then $\rho( \tilde{s}_\gamma )$ acts on $H^1( S , \mathbb{Z})$ as a Picard-Lefschetz transformation:
\begin{equation*}
\rho( \tilde{s}_\gamma ) x = x + \langle c_\gamma , x \rangle c_\gamma.
\end{equation*}
\end{theorem}

This gives us a complete description of the monodromy of rank $2$ twisted Hitchin systems. A system of generators for the monodromy group of the $SL(2,\mathbb{C})$-Hitchin fibration, in the untwisted case, had previously been computed by Copeland \cite{cop} for hyperelliptic Riemann surfaces and was applied in \cite{sch0,sch1} to determine the monodromy in the $SL(2,\mathbb{R})$ case. Copeland's method was combinatorial, relating the problem to computations involving a certain associated graph. The results of this paper are proved independently of \cite{cop} and \cite{sch0,sch1}, by different techniques. Moreover, our approach yields a different set of generators for the monodromy group compared with \cite{cop}, greatly facilitating the monodromy computations of subsequent sections of the paper. It should also be emphasised that while the $GL(2,\mathbb{C})$-monodromy completely determines the $SL(2,\mathbb{C})$-monodromy, the converse is not true. Thus even in the case of untwisted Higgs bundles, our computations yield new results.\\

In \textsection \ref{sectcc}, we proceed to determine the twisted Chern class, again for $r=2$. Even for the case of untwisted Higgs bundles, these have never previously been computed. From Theorem \ref{thmaffine1}, the twisted Chern class of $\mathcal{M}^0(d,L)$ depends only on the value of $d \; ( {\rm mod} \; 2)$. When $d = deg(L) \; ( {\rm mod} \; 2)$, the twisted Chern class is zero, which is most easily seen by noting that the Hitchin section maps into the degree $d = deg(L)$ component. Let $c \in H^2( \mathcal{A}^0_{\rm reg}(L) , \Lambda_S )$ denote the twisted Chern class of $\mathcal{M}^0(d,L)$, where $d = deg(L) + 1 \; ( {\rm mod} \; 2)$. Let $\Lambda_S[2] = \Lambda_S \otimes \mathbb{Z}_2 = H^1( S , \mathbb{Z}_2)$. We show that $c$ is the coboundary of a class $\beta \in H^1( \mathcal{A}^0_{\rm reg}(L) , \Lambda_S[2] )$. Such a cohomology class is represented by a map $\beta : \pi_1( \mathcal{A}^0_{\rm reg}(L) , a_0 ) \to H^1( S , \mathbb{Z}_2)$ satisfying the cocycle condition $\beta(gh) = \beta(g) + \rho(g) \beta(h)$. The second key result of this paper, Theorem \ref{thmbeta}, is a description of this cocycle on the generators of $\pi_1( \mathcal{A}^0_{\rm reg}(L) , a_0 )$:

\begin{theorem}
Let $\tau$ be the loop in $\mathcal{A}_{\rm reg}^0(L)$ generated by the $\mathbb{C}^*$-action. Then $\beta(\tau) = 0$. Let $\tilde{s}_\gamma \in \pi_1( \mathcal{A}_{\rm reg}^0(L) , a_0)$ be a lift of a swap of $b_i,b_j$ along the path $\gamma$. Then
\begin{equation*}
\beta( \tilde{s}_\gamma ) = \begin{cases} 0 & \text{if } 1 \notin \{i,j\}, \\  
c_\gamma & \text{if } 1 \in \{i,j\}. \end{cases}
\end{equation*}
\end{theorem}
We are also able to compute corresponding classes $\check{\beta},\hat{\beta}$ for the $SL(2,\mathbb{C})$ and $PSL(2,\mathbb{C})$-moduli spaces. As the results are similar to the $GL(2,\mathbb{C})$-case, we leave the details to \textsection \ref{sectcc}.\\

In Sections \textsection \ref{secmonoact}, \textsection \ref{secmonoact2}, \textsection \ref{secmonoact3}, we give explicit descriptions of the monodromy representation taken with $\mathbb{Z}_2$-coefficients. The reason for our interest in $\mathbb{Z}_2$-coefficients is the fact that points of order $2$ in the fibres of the $GL(2,\mathbb{C})$-Hitchin system correspond to $GL(2,\mathbb{R})$-Higgs bundles. Similar statements hold in the $SL(2,\mathbb{C})$ and $PSL(2,\mathbb{C})$ cases. The main result is Theorem \ref{thmmonogp}, which fully describes the group of monodromy transformations on $\Lambda_S[2] = H^1(S, \mathbb{Z}_2)$. To describe this result, let $\mathbb{Z}_2 B$ be the $\mathbb{Z}_2$-vector spaces with basis given by the set $B = \{b_1 , b_2 , \dots , b_{2l} \}$. Let $(( \; , \; ))$ be the bilinear form on $\mathbb{Z}_2 B$ given by $(( b_i , b_j )) = 1$ if $i = j$ and $0$ otherwise. Let $b_o = b_1 + b_2 + \dots + b_{2l}$ and set $W = (b_o)^\perp/(b_o)$. Note that $(( \; , \; ))$ induces a pairing on $W$ which will also be denoted by $(( \; , \; ))$. We will use $\Lambda_\Sigma[2]$ to denote $H^1(\Sigma , \mathbb{Z}_2)$ and we use $\langle \; , \; \rangle$ to denote the Weil pairings on $\Lambda_S[2]$ and $\Lambda_\Sigma[2]$. Then according to Proposition \ref{propsplit1}, we have an identification
\begin{equation*}
\Lambda_S[2] = \Lambda_\Sigma[2] \oplus W \oplus \Lambda_\Sigma[2],
\end{equation*}
under which the Weil pairing on $\Lambda_S[2]$ is given by:
\begin{equation*}
\langle ( a , b , c) , (a' , b' , c' ) \rangle = \langle a , c' \rangle + (( b , b' )) + \langle c , a' \rangle.
\end{equation*}
When $l = deg(L)$ is even, we introduce a quadratic refinement $q : \Lambda_S[2] \to \mathbb{Z}_2$ of $\langle \; , \; \rangle$ given by 
\begin{equation*}
q(a,b,c) = \langle a , c \rangle + q_W(b),
\end{equation*}
where $q_W : W \to \mathbb{Z}_2$ is the unique quadratic refinement of $(( \; , \; ))$ on $W$ for which $q_W(b_i+b_j) = 1$ for all $1 \le i < j \le 2l$. From Lemma \ref{lemphi0} and Proposition \ref{propqref}, we have that the function $q + l/2$ is the mod $2$ index on $\Lambda_S[2] = H^1(S,\mathbb{Z})$ associated to a naturally defined spin structure on $S$. We may now give the statement of Theorem \ref{thmmonogp}:

\begin{theorem}
Let $G \subseteq GL( \Lambda_{S}[2] )$ be the group generated by the monodromy action of $\rho$ on $\Lambda_{S}[2]$. Then $G$ is isomorphic to a semi-direct product $G = S_{2l} \ltimes H$ of the symmetric group $S_{2l}$ and the group $H$ described below. The symmetric group $S_{2l}$ acts on $W$ through permutations of the set $B$. Let $K$ be the subgroup of elements of $GL( \Lambda_S[2])$ of the form:
\begin{equation*}
\left[ \begin{matrix} I_{2g} & A & B \\ 0 & I & A^t \\ 0 & 0 & I_{2g} \end{matrix} \right],
\end{equation*}
where $A : W \to \Lambda_\Sigma[2]$, $B : \Lambda_\Sigma[2] \to \Lambda_\Sigma[2]$,  and $A^t : \Lambda_\Sigma[2] \to W$ is the adjoint of $A$, so $\langle Ab , c \rangle = (( b , A^tc ))$. Then:
\begin{enumerate}
\item{If $l$ is odd then $H$ is the subgroup of $K$ preserving the intersection form $\langle \; , \; \rangle$, or equivalently, the elements of $K$ satisfying: 
\begin{equation*}
\langle Bc , c' \rangle + \langle Bc' , c \rangle + \langle A^tc , A^t c' \rangle = 0.
\end{equation*}
}
\item{If $l$ is even then $H$ is the subgroup of $K$ preserving the quadratic refinement $q$ of $\langle \; , \; \rangle$, or equivalently, the elements of $K$ satisfying:
\begin{equation*}
\langle Bc , c \rangle + q_W( A^t c ) = 0.
\end{equation*}
}
\end{enumerate}
\end{theorem}
Sections \textsection \ref{secmonoact2} and \textsection \ref{secmonoact3} consider the monodromy action on some closely related representations, relevant to our study of real Higgs bundles in the later sections of the paper.\\

In \textsection \ref{secrthbm} we consider moduli spaces of $L$-twisted Higgs bundles corresponding to the real groups $GL(2,\mathbb{R})$, $SL(2,\mathbb{R})$, $PGL(2,\mathbb{R})$ and $PSL(2,\mathbb{R})$ and study the monodromy of the associated Hitchin fibrations. For these groups, the non-singular fibres of the Hitchin fibration are affine spaces over certain $\mathbb{Z}_2$-vector spaces. The regular loci of these real moduli spaces are thus certain covering spaces of $\mathcal{A}^0_{\rm reg}(L)$. Using spectral data, we give in Proposition \ref{proprsd} a precise description of the fibres. This allows us to describe the regular loci in terms of the monodromy representation of $\rho$ with $\mathbb{Z}_2$-coefficients, as studied in Sections \textsection \ref{secmonoact}, \textsection \ref{secmonoact2}, \textsection \ref{secmonoact3}. Assocated to Higgs bundles for a real group are certain topological invariants which can be used to distinguish connected components of the moduli spaces. Proposition \ref{propspecinv} gives a description of these invariants in terms of spectral data, hence in terms of monodromy representations.\\

In \textsection \ref{seccrcv}, we use our monodromy calculations to compute the number of connected components of the moduli space of $L$-twisted real Higgs bundles for the groups $GL(2,\mathbb{R})$, $SL(2,\mathbb{R})$, $PGL(2,\mathbb{R})$ and $PSL(2,\mathbb{R})$. We introduce the notion of maximal components for $L$-twisted Higgs bundles, generalising the notion of maximal representations to the $L$-twisted setting. We determine the number of maximal components in Corollary \ref{cormax1}. We show in Proposition \ref{propconnreg} that every connected component of these moduli spaces meets the regular locus. Hence the number of orbits of the monodromy gives an upper bound for the number of connected components of the moduli space. On the other hand we have a lower bound on the number of components given by counting the number of maxmal components plus the number of possible values for the topological invariants of non-maximal components. We show in Theorem \ref{thmcomponents} that these numbers coincide and thus give the number of connected components, which are:
\begin{theorem}
Suppose that $L = K$ or $l = deg(L) > 2g-2$ and that $l$ is even. The number of connected components of the $L$-twisted real Higgs bundle moduli spaces are as follows:
\begin{enumerate}
\item{$3.2^{2g} + (l-4)/2$ for $GL(2,\mathbb{R})$}
\item{$2.2^{2g}+(l-1)$ for $SL(2,\mathbb{R})$}
\item{$2^{2g}+l/2$ for $PGL(2,\mathbb{R})$ of degree $0$ and $2^{2g}+l/2-1$ for $PGL(2,\mathbb{R})$ of degree $1$}
\item{$l+1$ for $PSL(2,\mathbb{R})$ of degree $0$ and $l$ for $PSL(2,\mathbb{R})$ of degree $1$.}
\end{enumerate}
\end{theorem}

Let $Rep(G)$ denote the character variety of reductive representations of $\pi_1(\Sigma)$ in $G$ (see \textsection \ref{secrcv}). The non-abelian Hodge correspondence gives homeomorphisms between character varieties of reductive groups and certain moduli spaces of untwisted Higgs bundles. Applying Theorem \ref{thmcomponents}, we immediately have:
\begin{corollary}
For the following real character varieties, the number of connected components are:
\begin{enumerate}
\item{$3.2^{2g}+g-3$ for $Rep(GL(2,\mathbb{R}))$}
\item{$2.2^{2g}+2g-3$ for $Rep(SL(2,\mathbb{R}))$}
\item{$2^{2g}+g-1$ for $Rep_0(PGL(2,\mathbb{R}))$ and $2^{2g}+g-2$ for $Rep_1(PGL(2,\mathbb{R}))$}
\item{$2g-1$ for $Rep_0(PSL(2,\mathbb{R}))$ and $2g-2$ for $Rep_1(PSL(2,\mathbb{R}))$.}
\end{enumerate}
\end{corollary}
The number of components for $Rep(SL(2,\mathbb{R}))$ and $Rep(PSL(2,\mathbb{R}))$ were obtained by Goldman in \cite{goldm2} and the number of components of $Rep(PSL(2,\mathbb{R}) )$ by Xia in \cite{xia1,xia2}. To the best of the authors' knowledge, the number of components for $Rep(GL(2,\mathbb{R}))$ has not previously appeared in the literature.\\

In a similar manner, we have a correspondence between representations of $Sp(4,\mathbb{R})$ with maximal Toledo invariant and $K^2$-twisted $GL(2,\mathbb{R})$-Higgs bundles. This immediately gives a new proof of the following:
\begin{corollary}
The number of components of $Rep_{2g-2}(Sp(4,\mathbb{R}))$ is given by $3.2^{2g}+2g-4$.
\end{corollary}

Finally, in \textsection \ref{secso22} we apply our results on the monodromy for $GL(2,\mathbb{R})$ Higgs bundles to determine the monodromy of the $SO(2,2)$-Hitchin fibration. In particular, this allows us to compute the number of components of the character variety $Rep( SO(2,2) )$, by counting orbits of the monodromy:
\begin{corollary}
The number of components of $Rep(SO(2,2))$ is given by $6.2^{2g} + 4g^2 - 6g - 3$.
\end{corollary}

%%%%%%%%%%%%%%%%%%%%%%%%%%%%%%%%%%%%%%%%%%%
%%%%%%%%%%%%%%%%%%%%%%%%%%%%%%%%%%%%%%%%%%%
% SECTION
%%%%%%%%%%%%%%%%%%%%%%%%%%%%%%%%%%%%%%%%%%%
%%%%%%%%%%%%%%%%%%%%%%%%%%%%%%%%%%%%%%%%%%%
\section{Review of the Hitchin system}\label{sec:review}

\subsection{Twisted Higgs bundles}\label{secthb}

Let $\Sigma$ be a compact Riemann surface of genus $g > 1$ and let $L$ be a line bundle on $\Sigma$. An $L$-twisted {\em Higgs bundle} is a pair $(E,\Phi$), where $E$ is a holomorphic vector bundle and $\Phi$ is a holomorphic section of $End(E) \otimes L$, called the {\em Higgs field}. The case where $L$ is the canonical bundle $K := T^*\Sigma$, corresponds to the usual definition of Higgs bundles as defined by Hitchin and Simpson \cite{hit1,hit2,sim1,sim2}. One can define notions of stability and $S$-equivalence for twisted Higgs bundles in exactly the same way as for ordinary Higgs bundles. We let $\mathcal{M}(r,d,L)$ denote the moduli space of $S$-equivalence classes of semi-stable $L$-twisted Higgs bundles $(E,\Phi)$, where $E$ has rank $r$ and degree $d$. Nitsure constructed $\mathcal{M}(r,d,L)$ as a quasi-projective complex algebraic variety \cite{nit}.\\

Let $l = deg(L)$ be the degree of $L$. Throughout we will assume that either $L = K$ or $l > 2g-2$. Under these conditions, the dimension of $\mathcal{M}(r,d,L)$ is $r^2 l + 1 + dim( H^1(\Sigma , L ) )$ \cite[Proposition 7.1]{nit}. We let $\mathcal{M}^0(r,d,L)$ be the subvariety of $\mathcal{M}(r,d,L)$ consisting of pairs $(E,\Phi)$ with trace-free Higgs field. Any $\Phi$ can be written in the form $\Phi = \Phi_0 + \frac{\mu}{r} Id$, where $\Phi_0$ is trace-free and $\mu = tr(\Phi) \in H^0(\Sigma , L)$. Thus we have an identification $\mathcal{M}(r,d,L) \simeq \mathcal{M}^0(r,d,L) \times H^0(\Sigma , L)$. It follows by Riemann-Roch that the dimension of $\mathcal{M}^0(r,d,L)$ is $(r^2-1)l +g$.\\

For a line bundle $D$ of degree $d$, we let $\check{\mathcal{M}}(r,D,L) \subseteq \mathcal{M}^0(r,d,L)$ be the subvariety of pairs $(E,\Phi)$ where $\Phi$ is trace-free and $\det(E) = D$. The dimension of $\check{\mathcal{M}}(r,D,L)$ is $(r^2-1)l$. For any line bundle $M$, the tensor product $(E,\Phi) \mapsto (E \otimes M , \Phi \otimes Id )$ defines an isomorphism $\otimes M : \check{\mathcal{M}}(r,D,L) \to \check{\mathcal{M}}(r,D \otimes M^r,L)$. This shows that as an algebraic variety $\check{\mathcal{M}}(r,D,L)$ depends on $D$ only through the value of $d = deg(D)$ modulo $r$.\\

We say that two trace-free $L$-twisted Higgs bundles $(E,\Phi) , (E' ,\Phi')$ are projectively equivalent if $(E,\Phi)$ is isomorphic to $(E' \otimes A , \Phi' \otimes Id )$ for some line bundle $A$. In this paper we define an $L$-twisted $PGL(r,\mathbb{C})$-Higgs bundle to be the projective equivalence class of a trace-free $L$-twisted Higgs bundle. Note that such an equivalence class $[ (E,\Phi) ]$ has a well-defined degree $d = deg(E)$ modulo $r$. Let $D$ be a fixed line bundle of degree $d$. Then every $L$-twisted $PGL(r,\mathbb{C})$-Higgs bundle of degree $d$ has a representative $(E,\Phi)$ for which $det(E) = D$. This representative is unique up to the tensor product action of $\Lambda_\Sigma[r] := Jac(\Sigma)[r]$, the group of line bundles on $\Sigma$ of order $r$. We let $\hat{\mathcal{M}}(r,d,L)$ denote the moduli space of $S$-equivalence classes of $L$-twisted $PGL(r,\mathbb{C})$-Higgs bundles of degree $d$. This may either be viewed as the quotient of $\mathcal{M}^0(r,d,L)$ by the action of $Jac(\Sigma)$ or as the quotient of $\check{\mathcal{M}}(r,D,L)$ by the finite group $\Lambda_\Sigma[r]$. Clearly $\hat{\mathcal{M}}(r,d,L)$ has dimension $(r^2-1)l$.

%%%%%%%%%%%%%%%%%%%%%%%%%%%%%%%%%%%%%%%%%%%
%%%%%%%%%%%%%%%%%%%%%%%%%%%%%%%%%%%%%%%%%%%
\subsection{Spectral curves and the Hitchin fibration}\label{secschf}

Set $\mathcal{A}(r,L) = H^0( \Sigma , L) \oplus H^0(\Sigma , L^2) \oplus \dots \oplus H^0(\Sigma , L^r)$. As with ordinary Higgs bundles, taking coefficients of the characteristic polynomial of $\Phi$ gives a map $h : \mathcal{M}(r,d,L) \to \mathcal{A}(r,L)$ called the {\em Hitchin map} or {\em Hitchin fibration} \cite{hit2}. More precisely if $(E,\Phi) \in \mathcal{M}(r,d,L)$, then we set $h(E,\Phi) = (a_1,a_2, \dots , a_r)$, where the characteristic polynomial of $\Phi$ is:
\begin{equation*}
det(\lambda - \Phi ) = \lambda^r + a_1 \lambda^{r-1} + \dots + a_r.
\end{equation*}
Thus $a_j \in H^0(\Sigma , L^j)$ is given by $a_j = (-1)^j Tr ( \wedge^j \Phi : \wedge^j E \to \wedge^j E \otimes L^j )$. Note that since $a_1 = - Tr( \Phi)$, we find that $h$ sends $\mathcal{M}^0(r,d,L)$ to the subspace $\mathcal{A}^0(r,L) = H^0( \Sigma , L^2) \oplus H^0(\Sigma , L^3) \oplus \dots \oplus H^0(\Sigma , L^r)$. Similarly we have Hitchin maps $\check{h} : \check{\mathcal{M}}(r,D,L) \to \mathcal{A}^0(r,L)$ and $\hat{h} : \hat{\mathcal{M}}(r,d,L) \to \mathcal{A}^0(r,L)$.\\

There is an action $\theta : H^0( \Sigma , L ) \times \mathcal{M}(r,d,L) \to \mathcal{M}(r,d,L)$ of $H^0(\Sigma , L)$ on $\mathcal{M}(r,d,L)$ given by $\theta( \mu , (E, \Phi) ) = (E , \Phi - (\mu/r) Id )$ and a corresponding action $\theta_{\mathcal{A}} : H^0( \Sigma , L) \times \mathcal{A}(r,L) \to \mathcal{A}(r,L)$ of the form $\theta( \mu , (a_1 , a_2 , \dots , a_r) ) = (a'_1 , a'_2 , \dots , a'_r)$, where $(a'_1  , \dots , a'_r)$ is determined by
\begin{equation*}
( \lambda + \mu/r)^r + a_1 ( \lambda + \mu/r)^{r-1} + \dots + a_r = \lambda^r + a'_1 \lambda^{r-1} + \dots + a'_r,
\end{equation*}
in particular, $a'_1 = a_1 + \mu$. The Hitchin map intertwines the two actions. It is clear that the map $f : H^0(\Sigma , L) \oplus \mathcal{A}^0(r,L) \to \mathcal{A}(r,L)$ given by $f(\mu , a) = \theta_{\mathcal{A}}(\mu , a)$ is an isomorphism of complex algebraic varieties. Define $p : \mathcal{A}(r,L) \to \mathcal{A}^0(r,L)$ by $p(a) = p_2( f^{-1}(a) )$, where $p_2 : H^0(\Sigma , L) \oplus \mathcal{A}^0(r,L) \to \mathcal{A}^0(r,L)$ is the projection to the second factor. Then $\mathcal{M}(r,d,L) \to \mathcal{A}(r,L)$ may be identified with the pullback of $\mathcal{M}^0(r,d,L) \to \mathcal{A}^0(r,L)$ under the map $p$. This will allow us to mostly consider $\mathcal{M}^0(r,d,L)$ instead of the larger space $\mathcal{M}(r,d,L)$.\\

Under our assumptions on $L$, the generic fibre of the Hitchin fibration is an abelian variety. To see this, we recall the contruction of spectral curves from \cite{hit2,bnr}. Let $a = (a_1 , a_2 , a_3 , \dots , a_r) \in \mathcal{A}(r,L)$. We let $\pi : L \to \Sigma$ denote the projection from the total space of $L$ to $\Sigma$ and let $\lambda$ denote the tautological section of $\pi^*(L)$. Define $s_a \in H^0(K , \pi^*(L^r))$ by:
\begin{equation}
s_a = \lambda^r + \pi^*(a_1) \lambda^{r-1} + \dots + \pi^*(a_r).
\end{equation}
The zero set $S_a \subset L$ of $s_a$ is called the {\em spectral curve} associated to $a$. Our assumptions on $L$ together with Bertini's theorem implies that $S_a$ is smooth for generic points in $\mathcal{A}(r,L)$. Let $\mathcal{A}_{\rm reg}(r,L)$ denote the Zariski open subset of points of $\mathcal{A}(r,L)$ for which the corresponding spectral curve is smooth and let $\mathcal{M}_{\rm reg}(r,d,L)$ denote the points of $\mathcal{M}(r,d,L)$ lying over $\mathcal{A}_{\rm reg}(r,L)$. Similarly define $\mathcal{A}^0_{\rm reg}(r,L) \subset \mathcal{A}^0(r,L)$ and corresponding open subsets $\mathcal{M}^0_{\rm reg}(r,d,L),\check{\mathcal{M}}_{\rm reg}(r,D,L),\hat{\mathcal{M}}_{\rm reg}(r,d,L)$. To simplify notation we will write $S$ for the spectral curve whenever the point $a \in \mathcal{A}(r,L)$ is understood. We then denote the restriction of $\pi$ to $S$ simply as $\pi$. For any $a \in \mathcal{A}_{\rm reg}(r,L)$, we have thus constructed a degree $r$ branched cover $\pi : S \to \Sigma$.\\

The fibres of the Hitchin system may be described in terms of certain line bundles on $S$ as follows. Given a line bundle $M$ on $S$, consider the rank $r$ vector bundle $E = \pi_*(M)$. The tautological section $\lambda$ defines a map $\lambda : M \to M \otimes \pi^* L$, which pushes down to a map $\Phi : E \to E \otimes L$, giving an $L$-twisted Higgs bundle pair $(E,\Phi)$. As in \cite{bnr}, one finds that the characteristic polynomial is $s_a$, so that $(E,\Phi)$ lies in the fibre of the Hitchin map over $a$. Conversely any $L$-twisted Higgs bundle $(E,\Phi)$ with characteristic polynomial $a$ corresponds to some line bundle $M$ on $S$ \cite{bnr}.\\

Let $K_S$ denote the canonical bundle of $S$. By the adjunction formula we have $K_S \cong \pi^*(K \otimes L^{r-1})$. It follows that for any line bundle $M$ on $S$ we have:
\begin{equation*}
det( \pi_*(M) ) = Nm( M ) \otimes L^{-r(r-1)/2},
\end{equation*}
where $Nm : Pic(S) \to Pic(\Sigma)$ is the norm map. Let $\tilde{d} := d + lr(r-1)/2$ and let $Jac_{\tilde{d}}(S)$ be the degree $\tilde{d}$ line bundles on $S$. For $M \in Jac_{\tilde{d}}(S)$ it follows that $E = \pi_*(M)$ has degree $d$. By the discussion above, the correspondence $M \mapsto (E,\Phi)$ identifies the fibre of $\mathcal{M}(r,d,L)$ over $a \in \mathcal{A}_{\rm reg}(r,L)$ with $Jac_{\tilde{d}}(S)$, which is a torsor over $Jac(S)$. In a similar manner, the fibre of $\check{\mathcal{M}}(r,D,L)$ over $a$ may be identified with $\{ M \in Jac_{\tilde{d}}(S) \; | \; Nm(M) = D \otimes L^{r(r-1)/2} \}$. This is a torsor over the Prym variety:
\begin{equation*}
Prym(S,\Sigma) := \{ M \in Jac(S) \; | \; Nm(M) = \mathcal{O} \}.
\end{equation*}

The fibre of $\hat{\mathcal{M}}(r,d,L)$ over $a$ may be identified with the quotient of $Jac_{\tilde{d}}(S)$ under the tensor product action of $\pi^*( Jac(\Sigma) )$. This is a torsor over the abelian variety:
\begin{equation*}
\hat{ Prym}(S,\Sigma) := Jac(S)/ \pi^*( Jac(\Sigma)) \simeq Prym(S,\Sigma) / \Lambda_\Sigma[r],
\end{equation*}
which is the dual abelian variety of $Prym(S,\Sigma)$.\\

In this paper we are mainly concerned with the case $r=2$. In this case the spectral curve $\pi : S \to \Sigma$ is a branched double cover, so there is a naturally defined involution $\sigma : S \to S$ which exchanges the two sheets of the cover. Let $\sigma^* : Pic(S) \to Pic(S)$ be the pullback. By considering the action of $\sigma$ on divisors, it is clear that for any $M \in Pic(S)$, one has
\begin{equation}\label{equnorm1}
\sigma^*(M) \otimes M = \pi^*( Nm(M) ).
\end{equation}
In particular, we have $Prym(S,\Sigma) = \{ M \in Pic(S) \; | \; \sigma^*(M) = M^* \}$.

%%%%%%%%%%%%%%%%%%%%%%%%%%%%%%%%%%%%%%%%%%%%
%%%%%%%%%%%%%%%%%%%%%%%%%%%%%%%%%%%%%%%%%%%%
% SECTION
%%%%%%%%%%%%%%%%%%%%%%%%%%%%%%%%%%%%%%%%%%%%
%%%%%%%%%%%%%%%%%%%%%%%%%%%%%%%%%%%%%%%%%%%%
\section{Affine structure of the regular locus}\label{sec:affine}

%%%%%%%%%%%%%%%%%%%%%%%%%%%%%%%%%%%%%%%%%%%%%
\subsection{Affine torus bundles}
%%%%%%%%%%%%%%%%%%%%%%%%%%%%%%%%%%%%%%%%%%%%%

Let $\Lambda$ be a rank $n$ lattice, $\mathfrak{t} := \Lambda \otimes_{\mathbb{Z}} \mathbb{R}$ and $T := \mathfrak{t}/\Lambda$. Let $Aut(T)$ be the automorphism group of the Lie group $T$. We define ${\rm Aff}(T)$, the group of affine transformations of $T$ to be the semi-direct product ${\rm Aff}(T) = Aut(T) \ltimes T$ which acts on $T$ by affine transformations:
\begin{equation*}
(g,s)t = g(t)s,
\end{equation*}
where $(g,s) \in Aut(T) \ltimes T$, $t \in T$. An {\em affine torus bundle} over a topological space $B$, is a locally trivial torus bundle $f : X \to B$ with structure group ${\rm Aff}(T)$. Equivalently, $X$ is the bundle $X = P \times_{{\rm Aff}(T)} T$ associated to a principal ${\rm Aff}(T)$-bundle $P \to B$.\\

If $P \to B$ is a principal ${\rm Aff}(T)$-bundle, then the quotient $P/T$ of $P$ by the subgroup $T \subset {\rm Aff}(T)$ is a principal ${\rm Aff}(T)/T = Aut(T)$-bundle. Since $Aut(T)$ is discrete, such bundles correspond to representations $\rho : \pi_1( B , b_o ) \to Aut(T)$. Given such a representation $\rho$, we let $\Lambda_\rho$ be the local system associated to $\rho$ through the action of $Aut(T)$ on the lattice $\Lambda = H_1( T , \mathbb{Z} )$. Lifts of the principal $Aut(T)$-bundle associated to $\rho$ to a principal ${\rm Aff}(T)$-bundle are classified by $H^2( B , \Lambda_\rho)$. In this way, we obtain the following classification (see \cite{bar1t,bar2t}):

\begin{proposition}
Affine torus bundles on a locally contractible, paracompact space $B$ are in bijection with equivalence classes of pairs $(\rho , c )$, where:
\begin{itemize}
\item[(1)]{$\rho$ is a representation $\rho : \pi_1(B,b_o) \to Aut(T)$, called the monodromy.}
\item[(2)]{$c$ is a class in $H^2( B , \Lambda_\rho)$, called the twisted Chern class.}
\end{itemize}
Two pairs $(\rho_1 , c_1 ), (\rho_2 , c_2 )$ are equivalent if there is an isomorphism $\phi : \Lambda_{\rho_1} \to \Lambda_{\rho_2}$ of local systems for which $\phi(c_1) = c_2$.
\end{proposition}

\begin{remark}
Let $f : X \to B$ be the affine torus bundle associated to $(\rho , c)$.
\begin{itemize}
\item[(1)]{The local system $\Lambda_\rho$ can be more intrinsically defined as the dual of the Gauss-Manin local system $R^1f_* \mathbb{Z}$, i.e. $\Lambda_\rho = Hom( R^1 f_* \mathbb{Z} , \mathbb{Z} )$.}
\item[(2)]{The twisted Chern class is the obstruction to the existence of a section $s : B \to X$.}
\end{itemize}
\end{remark}

%%%%%%%%%%%%%%%%%%%%%%%%%%%%%%%%%%%%%%%
\subsection{Affine structure of the Hitchin system}
%%%%%%%%%%%%%%%%%%%%%%%%%%%%%%%%%%%%%%%

Fix an integer $r \ge 2$ and a degree $l$ line bundle $L$ with $l > 2g-2$ or $L = K$. Fix a basepoint $a_0 \in \mathcal{A}_{\rm reg}(r,L)$ with spectral curve $\pi : S \to \Sigma$. Let $\Lambda_S := H^1(S , \mathbb{Z}) \simeq H^1( Jac(S) , \mathbb{Z})$ and let $\Lambda_\Sigma := H^1( \Sigma , \mathbb{Z} ) \simeq H^1( Jac(\Sigma) , \mathbb{Z})$. We let $\langle \, , \, \rangle$ denote the intersection forms on $\Lambda_S$ and $\Lambda_\Sigma$. The pullback and norm maps $\pi^* : Jac(\Sigma) \to Jac(S)$, $Nm : Jac(S) \to Jac(\Sigma)$ induce pullback and pushforward maps in cohomology $\pi^* : \Lambda_\Sigma \to \Lambda_S$ and $\pi_* : \Lambda_S \to \Lambda_\Sigma$ with $\pi_* \pi^*(x) = rx$. Set $\Lambda_P := H^1( Prym(S,\Sigma) , \mathbb{Z} )$. 

\begin{proposition}\label{propprymcohom}
We have $\Lambda_P = ker( \pi_* : \Lambda_S \to \Lambda_\Sigma )$.
\end{proposition}
\begin{proof}
Applying the homotopy long exact sequence to the short exact sequence of abelian varieties 
\begin{equation*}\xymatrix{
1 \ar[r] & Prym(S,\Sigma) \ar[r] & Jac(S) \ar[r]^-{ Nm } & Jac(\Sigma) \ar[r] & 1,
}
\end{equation*}
shows that $H_1( Prym(S , \Sigma) , \mathbb{Z} ) = ker( \pi_* : H_1(S ,\mathbb{Z}) \to H_1( \Sigma , \mathbb{Z} ) )$. The proposition follows by applying Poincar\'e duality.
\end{proof}

We will identify $\Lambda_S$ with the local system on $\mathcal{A}_{\rm reg}(r,L)$ given by $\Lambda_S = R^1h_* \mathbb{Z}$. By restriction we will also regard $\Lambda_S$ as a local system on $\mathcal{A}^0_{\rm reg}(r,L)$. In a similar manner we identify $\Lambda_P$ with the local system on $\mathcal{A}_{\rm reg}^0(r,L)$ given by $\Lambda_P = R^1 \check{h}_* \mathbb{Z}$ and we view $\Lambda_\Sigma$ as a trivial local system. Note that the intersection forms on $\Lambda_S,\Lambda_\Sigma$ and the pullback and pushforward maps $\pi^*,\pi_*$ are all defined at the level of local systems. 

\begin{remark}
Since $\Lambda_P = ker( \pi_* )$, the dual local system is given by $\Lambda_P^* = \Lambda_S/ \pi^*(\Lambda_\Sigma)$. But this is precisely $R^1 \hat{h}_* \mathbb{Z}$. So the local systems $R^1 \check{h}_* \mathbb{Z}$ and $R^1 \hat{h}_* \mathbb{Z}$ are dual to each other.
\end{remark}

\begin{theorem}\label{thmaffine1}
For any integer $d$, we have:
\begin{enumerate}
\item{The Hitchin fibrations $\mathcal{M}_{\rm reg}(r,d,L) \to \mathcal{A}_{\rm reg}(r,L)$, $\mathcal{M}^0_{\rm reg}(r,d,L) \to \mathcal{A}^0_{\rm reg}(r,L)$ are affine torus bundles.}
\item{The monodromy representation $\rho : \pi_1(\mathcal{A}^0_{\rm reg}(r,L) , a_0) \to Aut( Jac(S) )$ of $\mathcal{M}^0_{\rm reg}(r,d,L)$ is the same for each value of $d \in \mathbb{Z}$.}
\item{Let $c \in H^2( \mathcal{A}^0_{\rm reg}(r,L) , \Lambda_S  )$ be the twisted Chern class of $\mathcal{M}^0_{\rm reg}(r,1-lr(r-1)/2,L)$. Then $\mathcal{M}^0_{\rm reg}(r,d,L)$ has twisted Chern class $\tilde{d}c$, where $\tilde{d} = d + lr(r-1)/2$.}
\item{$c$ is $r$-torsion, i.e. $rc = 0$.}
\end{enumerate}
\end{theorem}
\begin{proof}
We will give the proofs for $\mathcal{M}^0_{\rm reg}(r,d,L)$, the case of $\mathcal{M}_{\rm reg}(r,d,L)$ being essentially the same. Consider the union:
\begin{equation*}
\mathcal{M}^0_{\rm reg}(r,L) = \cup_{d \in \mathbb{Z} } \mathcal{M}^0_{\rm reg}(r,d,L).
\end{equation*}
Then $\mathcal{M}^0_{\rm reg}(r,L)$ is a bundle of groups with fibre $Pic(S) \simeq Jac(S) \times \mathbb{Z}$. Let $N$ be a line bundle on $S$ of degree $1$. This gives an explicit isomorphism $Jac(S) \times \mathbb{Z} \to Pic(S)$ sending $( A , m )$ to $A \otimes N^m$. As in Section \ref{secschf}, we set $\tilde{d} = d + lr(r-1)/2$. Then the component $\mathcal{M}^0_{\rm reg}(r,d,L)$ of $\mathcal{M}^0_{\rm reg}(r,L)$ corresponds to the component $Jac_{\tilde{d}}(S) = Jac(S) \times \{ \tilde{d} \}$ of the fibre. Let $n = 2g_S$ and let $T^n$ be a rank $n$ torus. Then $\mathcal{M}^0_{\rm reg}(r,L)$ is a bundle of groups with fibres isomorphic to $T^n \times \mathbb{Z}$. Let $Aut(T^n \times \mathbb{Z})$ be the automorphism group of $T^n \times \mathbb{Z}$ and let $p_2 : T^n \times \mathbb{Z} \to \mathbb{Z}$ be the projection to the second factor. We let $Aut^+(T^n \times \mathbb{Z})$ be those automorphisms $\phi : T^n \times \mathbb{Z} \to T^n \times \mathbb{Z}$ preserving $p_2$, i.e. $p_2 \circ \phi = p_2$. Then clearly the transtion functions for $\mathcal{M}^0_{\rm reg}(r,L)$ are valued in $Aut^+(T^n \times \mathbb{Z} )$ because we have a well-defined degree $d$.\\

Next we observe that there is an isomorphism ${\rm Aff}( T^n ) \simeq Aut^+( T^n \times \mathbb{Z} )$ given as follows: let $(g,s) \in Aut(T^n) \ltimes T = {\rm Aff}(T^n)$. Then we let $(g,s)$ act as an automorphism of $T^n \times \mathbb{Z}$ by: $(g,s)(  t , m ) = ( g(t) s^m , m )$, where $(t,m) \in T^n \times \mathbb{Z}$. Note that this is an automorphism of $T^n \times \mathbb{Z}$ preserving $p_2$ and that every such automorphism is of this form. Note also that $(g,s)$ acts on the component $T^n \times \{ 1 \}$ by the affine action $(g,s)( t , 1) = (g(t)s , 1 )$. This shows that each component $\mathcal{M}^0_{\rm reg}(r,d,L)$ is an affine torus bundle and that the mondromy is independent of $d$.\\

Let $c \in H^2( \mathcal{A}^0_{\rm reg}(r,L) , \Lambda_S  )$ be the twisted Chern class of $\mathcal{M}^0_{\rm reg}(r,1-lr(r-1)/2,L)$. Then $c$ is the twisted Chern class of the affine torus bundle associated to the component $T^n \times \{1 \} \subset T^n \times \mathbb{Z}$. Since $(g,s)$ acts on the component $T^n \times \{\tilde{d}\}$ by $(g,s)(t,\tilde{d}) = (g(t) s^{\tilde{d}} , \tilde{d})$, we see that the twisted Chern class of the affine torus bundle $\mathcal{M}^0_{\rm reg}(r,d,L)$ is $\tilde{d} c$.\\

Finally, let $A$ be a line bundle on $\Sigma$ of degree $1$. Tensoring by $A$ gives an isomorphism of affine torus bundle $\mathcal{M}^0_{\rm reg}(r,d,L) \simeq \mathcal{M}^0_{\rm reg}(r,d+r,L)$ for any $d \in \mathbb{Z}$. Comparing twisted Chern classes, we see that $r c = 0$.
\end{proof}

We have similar results for the $SL(r,\mathbb{C})$ and $PSL(r,\mathbb{C})$ moduli spaces:
\begin{theorem}\label{thmaffine2}
Let $D$ be a line bundle of degree $d$.
\begin{enumerate}
\item{The Hitchin fibrations $\check{\mathcal{M}}_{\rm reg}(r,D,L) \to \mathcal{A}^0_{\rm reg}(r,L)$, $\hat{\mathcal{M}}_{\rm reg}(r,d,L) \to \mathcal{A}^0_{\rm reg}(r,L)$ are affine torus bundles.}
\item{The monodromy representations $\check{\rho} : \pi_1(\mathcal{A}^0_{\rm reg}(r,L) , a_0) \to Aut( Prym(S,\Sigma) )$ and \linebreak $\hat{\rho} : \pi_1(\mathcal{A}^0_{\rm reg}(r,L) ,a_0) \to Aut( \hat{Prym}(S,\Sigma) )$ of $\check{\mathcal{M}}_{\rm reg}(r,D,L)$ and $\hat{\mathcal{M}}_{\rm reg}(r,d,L)$ are independent of $d \in \mathbb{Z}$.}
\item{The representations $\check{\rho},\hat{\rho}$ are duals.}
\item{$\Lambda_P \subset \Lambda_S$ is preserved by $\rho$ and the restriction of $\rho$ to $\Lambda_P$ is $\check{\rho}$.}
\item{$\Lambda_\Sigma \subset \Lambda_S$ is preserved by $\rho$ and $\hat{\rho}$ is the induced representation on $\Lambda_S/\Lambda_\Sigma \simeq \Lambda_P^*$.}
\item{Let $\check{c} \in H^2( \mathcal{A}^0_{\rm reg}(r,L) , \Lambda_P  )$ be the twisted Chern class of $\check{\mathcal{M}}_{\rm reg}(r,N,L)$, where $N$ has degree $1-lr(r-1)/2$. Then $\check{c}$ is independent of $N$ and for any line bundle $M$ of degree $d$, $\check{\mathcal{M}}_{\rm reg}(r,M,L)$ has twisted Chern class $\tilde{d}\check{c}$.}
\item{$\check{c}$ is $r$-torsion, i.e. $r\check{c} = 0$.}
\item{$\check{c}$ maps to $c$ under the natural map $H^2( \mathcal{A}^0_{\rm reg}(r,L) , \Lambda_P ) \to H^2( \mathcal{A}^0_{\rm reg}(r,L) , \Lambda_S )$.}
\item{Let $\hat{c} \in H^2( \mathcal{A}^0_{\rm reg}(r,L) , \Lambda_P^*  )$ be the twisted Chern class of $\hat{\mathcal{M}}_{\rm reg}(r,1-lr(r-1)/2,L)$. Then $\hat{\mathcal{M}}_{\rm reg}(r,d,L)$ has twisted Chern class $\tilde{d}\hat{c}$.}
\item{$\hat{c}$ is the image of $c$ under the map $H^2( \mathcal{A}^0_{\rm reg}(r,L) , \Lambda_S) \to H^2( \mathcal{A}^0_{\rm reg}(r,L) , \Lambda_P^* )$ induced by $\Lambda_S \to \Lambda_S/\Lambda_\Sigma \simeq \Lambda_P^*$.}
\end{enumerate}
\end{theorem}
\begin{proof}
Items 1), 2), 6), 7) and 9) are proved as in Theorem \ref{thmaffine1}. Item 3) follows since $Prym(S,\Sigma)$ and $\hat{Prym}(S,\Sigma)$ are dual abelian varieties. Items 4) and 8) follow from the natural inclusion $Prym(S,\Sigma) \subset Jac(S)$. Lastly, items 5) and 10) follow from the natural inclusion $\pi^* : Jac(\Sigma) \to Jac(S)$ and the identification $\hat{Prym}(S,\Sigma) \simeq Jac(S)/\pi^*(Jac(\Sigma))$.
\end{proof}

%%%%%%%%%%%%%%%%%%%%%%%%%%%%%%%%%%%%%%%%%%%%
%%%%%%%%%%%%%%%%%%%%%%%%%%%%%%%%%%%%%%%%%%%%
% SECTION
%%%%%%%%%%%%%%%%%%%%%%%%%%%%%%%%%%%%%%%%%%%%
%%%%%%%%%%%%%%%%%%%%%%%%%%%%%%%%%%%%%%%%%%%%
\section{Monodromy of twisted Hitchin systems}\label{secmths}

%%%%%%%%%%%%%%%%%%%%%%%%%%%%%%%%%%%%%%%%%%%%
\subsection{Fundamental group calculations}\label{secfgc}
%%%%%%%%%%%%%%%%%%%%%%%%%%%%%%%%%%%%%%%%%%%%

Henceforth we will consider exclusively the case of $L$-twisted rank $2$ Higgs bundles. To simplify notation we omit the $r$ and $L$ labels on the moduli spaces and Hitchin base. In particular, we have $\mathcal{A}^0 = H^0( \Sigma , L^2 )$. For a line bundle $N$, we let $H^0(\Sigma , N)^{\rm simp}$ be the space of sections of $N$ having only simple zeros. A point $a_2 \in H^0( \Sigma , L^2)$ defines a smooth spectral curve if and only if $a_2$ has only simple zeros, thus $\mathcal{A}^0_{\rm reg} = H^0( \Sigma , L^2)^{\rm simp}$. If $N$ is such that $H^0(\Sigma , N) \neq \{ 0\}$, we let $\mathbb{P}( H^0( \Sigma , N)^{\rm simp} )$ denote the image of $H^0(\Sigma , N)^{\rm simp}$ under the quotient map $H^0(\Sigma , N) \setminus \{0\} \to \mathbb{P}(H^0(\Sigma , N))$. Similarly, we write $\mathbb{P}( \mathcal{A}_{\rm reg}^0 )$ for $\mathbb{P}( H^0( \Sigma , L^2)^{\rm simp} )$.\\

Let $S^n \Sigma$ be the space of unordered $n$-tuples of points in $\Sigma$ and let $\alpha : S^n \Sigma \to Jac_n(\Sigma)$ be the Abel-Jacobi map sending a divisor $b_o$ to the corresponding line bundle $[b_o]$. We let $\widetilde{S}^n \Sigma \subseteq S^n \Sigma$ be those divisors consisting of distinct points and let $\tilde{\alpha} : \widetilde{S}^n \Sigma \to Jac_n(\Sigma)$ be the restriction of $\alpha$ to $\widetilde{S}^n \Sigma$. The fibre of $\tilde{\alpha}$ over $N \in Jac_n(\Sigma)$ is then $\mathbb{P}( H^0(\Sigma , N)^{\rm simp} )$. The fundamental group $\pi_1( \widetilde{S}^n \Sigma , b_o )$ is called the {\em $n$-th braid group of $\Sigma$} and will be denoted as $Br_n(\Sigma , b_o)$ \cite{bir}. The Abel-Jacobi map $\tilde{\alpha} : \widetilde{S}^n \Sigma \to Jac_n(\Sigma)$ induces a homomorphism $\tilde{\alpha}_* : Br_n(\Sigma , b_o) \to H_1( \Sigma , \mathbb{Z} ) \simeq \pi_1( Jac_k(\Sigma) , [b_o] )$. We then have:

\begin{proposition}[\cite{doli}]\label{propexact1}
Let $N$ be a line bundle of degree $n > 2g-2$, let $a_0 \in H^0( \Sigma , N)^{\rm simp}$ and let $b_o$ be the divisor of $a_0$. We have that $\tilde{\alpha} : \widetilde{S}^n \Sigma \to Jac_n(\Sigma)$ is a Serre fibration. In particular, we have an isomorphism
\begin{equation*}
\pi_1( \mathbb{P}(H^0(\Sigma , N)^{\rm simp}) , b_o ) \simeq ker( \tilde{\alpha}_* : Br_n(\Sigma , b_o ) \to H_1( \Sigma , \mathbb{Z} ) ). 
\end{equation*}
\end{proposition}

Write the divisor $b_o \in \widetilde{S}^n \Sigma$ as $b_o = b_1 + b_2 + \dots + b_n$, where the $b_i$ are distinct points in $\Sigma$. Suppose that $\gamma : [0,1] \to \Sigma$ is an embedded path joining $b_i = \gamma(0)$ to $b_j = \gamma(1)$, where $i \neq j$ and such that $\gamma$ meets no other point of $b_o$. When necessary, we shall write $\gamma$ with subscripts $\gamma_{ij}$, to indicate the endpoints. Let $D^2$ be the unit disc in $\mathbb{R}^2$. Choose an orientation preserving embedding $e : D^2 \to \Sigma$ such that $\gamma(t) = e( t-1/2 , 0)$ and such that $e(D^2)$ contains no other points of the divisor $b_o$. Next we define modified curves $\gamma^+,\gamma^-$, by setting $\gamma^+(t) := e( t-1/2 , \sin( \pi t) )$ and $\gamma^-(t) := e( 1/2 - t , -\sin(\pi t) )$. This defines a loop $p_\gamma(t)$ in $\widetilde{S}^n\Sigma$ based at $b_o$ by setting $p_\gamma(t) = b_1(t) + b_2(t) + \dots + b_{2l}(t)$, where $b_i(t) = \gamma^+(t)$, $b_j(t) = \gamma^-(t)$ and $b_k(t) = b_k$ for $k \neq i,j$, see Figure \ref{figswap}. The homotopy class $s_\gamma := [p_\gamma]$ of $p_\gamma(t)$ in $Br_n(\Sigma , b_o)$ clearly depends only on the choice of path $\gamma$. We call $s_\gamma$ the {\em swap associated to $\gamma$}. An element of $Br_n(\Sigma , b_o)$ of this form will be called a {\em swap of $b_i$ and $b_j$}, or simply a {\em swap}. Note that the swaps associated to $\gamma$ and $\gamma^{-1}$ are the same element of $Br_n(\Sigma , b_o)$.\\

\begin{figure}[h]
        \centering

\epsfig{file=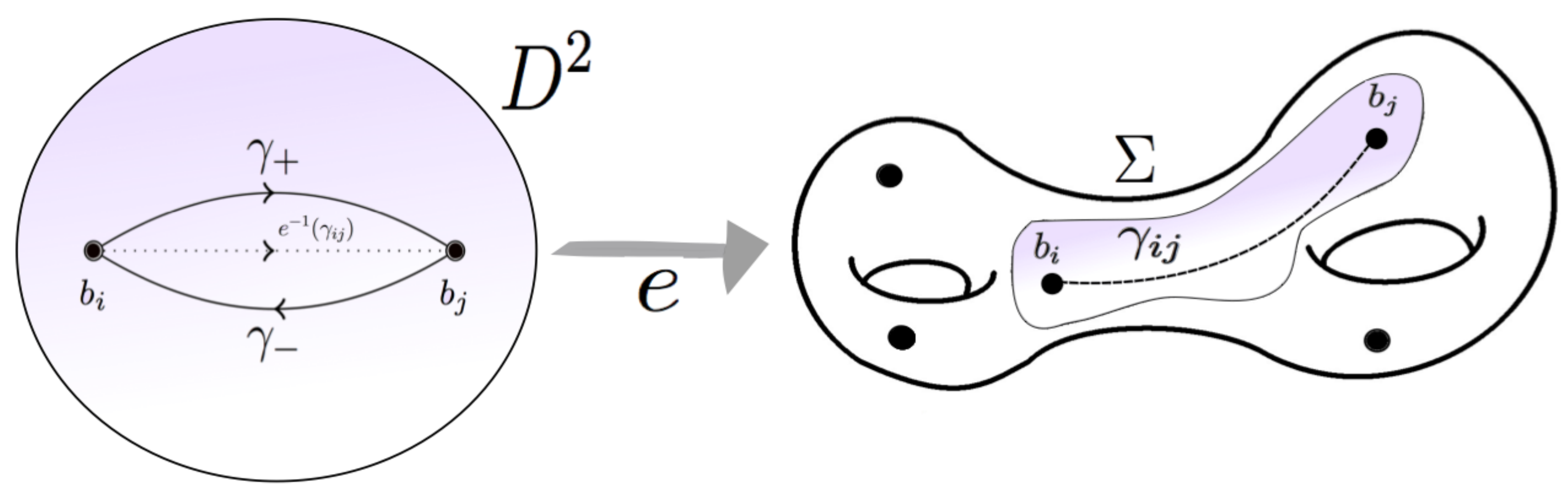, width=0.7\textwidth}

\caption{\small{A swap of $b_i$ and $b_j$ along the path $\gamma_{ij}$.}}
\label{figswap}
\end{figure}

\begin{theorem}\label{thmswap}
Suppose that $n \ge 4g-2$, or $n \ge 4g-4$ and $g > 2$. Then the kernel of $\tilde{\alpha}_* : Br_n(\Sigma , b_o) \to H_1(\Sigma , \mathbb{Z})$ is the subgroup of $Br( \Sigma , b_o)$ generated by swaps.
\end{theorem}
\begin{proof}
Clearly any swap lies in the kernel of $\tilde{\alpha}_*$, so we only need to show that the kernel of $\tilde{\alpha}_*$ may be generated by swaps. This follows easily from a result of Copeland \cite{cop0} and Walker \cite[Corollary 4.7]{walk}.
\end{proof}

Fix a point $p \in \Sigma$ and let $\mathcal{L}_n \to \Sigma \times Jac_n(\Sigma)$ be the Poincar\'e bundle of degree $k$ normalised with respect to $p$ \cite[Proposition 11.3.2]{bila}. This is the unique line bundle $\mathcal{L}_n$ on $\Sigma \times Jac_n(\Sigma)$ satisfying:
\begin{enumerate}
\item{$\mathcal{L}_n|_{\Sigma \times \{N\} } \simeq N$, for all $N \in Jac_n(\Sigma)$,}
\item{$\mathcal{L}_n|_{\{ p \} \times Jac_n(\Sigma)}$ is trivial.}
\end{enumerate}
For $n > 2g-2$, we obtain a vector bundle $q : V_n \to Jac_n(\Sigma)$ by letting the fibre of $V_n$ over $N \in Jac_n(\Sigma)$ be $H^0(\Sigma \times \{N\} , \mathcal{L}_n|_{\Sigma \times \{N\} } ) \simeq H^0( \Sigma , N)$. Taking divisors gives an isomorphism $\mathbb{P}(V_n) \simeq S^n \Sigma$, under which the Abel-Jacobi map is simply the projection $\mathbb{P}(V_n) \to Jac_n(\Sigma)$. This shows that $\alpha : S^n \Sigma \to Jac_n(\Sigma)$ is a locally trivial projective bundle, which moreover lifts to a vector bundle. Let $\widetilde{V}_n$ be the points in $V_n$ lying over $\widetilde{S}^n \Sigma$. This is a principal $\mathbb{C}^*$-bundle $\tilde{q} : \widetilde{V}_n \to \widetilde{S}^n \Sigma$. The fibre of $\widetilde{V}_n$ over $N \in Jac_n(\Sigma)$ is precisely $H^0(\Sigma , N )^{\rm simp}$.

\begin{proposition}\label{propexact2}
Let $a_0 \in H^0(\Sigma , N)^{\rm simp}$, where $N \in Jac_n(\Sigma)$ and $n > 2g-2$. Then we have an exact sequence:
\begin{equation*}\xymatrix{
\pi_1( H^0(\Sigma , N)^{\rm simp} , a_0 ) \ar[r]^-{i_*} & \pi_1( \widetilde{V}_n , a_0 ) \ar[rr]^-{\tilde{\alpha}_* \circ \tilde{q}_*} & & \pi_1( Jac_n(\Sigma) , N ) \ar[r] & 1,
}
\end{equation*}
where $i : H^0(\Sigma , N)^{\rm simp} \to \widetilde{V}_n$ is the inclusion map.
\end{proposition}
\begin{proof}
The commutative diagram
\begin{equation*}\xymatrix{
\mathbb{C}^* \ar@{=}[r] \ar[d] & \mathbb{C}^* \ar[d] & & \\
H^0(\Sigma , N)^{\rm simp} \ar[r]^-{i} \ar[d] & \widetilde{V}_n \ar[d]^-{\tilde{q}} \ar[rr]^-{\tilde{\alpha} \circ  \tilde{q}} & & Jac_n(\Sigma) \ar@{=}[d] \\
\mathbb{P}(H^0(\Sigma , N)^{\rm simp}) \ar[r]^-{i} & \widetilde{S}^n \Sigma \ar[rr]^-{\tilde{\alpha} } & & Jac_n(\Sigma)
}
\end{equation*}
gives rise to a commutative diagram of fundamental groups with exact columns:
\begin{equation*}\xymatrix{
\mathbb{Z} \ar@{=}[r] \ar[d] & \mathbb{Z} \ar[d] & & & \\
\pi_1(H^0(\Sigma , N)^{\rm simp},a_0) \ar[r]^-{i_*} \ar[d] & \pi_1(\widetilde{V}_n,a_0) \ar[d]^-{\tilde{q}} \ar[rr]^-{\widetilde{\alpha}_* \circ \tilde{q}_*} & & \pi_1(Jac_n(\Sigma),N) \ar@{=}[d] \ar[r] & 1 \\
\pi_1(\mathbb{P}(H^0(\Sigma , N)^{\rm simp}),b_o) \ar[d] \ar[r]^-{i_*} & \pi_1(\widetilde{S}^n \Sigma,b_o) \ar[d] \ar[rr]^-{\tilde{\alpha}_* } & & \pi_1(Jac_n(\Sigma),N) 
\ar[r] & 1 \\
1 & 1 & & &
}
\end{equation*}
By Proposition \ref{propexact1}, the third row of this diagram is exact. From this, exactness of the second row follows.
\end{proof}

By Proposition \ref{propexact1}, a swap gives an element in $\pi_1( \mathbb{P}(H^0(\Sigma , N)^{\rm simp}) , b_o )$. We now give a canonical procedure for lifting this to a loop in $\widetilde{V}_n$. Consider the swap of associated to a path $\gamma$ from $b_i$ to $b_j$ as in Figure \ref{figswap}. Let $e : D^2 \to \Sigma$ be an oriented embedding such that $\gamma(t) = e( t-1/2 , 0)$ and such that $e(D^2)$ contains no other points of the divisor $b_o$. Let $S^2( D^2)$ be the symmetric product of $D^2$. There is an induced map $i : S^2(D^2) \to S^n \Sigma$ sending a pair $u,v \in D^2$ to the divisor $e(u) + e(v) + \sum_{k \neq i,j} b_k$. In particular, $p_\gamma(t) = i( (t-1/2 , \sin(\pi t) ) , (1/2 - t , -\sin(\pi t) ) )$. Let $V'_n$ be $V_n$ with the zero section removed. The projection $q : V'_n \to S^n \Sigma$ is a principal $\mathbb{C}^*$-bundle. The pullback $i^*( V'_n)$ is then a principal $\mathbb{C}^*$-bundle over the contractible space $S^2( D^2)$ and thus admits a section, i.e. a map $s : S^2(D^2) \to V'_n$ such that $q \circ s = i$. We can also choose $s$ such that $s( (-1/2 , 0) , (1/2 , 0 ) ) = a_0$. Now let $\tilde{p}_\gamma (t) := s( (t-1/2 , \sin(\pi t) ) , (1/2 - t , -\sin(\pi t) ) )$. This is a lift of $p_\gamma$ to a loop in $\widetilde{V}_n$ based at $a_0$. It is clear that the homotopy class of the lift $[ \tilde{p}_\gamma ] \in \pi_1( \widetilde{V}_n , a_0 )$ is independent of the embedding $e$ and section $s$. From Proposition \ref{propexact2}, the class $\tilde{s}_\gamma := [ \tilde{p}_\gamma ]$ lies in the image of $i_* : \pi_1( H^0(\Sigma , N)^{\rm simp} , a_0 ) \to \pi_1( \widetilde{V}_n , a_0 )$. While this does not uniquely determine a lift of $[ p_\gamma]$ to a class in $\pi_1( H^0(\Sigma , N)^{\rm simp} , a_0 )$, it is sufficient for monodromy computations, as we will see that the monodromy representation of the Hitchin system factors through $i_*$.\\

We now consider the case where $n = 2l$ and $N = L^2$, so that $a_0 \in H^0( \Sigma , L^2)^{\rm simp} = \mathcal{A}_{\rm reg}^0$. The projection $\mathcal{A}_{\rm reg}^0 \to \mathbb{P}(\mathcal{A}_{\rm reg}^0)$ is a principal $\mathbb{C}^*$-bundle, so gives an exact sequence:
\begin{equation}\label{equcstarexact}
\pi_1( \mathbb{C}^* , a_0 ) \to \pi_1( \mathcal{A}_{\rm reg}^0 , a_0 ) \to \pi_1( \mathbb{P}(\mathcal{A}_{\rm reg}^0 ) , b_o ) \to 1.
\end{equation}
We then have:
\begin{proposition}\label{propgenerate}
The group $\pi_1( \mathcal{A}_{\rm reg}^0 , a_0)$ is generated by the loop given by the $\mathbb{C}^*$-action on $\mathcal{A}_{\rm reg}^0$ together with lifts of swaps.
\end{proposition}
\begin{proof}
By Proposition \ref{propexact2} and the exact sequence (\ref{equcstarexact}), it is enough to show that $\pi_1( \mathbb{P}(\mathcal{A}_{\rm reg}^0 ) , b_o )$ is generated by swaps. Suppose that $deg(L) > deg(K)$ or that $L = K$ and $g > 3$. Then we have $2l \ge 4g-2$ or $2l = 4g-4$ and $g > 3$ and the result follows by Theorem \ref{thmswap}.\\

It remains only to show that $\pi_1( \mathbb{P}(\mathcal{A}_{\rm reg}^0 ) , b_o )$ is generated by swaps when $L = K$ and $g = 2$. In this case, $\Sigma$ is a hyperelliptic curve, so there is a map $f : \Sigma \to \mathbb{P}^1$ such that $f$ is a branched double cover with $6$ branch points. We may identify $\mathbb{P}^1$ with $\mathbb{C} \cup \{ \infty \}$ and take $\infty$ to be one of the branch points, so there are $5$ other branch points $x_1, \dots ,x_5 \in \mathbb{C}$. Let $\iota : \Sigma \to \Sigma$ be the hyperelliptic involution. Then as $g=2$, all elements of $H^0( \Sigma , K^2 )$ are fixed by $\iota$. Thus any $a \in H^0(\Sigma , K^2)$ has zero set given as the pre-image under $f$ of two distinct points $u,v \in \mathbb{C} \setminus \{ x_1 , \dots , x_5 \}$. Let $\mathbb{C}_5 = \mathbb{C} \setminus \{ x_1 , \dots , x_5 \}$ denote the plane with the $5$ points $x_1,\dots , x_5$ removed. Then $\mathbb{P}(\mathcal{A}_{\rm reg}^0)$ is naturally identified with $\widetilde{S}^2 \mathbb{C}_5$. Thus $\pi_1( \mathbb{P}(\mathcal{A}_{\rm reg}^0 ) , b_o ) \simeq Br_2( \mathbb{C}_5 )$ is the $2$nd braid group of the plane with $5$ points removed (see also \cite[Theorem 5.1]{cop}). It remains to show that $Br_2( \mathbb{C}_5 )$ may be generated by elements corresponding to swaps.\\

Let $u,v \in \mathbb{C}_5$ be the two points in $\mathbb{C}$ corresponding to the zeros of $a_0$. We have that $Br_2( \mathbb{C}_5)$ is generated by $\sigma_1 , l_1 , \dots , l_5$ where $\sigma_1$ is the braid given by a swap of $u,v$ within an embedded disc containing $u,v$ but not the points $x_1, \dots , x_5$ and $l_i$ is the braid in which $u$ moves around a loop encircling $x_i$ while $v$ is held fixed. Clearly $\sigma_1$ corresponds to a product of two swaps in $\pi_1( \mathbb{P}(\mathcal{A}_{\rm reg}^0 ) , b_o )$ (the swaps of the pre-images of $u$ and $v$). Consider the braid $l_i$. Let $\mu_i$ be an embedded loop based at $u$ going around $x_i$ but not around $x_j$ for $j \neq i$. Then $l_i$ is the braid which moves $u$ along $\mu_i$ while $v$ is fixed. Now observe that since $x_i$ is a branch point of $\Sigma \to \mathbb{P}^1$, we have that the pre-image $f^{-1}(\mu_i)$ is an embedded path in $\Sigma$ joining the two points in $f^{-1}(u)$ and one easily finds that $l_i$ corresponds to a swap of these points along $f^{-1}(\mu_i)$.
\end{proof}

%%%%%%%%%%%%%%%%%%%%%%%%%%%%%%%%%%
\subsection{The monodromy representation}\label{sectmr}
%%%%%%%%%%%%%%%%%%%%%%%%%%%%%%%%%%

\begin{definition}\label{deftau}
We let $\tau : [0,1] \to \mathcal{A}_{\rm reg}^0$ be the loop in $\mathcal{A}_{\rm reg}^0$ generated by the $\mathbb{C}^*$-action, namely $\tau(t) = e^{2\pi i t} a_0$.
\end{definition}

\begin{proposition}\label{proploopmono}
The monodromy action of $\rho( \tau ) \in Aut( Jac(S) )$ is given by the pullback $\sigma^* : Jac(S) \to Jac(S)$, where $\sigma$ is the sheet swapping involution of the double cover $\pi : S \to \Sigma$.
\end{proposition}
\begin{proof}
Let $S_t$ be the spectral curve associated to $\tau(t) = e^{2\pi i t} a_0$, given by $S_t = \{ \lambda \in L \; | \; \lambda^2 + e^{2 \pi i t}a_0 = 0 \}$. Now if $\lambda_0 \in L$ is such that $\lambda_0^2 + a_0 = 0$, then setting $\lambda_t = e^{\pi i t} \lambda_0$, we have $\lambda_t^2 + e^{2\pi i t}a_0 = 0$. When $t = 1$, we get $\lambda_1 = -\lambda_0$ and so the monodromy around $\tau$ acts on $S = S_0$ by $\lambda \mapsto -\lambda$. This is exactly the sheet swapping involution $\sigma$.
\end{proof}

%%%%%%%%%%%%%%%%%%%%%%%%%%%%%%

It remains to determine the monodromy for lifts of swaps. For this it is convenient to map $\mathcal{A}_{\rm reg}^0$ into a larger family of branched double covers of $\Sigma$.\\

Let $sq : Jac_l(\Sigma) \to Jac_{2l}(\Sigma)$ be the squaring map $sq(L) = L^2$. We define spaces $Y_l,Z_l$ by the following pullback diagrams:
\begin{equation*}\xymatrix{
Y_l \ar[r]^-{\alpha'} \ar[d]^-{p} & Jac_l(\Sigma) \ar[d]^-{sq}  & & Z_l \ar[r]^-{q'} \ar[d]^-{p'} & Y_l \ar[d]^-{p} \\
S^{2l}\Sigma \ar[r]^-{\alpha} & Jac_{2l}(\Sigma) & & V'_{2l} \ar[r]^-{q} & S^{2l}\Sigma
}
\end{equation*}
where $V'_{2l}$ is $V_{2l}$ with the zero section removed. Let $\widetilde{Z}_{l} = (p')^{-1}( \widetilde{V}_{2l} )$ and $\widetilde{Y}_l = p^{-1}( \widetilde{S}^{2l} \Sigma )$, giving a similar pair of commutative squares:
\begin{equation*}\xymatrix{
\widetilde{Z}_{l} \ar[r]^-{\tilde{q}'} \ar[d]^-{\tilde{p}'} & \widetilde{Y}_{l} \ar[r]^-{\tilde{\alpha}'} \ar[d]^-{\tilde{p}} & Jac_l(\Sigma) \ar[d]^-{sq} \\
\widetilde{V}_{2l} \ar[r]^-{\tilde{q}} & \widetilde{S}^{2l} \Sigma \ar[r]^-{\tilde{\alpha}} & Jac_{2l}(\Sigma)
}
\end{equation*}   
A point $z \in \widetilde{Z}_l$ is given by a degree $l$ line bundle $M$ and an element $s \in H^0(\Sigma , M^2 )^{\rm simp}$. We therefore have a natural inclusion $\iota : \mathcal{A}_{\rm reg}^0 \hookrightarrow \widetilde{Z}_l$. To any $z \in \widetilde{Z}_l$ we associate a branched double cover $S_z := \{ y \in M \; | \; y^2 + s = 0 \}$. Letting $z$ vary we obtain a family $\widetilde{\mathbb{S}}_l$ of branched double covers with a commutative diagram
\begin{equation*}\xymatrix{
& \widetilde{\mathbb{S}}_l \ar[dl]_-{v} \ar[d]^-{w} \\
\Sigma \times \widetilde{Z}_l \ar[r]^-{p_2} & \widetilde{Z}_l
}
\end{equation*}
Such that for each $z \in \widetilde{Z}_l$, the fibre of $w$ over $z$ is the branched double cover $S_z$ and $v|_{S_z}$ is the covering map $S_z \to \Sigma$. Using the natural identification $Aut( H^1( S , \mathbb{Z} ) ) = Aut( Jac(S) )$, we obtain a representation $\rho_{\widetilde{Z}_l} : \pi_1( \widetilde{Z}_l , z_0) \to Aut( Jac(S) )$. Noting that the family of spectral curves over $\mathcal{A}_{\rm reg}^0$ is the pullback of $w : \widetilde{\mathbb{S}}_l \to \widetilde{Z}_l$ under $\iota$, we obtain:
\begin{proposition}\label{propmonofact}
We have an equality $\rho = \rho_{\widetilde{Z}_l} \circ \iota_*$.
\end{proposition}

Since $\tilde{p}' : \widetilde{Z}_{l} \to \widetilde{V}_{2l}$ is a covering space, we get an injection $\tilde{p}'_* : \pi_1( \widetilde{Z}_{l} , z_0 ) \to \pi_1(\widetilde{V}_{2l} , a_0)$. Combined with Proposition \ref{propmonofact}, we have a commutative diagram:
\begin{equation*}\xymatrix{
& Aut( Jac(S) ) \\
\pi_1( \mathcal{A}_{\rm reg}^0 , a_0) \ar[ur]^-{\rho} \ar[r]^-{\iota_*} \ar[dr]^-{i_*} & \pi_1( \widetilde{Z}_{l} , z_0 ) \ar[u]^-{\rho_{\widetilde{Z}_l} }
\ar[d]^-{\tilde{p}'_*} \\
& \pi_1( \widetilde{V}_{2l} , a_0 )
}
\end{equation*}

Recall from Section \ref{secfgc} that to a path $\gamma$ joining $b_i$ to $b_j$ we obtain a swap $s_\gamma \in Br_{2l}(\Sigma , b_o)$ and that we have a canonical lift $\tilde{s}_\gamma \in \pi_1( \widetilde{V}_{2l} , a_0)$ lying in the image of $\tilde{p}'_*$. Injectivity of $\tilde{p}'_* : \pi_1( \widetilde{Z}_{l} , z_0 ) \to \pi_1(\widetilde{V}_{2l} , a_0)$, implies that there is a well-defined monodromy action $\rho_{\widetilde{Z}_l}( \tilde{s}_\gamma ) \in Aut( Jac(S) )$. Moreover, if $s'_\gamma \in \pi_1( \mathcal{A}_{\rm reg}^0 , a_0)$ is any lift of $\tilde{s}_\gamma$ to a class in $\pi_1( \mathcal{A}_{\rm reg}^0 , a_0)$, then $\rho( s'_\gamma ) = \rho_{\widetilde{Z}_l}( \tilde{s}_\gamma )$. Therefore it remains only determine the element $\rho_{\widetilde{Z}_l}( \tilde{s}_\gamma ) \in Aut( Jac(S) )$ associated to $\gamma$.

\begin{theorem}\label{thmswapmono}
Let $l_\gamma$ be the embedded loop in $S$ given by the preimage $\pi^{-1}(\gamma)$ of $\gamma$. The monodromy action $\rho_{\widetilde{Z}_l}( \tilde{s}_\gamma ) \in Aut( H^1(S,\mathbb{Z}) )$, is the automorphism of $H^1(S,\mathbb{Z})$ induced by a Dehn twist of $S$ around $l_\gamma$.
\end{theorem}

\begin{notation}\label{not1}
We use $l_\gamma$ to denote the loop in $S$ associated to $\gamma$. Note that the homology class $[l_\gamma] \in H_1( S ,\mathbb{Z})$ satisfies $\pi_* [l_\gamma] = 0$. Let $c_\gamma \in H^1( S , \mathbb{Z} ) = \Lambda_S$ denote the Poincar\'e dual class. Then $c_\gamma \in \Lambda_P$. A Dehn twist of $S$ around $l_\gamma$ acts on $H^1(S , \mathbb{Z})$ as a Picard-Lefschetz transformation. Thus the monodromy action of the loop associated to $\gamma$ is:
\begin{equation}\label{equpl1}
\gamma \cdot x := \rho_{\widetilde{Z}_l}( \tilde{s}_\gamma )x = x + \langle c_\gamma , x \rangle c_\gamma.
\end{equation}
Such a transformation is also referred to as a {\em symplectic transvection}. Note that the isotopy class of a Dehn twist around $\gamma$ depends only on the isotopy class of the embedded loop $l_\gamma$, and does not depend on a choice of orientation of $l_\gamma$. Recall from Definition \ref{deftau}, that $\tau$ is the loop in $\mathcal{A}_{\rm reg}^0$ generated by the $\mathbb{C}^*$-action. We will write $\tau \cdot x$ for the monodromy action of $\rho(\tau)$ on $x$. Proposition \ref{proploopmono} and Equation (\ref{equnorm1}) give:
\begin{equation}\label{equloopmono}
\tau \cdot x = \sigma^*(x) = -x + \pi^*( \pi_*(x)).
\end{equation}
Note that since $\sigma( l_\gamma) = l_\gamma$, the action of $\tau$ commutes with the action of $\gamma$. This can be also checked directly from (\ref{equpl1})-(\ref{equloopmono}) using $\pi_* ( c_\gamma) = 0$.
\end{notation}
\begin{proof}[Proof of Theorem \ref{thmswapmono}:]
Let $\gamma$ be an embedded path in $\Sigma$ joining branch points $b_i,b_j$ and avoiding all other branch points. As in Figure \ref{figswap}, choose an embedding $e : D^2 \to \Sigma$ of the unit disc $D^2$ into $\Sigma$ containing all branch points $b_1, \dots , b_{2l}$ as well as the path $\gamma$. The swap associated to $\gamma$ defines a loop $b(t)$ based at $b_o$ in the space of degree $2l$ divisors with simple zeros contained in $e(D^2)$. Let $\pi_t : S_t \to \Sigma$ for $t \in [0,1]$ be the resulting family of branched double covers of $\Sigma$. Clearly no change is made to the double cover outside of the image $e(D^2)$, so the problem reduces to understanding the family $S_t |_{\pi^{-1}_t( e(D^2) )}$ of branched covers of the disc $D^2$. It is well-known from Picard-Lefschetz theory \cite{agzv} (see also \cite{chm}) that the monodromy is described by a Dehn twist of $S|_{\pi^{-1}(e(D^2))}$ around the cycle $l_\gamma$. This acts trivially on the boundary of $S|_{\pi^{-1}(e(D^2))}$ and so extends to give a Dehn twist of $S$ around $l_\gamma$.
\end{proof}

%%%%%%%%%%%%%%%%%%%%%%%%%%%%%%%%%%%%%%%%%%%%
\subsection{Twisted Chern class}\label{sectcc}
%%%%%%%%%%%%%%%%%%%%%%%%%%%%%%%%%%%%%%%%%%%%

Let $\Lambda_S[2] = \Lambda_S \otimes_{\mathbb{Z}} \mathbb{Z}_2$ and similarly define $\Lambda_\Sigma[2], \Lambda_P[2]$. The local systems $\Lambda_S[2] , \Lambda_P[2]$ can be thought of as bundles of groups over $\mathcal{A}_{\rm reg}^0$, with fibres the points of order $2$ in $Jac(S)$ and $Prym(S,\Sigma)$ respectively. More generally, for $k \in \mathbb{Z}$, let $A$ be a fixed degree $k$ line bundle on $\Sigma$ and define
\begin{equation*}
\begin{aligned}
\Lambda_S^k[2] &= \{ M \in Jac_k(S) \; | \; M^2 = \pi^*(A) \; \}, \\
\Lambda_P^k[2] &= \{ M \in Jac_k(S) \; | \; \sigma(M) = M, \; M^2 = \pi^*(A) \; \}.
\end{aligned}
\end{equation*}
Then $\Lambda_S^k[2]$ may be thought of as a bundle of $\Lambda_S[2]$-torsors over $\mathcal{A}_{\rm reg}^0$ and similarly $\Lambda_P^k[2]$ as a bundle of $\Lambda_P[2]$-torsors. Note also that $\Lambda_S^k[2],\Lambda_P^k[2]$ are up to isomorphism independent of the choice of degree $k$ line bundle $A$. The $\Lambda_P[2]$-torsor $\Lambda_P^1[2]$ is classified by a class $\check{\beta} \in H^1( \mathcal{A}_{\rm reg}^0 , \Lambda_P[2] )$. Similarly $\Lambda_S^1[2]$ is classified by a class $\beta \in H^1( \mathcal{A}_{\rm reg}^0 , \Lambda_S[2] )$. The inclusion $\Lambda_P^1[2] \to \Lambda_S^1[2]$ shows that $\check{\beta}$ maps to $\beta$ under the natural map $H^1( \mathcal{A}_{\rm reg}^0 , \Lambda_P[2] ) \to H^1( \mathcal{A}_{\rm reg}^0 , \Lambda_S[2] )$.

\begin{proposition}
Let $A$ be a degree $1$ line bundle and let $\check{c} \in H^2( \mathcal{A}^0_{\rm reg} , \Lambda_P )$ be the twisted Chern class of $\check{\mathcal{M}}_{\rm reg}(A L^*)$, as in Theorem \ref{thmaffine2}. Then $\check{c}$ is the image of $\check{\beta}$ under the coboundary map $\delta : H^1( \mathcal{A}^0_{\rm reg} , \Lambda_P[2] ) \to H^2( \mathcal{A}^0_{\rm reg} , \Lambda_P )$ associated to $\Lambda_P \buildrel 2 \over \longrightarrow \Lambda_P \longrightarrow \Lambda_P[2]$.
\end{proposition}
\begin{proof}
This follows by simply observing that there is a natural inclusion $\Lambda_P^1[S] \subset \check{\mathcal{M}}_{\rm reg}(AL^*)$ compatible with the inclusion $\Lambda_P[2] \subset \check{\mathcal{M}}_{\rm reg}(L^*)$.
\end{proof}

Next, we proceed to give a description of the class $\check{\beta}$. Let $a_0 \in \mathcal{A}_{\rm reg}^0$ be the basepoint with spectral curve $\pi : S \to \Sigma$, and $b_o = b_1 + b_2 + \dots + b_{2l}$ the divisor of $a_0$. Let $u_k \in S$ be the ramification point lying over $b_k \in \Sigma$. As shown in Theorem \ref{thmaffine2}, the twisted Chern class $\check{c}$ of $\check{\mathcal{M}}(AL^*)$ is independent of the choice of $A \in Jac_1(\Sigma)$. A convenient choice will be to take $A = \mathcal{O}(p)$, where $p$ is a branch point. Without loss of generality, we may take $p = b_1$. Then $\Theta := \mathcal{O}(u_1) \in Jac_1(S)$ satisfies $Nm( \Theta ) = \mathcal{O}(b_1) = A$ and $\Theta^2 = \pi^*( A )$, hence $\Theta \in \Lambda_P^1[2]$.\\

A representative for $\check{\beta}$ is a map $\check{\beta} : \pi_1(\mathcal{A}_{\rm reg}^0 , a_0 ) \to \Lambda_P[2]$ satisfying the cocycle condition $\check{\beta}( gh ) = \check{\beta}(g) + g \cdot \check{\beta}(h)$. Our choice of origin $\Theta$ gives us a particular representative by setting $\check{\beta}(g) = g \cdot \Theta - \Theta$. Clearly $\check{\beta}(g)$ satisfies the cocycle condition and is valued in $\Lambda_P[2]$ because the monodromy action preserves $\pi^*$ and $Nm$. Next we determine the value of $\check{\beta}$ on the generators of $\pi_1(\mathcal{A}_{\rm reg}^0 , a_0)$ given in Proposition \ref{propgenerate}:
\begin{theorem}\label{thmbeta}
Let $\tau$ be the loop in $\mathcal{A}_{\rm reg}^0$ given as in Definition \ref{deftau}, then $\check{\beta}(\tau) = 0$. Let $\tilde{s}_\gamma \in \pi_1( \mathcal{A}_{\rm reg}^0 , a_0)$ be a lift of a swap of $b_i,b_j$ along the path $\gamma$. Then:
\begin{equation*}
\check{\beta}( \tilde{s}_\gamma ) = \begin{cases} 0 & \text{if } 1 \notin \{i,j\}, \\  
c_\gamma & \text{if } 1 \in \{i,j\}, \end{cases}
\end{equation*}
where $c_\gamma$ is defined as in Notation \ref{not1}.
\end{theorem}
\begin{proof}
Consider first the loop $\tau \in \pi_1(\mathcal{A}_{\rm reg}^0 , a_0)$. By Proposition $\ref{proploopmono}$ the action of $\tau$ on $Jac(S)$ was the map induced by the involution $\sigma : S \to S$. More generally, this applies with $Jac_d(S)$ in place of $Jac(S)$ and so we have:
\begin{equation*}
\check{\beta}(\tau) = \sigma^*(\Theta) - \Theta = \sigma^*( \mathcal{O}(u_1) ) - \mathcal{O}(u_1) = \mathcal{O}( \sigma(u_1) - u_1 ) = 0,
\end{equation*}
since $u_1$, being a ramification point, satisfies $\sigma(u_1) = u_1$.\\

Now consider the lift $\tilde{s}_\gamma \in \pi_1( \mathcal{A}_{\rm reg}^0 , a_0)$ of a swap along the path $\gamma$. We will denote $\check{\beta}( \tilde{s}_\gamma )$ more simply as $\check{\beta}(\gamma)$. We will approach the computation of $\check{\beta}(\gamma)$ by interpreting it in terms of monodromy of the covering space $\Lambda_P^1[2] \to \mathcal{A}_{\rm reg}^0$. Consider $\tilde{s}_\gamma$ as a loop in $\mathcal{A}_{\rm reg}^0$ based at $a_0$. Let $q : [0,1] \to \Lambda_P^1[2]$ be the unique lift of $\tilde{s}_\gamma$ to a path in $\Lambda_P^1[2]$ with $q(0) = \Theta$. Then $q(1) = \check{\beta}(\gamma) q(0)$. Suppose that $\gamma$ is a path from $b_i$ to $b_j$. There are three cases to consider: (i) $1 \notin \{ i,j \}$, (ii) $i=1$ and (iii) $j=1$.\\

Case (i): Here $b_1$ is a zero of $\tilde{s}_\gamma(t)$ for all $t$. Let $u_1(t)$ be the corresponding ramification point. Then $q(t) = \mathcal{O}( u_1(t) )$ and $q(1) = q(0)$, since $u_1(1) = u_1(0)$. So $\check{\beta}(\gamma) = 0$ in this case.\\

Case (ii): In this case $\gamma$ starts at $b_1 = b_i$. As $t$ varies the zeros of $\tilde{s}_\gamma$ move continuously and in particular, $b_1$ moves along $\gamma$. Let $u_1(t)$ be the corresponding ramification point. Then since $u_1(t)$ is the ramification point over $\gamma(t)$, we have $u_1(0) = u_1$, $u_1(1) = u_j$. Let $\Gamma : [0,1] \to Jac(\Sigma)$ be the unique path in $Jac(\Sigma)$ satisfying $\Gamma(0) = \mathcal{O}$ and $\Gamma(t)^2 = \mathcal{O}( \gamma(t) - x_1 )$. Then $q(t) = \mathcal{O}(u_1(t) ) \otimes \pi^*( \Gamma(t)^* )$. Therefore
\begin{equation*}
\check{\beta}(\gamma) = q(1) \otimes q(0)^* = \mathcal{O}( u_j - u_1) \otimes \pi^*( \Gamma(1)^* ).
\end{equation*}
In order to determine $\check{\beta}$ as an element of $\Lambda_P[2] \simeq H_1( S , \mathbb{Z}_2)$, we will evaluate $\check{\beta}$ on an arbitrary element $\omega \in H^1( S ,\mathbb{Z}_2)$. We can view $\omega$ as the mod $2$ reduction of a class in $H^1(S , \mathbb{Z})$, a closed $1$-form on $S$ with integral periods. We view $\check{\beta}(\gamma)$ as an element of $\frac{1}{2} H_1(S , \mathbb{Z} ) / H_1(S,\mathbb{Z})$ so that the pairing $\langle \check{\beta}(\gamma) , \omega \rangle$ is an element of $\mathbb{Z}_2 \simeq \frac{1}{2}\mathbb{Z} / \mathbb{Z} \subset \mathbb{R}/\mathbb{Z}$. Let $\gamma_1,\gamma_2$ be the two paths in $S$ from $u_1$ to $u_j$ lying over $\gamma$. Then:
\begin{equation*}
\begin{aligned}
\langle \check{\beta}(\gamma) , \omega \rangle &= \int_{\gamma_1} \omega - \langle \Gamma(1) , \pi_* \omega \rangle \; ({\rm mod} \; \mathbb{Z}) \\
&= \int_{\gamma_1} \omega - \frac{1}{2} \int_{\gamma} \pi_* \omega \; ({\rm mod} \; \mathbb{Z}) \\
&= \int_{\gamma_1} \omega - \frac{1}{2} \int_{\gamma_1} \omega - \frac{1}{2} \int_{\gamma_2} \omega \; ({\rm mod} \; \mathbb{Z}) \\
&= \frac{1}{2} \left( \int_{\gamma_1} \omega - \int_{\gamma_2} \omega \right) \; ({\rm mod} \; \mathbb{Z}) \\
&= \frac{1}{2} \int_{l_\gamma} \omega \; ({\rm mod} \; \mathbb{Z}).
\end{aligned}
\end{equation*}
In other words, we have shown that $\check{\beta}(\gamma) = c_\gamma$.\\

Case (iii): This case is similar to the previous case, except that we should replace $\gamma(t)$ with $\gamma(1-t)$. We again obtain $\check{\beta}(\gamma) = c_\gamma$, which is to be expected as we have already established that the monodromy does not depend on the orientation of $\gamma$.
\end{proof}

%%%%%%%%%%%%%%%%%%%%%%%%%%%%%%%%%%%%%%%%%
\subsection{Monodromy action on $\Lambda_P[2]$ and $\Lambda_S[2]$}\label{secmonoact}
%%%%%%%%%%%%%%%%%%%%%%%%%%%%%%%%%%%%%%%%%

Let $B = \{ b_1 , b_2 , \dots , b_{2l} \}$ be the set of branch points and $\mathbb{Z}_2 B$ the $\mathbb{Z}_2$-vector space with basis $b_1, \dots , b_{2l}$. Let $s : \mathbb{Z}_2 B \to \mathbb{Z}_2$ be the linear map with $s(b_i) = 1$ for all $i$. Let $(\mathbb{Z}_2 B)^{\rm ev}$ denote the kernel of $s$. By abuse of notation we will let $b_o$ denote the element $b_o = b_1 + b_2 + \dots + b_{2l} \in (\mathbb{Z}_2 B)^{\rm ev}$.\\

Recall that $\Lambda_P[2] = \Lambda_P \otimes_{\mathbb{Z}} \mathbb{Z}_2$, which can be naturally identified with the points of order $2$ in $Prym(S,\Sigma)$. Thus an element of $\Lambda_P[2]$ is a line bundle $M \in Jac(S)$ such that $M^2 = \mathcal{O}$ and $\sigma^*(M) \simeq M$. Alternatively, we may think of $M$ as a $\mathbb{Z}_2$-local system together with an isomorphism $\tilde{\sigma} : M \to M$ of $\mathbb{Z}_2$-local systems, which covers $\sigma$. Note that for such an $M$, the isomorphism $\tilde{\sigma}$ is only unique up to an overall sign change $\tilde{\sigma} \mapsto -\tilde{\sigma}$. If $u_i$ is the ramification point over $b_i$ then $\tilde{\sigma}$ sends $M_{u_i}$ to itself, acting either as $1$ or $-1$. Let $\epsilon_i \in \mathbb{Z}_2$ be defined such that $\tilde{\sigma}$ acts on $M_{u_i}$ by $(-1)^{\epsilon_i}$. The pair $(M , \tilde{\sigma})$ determines an element $\epsilon(M,\tilde{\sigma}) = \epsilon_1 b_1 + \dots +\epsilon_{2l} b_{2l} \in \mathbb{Z}_2 B$. In fact, $\epsilon(M,\tilde{\sigma})$ is valued in $(\mathbb{Z}_2 B)^{\rm ev}$. To see this we note that the restriction of $M$ to $S \setminus \{ u_1 , \dots , u_{2l} \}$ descends to a local system $M'$ on $\Sigma \setminus \{ b_1 , \dots , b_{2l} \}$. Let $\partial_i$ be the class in $H_1( \Sigma \setminus \{b_1 , \dots , b_{2l} \} , \mathbb{Z}_2 )$ given by a cycle around $b_i$. The holonomy of $M'$ around $\partial_i$ is $\epsilon_i$, but $\partial_1 + \dots + \partial_{2l} = 0$ and hence $\epsilon_1 + \dots + \epsilon_{2l} = 0$. It is clear that $\epsilon(M , -\tilde{\sigma}) = \epsilon(M , \tilde{\sigma}) + b_o$ and hence the image of $\epsilon(M , \tilde{\sigma})$ in $(\mathbb{Z}_2 B)^{\rm ev}/ (b_o)$ depends only on $M$ and not on the choice of isomorphism $\tilde{\sigma}$. This gives a well defined map $\epsilon : \Lambda_P[2] \to (\mathbb{Z}_2 B)^{\rm ev}/ (b_o)$.
\begin{proposition}\label{propprymexact1}
We have a short exact sequence:
\begin{equation}\label{equshortexactprym1}
\xymatrix{
0 \ar[r] & \Lambda_\Sigma[2] \ar[r]^-{\pi^*} & \Lambda_P[2] \ar[r]^-{\epsilon} &  (\mathbb{Z}_2 B)^{\rm ev}/ (b_o) \ar[r] & 0.
}
\end{equation}
\end{proposition}
\begin{proof}
As in the discussion above, we may view $\Lambda_P[2]$ as the group of flat $\mathbb{Z}_2$-local systems on $\Sigma \setminus \{ b_1 , \dots , b_{2l} \}$, modulo the unique non-trivial $\mathbb{Z}_2$-local system corresponding to the double cover $S \setminus \{ u_1 , \dots , u_{2l} \} \to \Sigma \setminus \{ b_1 , \dots , b_{2l} \}$. From this description the result easily follows.
\end{proof}

\begin{proposition}\label{propjoin}
Let $\gamma$ be a path joining distinct branch points $b_i,b_j$ and let $c_\gamma \in \Lambda_P[2]$ the corresponding cycle in $S$. Then
\begin{equation*}
\epsilon( c_\gamma ) = b_i + b_j.
\end{equation*}
Conversely if $c$ is any element of $\Lambda_P[2]$ with $\epsilon(c) = b_i + b_j$, then there exists an embedded path $\gamma$ from $b_i$ to $b_j$ for which $c = c_\gamma$.
\end{proposition}
\begin{proof}
Recall that $c_\gamma$ is the Poincar\'e dual of the cycle $l_\gamma \in H_1(S , \mathbb{Z}_2)$ which is obtained as the pre-image of $\gamma$ under $\pi : S \to \Sigma$. Thus if we view $c_\gamma$ as a certain $\mathbb{Z}_2$-local system on $S$ then the holonomy of $c_\gamma$ around a cycle $l$ in $S$ coincides with the intersection pairing of $l$ with $l_\gamma$. Let $\gamma_{km}$ be a path in $\Sigma$ joining two branch points $b_k$, $b_m$ and let $l_{\gamma_{km}}$ be the pre-image of $\gamma_{km}$ in $S$. We will assume that $\gamma_{km}$ has been chosen so that it is an embedded path in $\Sigma$ from $b_k$ to $b_m$ which avoids all other branch points. Then the intersection of $l_\gamma$ with $l_{\gamma_{km}}$ is the number of elements common to the sets $\{ i , j \}$ and $\{ k , m \}$, taken modulo $2$. On the other hand, we know that there is a lift of $\sigma$ to an involution $\tilde{\sigma}$ of the local system $c_\gamma$. Then $\epsilon(c_\gamma) = \epsilon_1 b_1 + \dots + \epsilon_{2l} b_{2l}$, where $\tilde{\sigma}$ acts on the fibre over $u_i$ as $(-1)^{\epsilon_i}$. The pre-image of $\gamma_{km}$ in $S$ consists of two paths $\gamma_1,\gamma_2$ from $u_k$ to $u_m$. Using $\tilde{\sigma}$ to compare parallel translation along these paths, we see that the holonomy around $l_{\gamma_{km}}$ is $(-1)^{\epsilon_k + \epsilon_m}$. This proves that $\epsilon( c_\gamma ) = b_i + b_j$.\\

To prove the converse it is sufficient to show that any class $a \in H_1(\Sigma , \mathbb{Z}_2 )$ may be represented by an embedded loop. Clearly we can restrict to the case $a \neq 0$. Now we observe that the mapping class group of $\Sigma$ acts on $H^1( \Sigma , \mathbb{Z}_2)$ as the group $Sp(2g , \mathbb{Z}_2)$ and this group acts transitively on $H^1( \Sigma , \mathbb{Z}_2)\setminus \{0\}$. Thus it is enough to find a single class $a \in H^1( \Sigma , \mathbb{Z}_2) \setminus \{0\}$ which can be represented as an embedded loop, which is certainly possible.
\end{proof}

Define a non-degenerate symmetric bilinear form $(( \; , \; )) : \mathbb{Z}_2 B \otimes \mathbb{Z}_2 B \to \mathbb{Z}_2$ by setting $(( b_i , b_j )) = 0$ if $i \neq j$ and $(( b_i , b_i )) = 1$. Note that $(\mathbb{Z}_2 B)^{\rm ev} = ( b_o )^{\perp}$ is the orthogonal complement of $b_o$, so that the restriction of $(( \, , \, ))$ to $(\mathbb{Z}_2 B)^{\rm ev}/ ( b_o ) = ( b_o )^{\perp} / ( b_o )$ is non-degenerate. Note also that $b_o$ is a characteristic for $(( \, , \, ))$, i.e. $((x, x)) = ((x , b_o))$ for any $x \in \mathbb{Z}_2 B$. The induced form on $(\mathbb{Z}_2 B)^{\rm ev}/ ( b_o )$ is thus even, i.e. $((x,x)) = 0$ for any $x \in (\mathbb{Z}_2 B)^{\rm ev}/ ( b_o )$. The subspace $\Lambda_\Sigma[2] \subset \Lambda_P[2]$ is completely null with respect to the restriction of the intersection form $\langle \, , \, \rangle$ to $\Lambda_P[2]$. Moreover, we have:

\begin{proposition}
The restriction of $\langle \, , \, \rangle$ to $\Lambda_P[2]$ is given by the pullback of $(( \, , \, ))$ under the map $\epsilon : \Lambda_P[2] \to (\mathbb{Z}_2 B)^{\rm ev}/(b_o)$. That is:
\begin{equation*}
\langle ( x_1 , y_1 ) , (x_2 , y_2 ) \rangle = ((y_1 , y_2)),
\end{equation*}
for all $(x_1,y_1),(x_2,y_2) \in \Lambda_P[2] \simeq \Lambda_\Sigma[2] \oplus (\mathbb{Z}_2 B)^{\rm ev} / (b_o)$.
\end{proposition}
\begin{proof}
By Proposition \ref{propjoin} and (\ref{equshortexactprym1}), we see that $\Lambda_P[2]$ is spanned by the image of $\Lambda_\Sigma[2]$ together with elements of the form $c_\gamma$, where $\gamma$ is an embedded path joining two branch points. Since $\Lambda_\Sigma[2]$ is completely null with respect to the restriction of $\langle \; , \; \rangle$ to $\Lambda_P[2]$, we just need to verify the proposition for a pair $c_\gamma , c_{\gamma'}$. However this has already been done in the proof of Proposition \ref{propjoin}, where it was shown that if $\gamma$ joins $b_i$ to $b_j$ and $\gamma'$ joins $b_k$ to $b_l$, then $\langle c_\gamma , c_{\gamma'} \rangle$ is the number of elements common to $\{ i,j\}$ and $\{k,l\}$, taken modulo $2$. This is the same as $(( b_i + b_j , b_k + b_l )) = (( \epsilon(c_\gamma) , \epsilon(c_{\gamma'}) ))$.
\end{proof}

From Proposition \ref{propprymcohom} we have a short exact sequence:
\begin{equation}\label{equshortexactprym2}
\xymatrix{
0 \ar[r] & \Lambda_P[2] \ar[r] & \Lambda_S[2] \ar[r]^-{\pi_*} & \Lambda_\Sigma[2] \ar[r] & 0.
}
\end{equation}

\begin{proposition}\label{propsplit1}
Choose a splitting $\Lambda_P[2] = \Lambda_\Sigma[2] \oplus (\mathbb{Z}_2 B)^{\rm ev}/(b_o)$ of (\ref{equshortexactprym1}). Then there exists a splitting of (\ref{equshortexactprym2}), such that under the resulting identifications
\begin{equation*}
\Lambda_S[2] \simeq \Lambda_P[2] \oplus \Lambda_\Sigma[2] \simeq \Lambda_\Sigma[2] \oplus (\mathbb{Z}_2 B)^{\rm ev}/(b_o) \oplus \Lambda_\Sigma[2]
\end{equation*}
given by these splittings, the intersection form on $\Lambda_S[2]$ is given by:
\begin{equation}\label{equintpairing}
\langle ( a , b , c) , (a' , b' , c' ) \rangle = \langle a , c' \rangle + (( b , b' )) + \langle c , a' \rangle
\end{equation}
for all $a,a',c,c' \in \Lambda_\Sigma[2]$, $b,b' \in (\mathbb{Z}_2 B)^{\rm ev}/(b_o)$.
\end{proposition}
\begin{proof}
For notational convenience, set $W = (\mathbb{Z}_2 B)^{\rm ev}/(b_o)$. Choose a splitting of (\ref{equshortexactprym1}), so $\Lambda_P[2] = \Lambda_\Sigma[2] \oplus W$ and we may regard $W$ as a subspace of $\Lambda_S[2]$. The restriction of the intersection form to $W$ is the bilinear form $(( \, , \, ))$, which is non-degenerate. Thus we have an orthogonal splitting $\Lambda_S[2] = W \oplus W^\perp$. The restriction $\pi_*|_{W^\perp} : W^\perp \to \Lambda_\Sigma[2]$ is surjective because $W \subset \Lambda_P[2] = ker(\pi_*)$. The kernel of $\pi_*|_{W^\perp}$ is $ker(\pi_*) \cap W^\perp = \Lambda_P[2] \cap W^\perp = \pi^*(\Lambda_\Sigma[2])$. So we have a short exact sequence:
\begin{equation}\label{equshortexactprym3}
\xymatrix{
0 \ar[r] & \Lambda_\Sigma[2] \ar[r]^-{\pi^*} & W^\perp \ar[r]^-{ \pi_* |_{W^\perp} } & \Lambda_\Sigma[2] \ar[r] & 0.
}
\end{equation}
Let $\iota : \Lambda_\Sigma[2] \to W^\perp$ be a splitting of (\ref{equshortexactprym3}). We say that $\iota$ is an isotropic splitting if the image of $\iota$ is isotropic in $W^\perp$. We claim that an isotropic splitting exists. Indeed, let $\iota$ be any choice of splitting and let $\beta : \Lambda_\Sigma[2] \otimes \Lambda_\Sigma[2] \to \mathbb{Z}_2$ be given by $\beta(a,b) = \langle \iota a , \iota b \rangle$. This is an even symmetric bilinear form on $\Lambda_\Sigma[2]$. Let $F$ be an endomorphism of $\Lambda_\Sigma[2]$. We obtain a new splitting $\iota' : \Lambda_\Sigma[2] \to W^\perp$ by setting $\iota'(a) = \iota(a) + \pi^*( Fa )$. We then find:
\begin{equation*}
\begin{aligned}
\langle \iota'(a) , \iota'(b) \rangle &= \langle \iota(a) + \pi^*(Fa) , \iota(b) + \pi^*( Fb) \rangle \\
&= \beta(a,b) + \langle \iota(a) , \pi^*( Fb) \rangle + \langle \pi^*( Fa ) , \iota(b) \rangle \\
&= \beta(a,b) + \langle \pi_* \iota(a) , Fb \rangle + \langle Fa , \pi_* \iota(b) \rangle \\
&= \beta(a,b) + \langle a , Fb \rangle + \langle Fa , b \rangle.
\end{aligned}
\end{equation*}
Now, since $\beta$ is even and symmetric we can find an $F$ such that $\langle \iota'(a) , \iota'(b) \rangle$ vanishes for all $a,b \in \Lambda_\Sigma[2]$, i.e. $\iota'$ is an isotropic splitting.\\

Given an isotropic splitting $\iota : \Lambda_\Sigma[2] \to W^\perp$, we obtain a splitting of (\ref{equshortexactprym2}) by composing with the inclusion $W^\perp \to \Lambda_S[2]$. Under this splitting, equation (\ref{equintpairing}) follows easily. The only term that needs checking is $\langle ( a , 0 , 0 ) , (0 , 0 , c' ) \rangle$, but this is $\langle \pi^* (a) , \iota(c') \rangle = \langle a , \pi_* \iota(c') \rangle = \langle a , c' \rangle$.
\end{proof}

For the rest of this section we will assume that splittings have been chosen as in Proposition \ref{propsplit1}, so that $\Lambda_S[2] = \Lambda_\Sigma[2] \oplus (\mathbb{Z}_2 B)^{\rm ev}/(b_o) \oplus \Lambda_\Sigma[2]$ with $\pi^*(a) = (a,0,0)$, $\pi_*(a,b,c) = c$ and with intersection pairing given as in Equation (\ref{equintpairing}). We again let $W = (\mathbb{Z}_2 B)^{\rm ev}/(b_o)$. For $1 \le i < j \le 2l$ we let $b_{ij} = b_i + b_j$. Then $b_{12},b_{23},\dots , b_{2l-1,2l}$ span $W$ and are subject to one relation $b_{12} + b_{34} + b_{56} + \dots + b_{2l-1,2l} = 0$. We now proceed to work out the monodromy action on $\Lambda_P[2]$ and $\Lambda_S[2]$.\\

Given an element $c \in \Lambda_S[2]$, we let $s_c$ denote the corresponding Picard-Lefschetz transformation $s_c(x) = x + \langle c , x \rangle c$. By Propositions \ref{propgenerate}, \ref{proploopmono}, \ref{propjoin} and Theorem \ref{thmswapmono} we have that the monodromy action of $\rho$ on $\Lambda_S[2]$ is generated by the involution $\sigma$ together with Picard-Lefschetz transformations $s_c$, where $c$ is any element of $\Lambda_P[2]$ of the form $c = (a , b_{ij} , 0)$, with $a \in \Lambda_\Sigma[2]$ and $1 \le i < j \le 2l$. Consider first those $c$ of the form $c = (0,b_{ij} , 0)$. To simplify notation, we also let $s_{ij}$ denote $s_{(0,b_{ij} , 0)}$. Given a permutation $\omega \in S_{2l}$, we let $\omega$ act on $B = \{ b_1 , b_2 , \dots , b_{2l} \}$ by $\omega(b_i) = b_{\omega(i)}$ and extend this action linearly to $\mathbb{Z}_2 B$. The action preserves $(\mathbb{Z}_2 B)^{\rm ev}$ and descends to $W = (\mathbb{Z}_2 B)^{\rm ev}/(b_o)$. We now find that $s_{ij}$ has the form:
\begin{equation}\label{equtransposition}
s_{ij} = \left[ \begin{matrix} I_{2g} & 0 & 0 \\ 0 & \sigma_{ij} & 0 \\ 0 & 0 & I_{2g} \end{matrix} \right],
\end{equation}
where $\sigma_{ij} \in S_{2l}$ denotes the transposition of $i$ and $j$. Next we define linear transformations $A_{ij}^x$ by $A_{ij}^x = s_{b_{ij}} s_{ b_{ij} + x}$. Then:
\begin{equation}\label{equaijx}
A_{ij}^x = \left[ \begin{matrix} I_{2g} & L_{ij}^x & S^x \\ 0 & I & (L_{ij}^x)^t \\ 0 & 0 & I_{2g} \end{matrix} \right],
\end{equation}
where $L_{ij}^x : W \to \Lambda_\Sigma[2]$ is given by $L_{ij}^x(b) = (( b_{ij} , b))x$, $(L_{ij}^x)^t : \Lambda_\Sigma[2] \to W$ is the adjoint map, $(L_{ij}^x)^t a = \langle x,a \rangle b_{ij}$ and $S^x : \Lambda_\Sigma[2] \to \Lambda_\Sigma[2]$ is given by $S^x(a) = \langle x, a \rangle x$. The monodromy action on $\Lambda_S[2]$ is generated by $\sigma$, the $s_{ij}$ and the $A_{ij}^x$.\\

Suppose that $l = deg(L)$ is even. In this case we define a quadratic refinement $q_W$ of $(( \; , \; ))$ on $W$, i.e. a function $q_W : W \to \mathbb{Z}_2$ satisfying $q_W(a+b) = q_W(a) + q_W(b) + ((a,b))$. The function $q_W$ is given by $q_W(b) = k$ where $b = b_{i_1} + b_{i_2} + \dots + b_{i_{2k}}$ and $i_1,i_2,\dots, i_{2k}$ are distinct. We may then define a quadratic refinement $q$ of $\langle \; , \; \rangle$ on $\Lambda_S[2]$ by setting $q(a,b,c) = \langle a , c\rangle + q_W(b)$.

\begin{lemma}\label{lemrefine}
Suppose that $l$ is even, so that $q : \Lambda_S[2] \to \mathbb{Z}_2$ is defined. In this case, the monodromy action on $\Lambda_S[2]$ preserves $q$.
\end{lemma}
\begin{proof}
By (\ref{equnorm1}), we have $\sigma(x) = -x + \pi^* \pi_*(x) = x + \pi^* \pi_*(x)$, for all $x \in \Lambda_S[2]$. Thus $\sigma( a,b,c) = (a+c,b,c)$. We find $q( \sigma(a,b,c)) = q(a+c,b,c) = \langle a+c , c \rangle + q_W(b) = \langle a,c \rangle + q_W(b) = q(a,b,c)$, so $q$ is $\sigma$ invariant. It remains to show that $q$ is invariant under the Picard-Lefschetz transformations $s_{(a,b_{ij},0)}$. More generally, let $c \in \Lambda_S[2]$ and consider the Picard-Lefschetz transformation $s_c(x) = x + \langle x,c\rangle c$. Then
\begin{equation*}
\begin{aligned}
q( s_c(x) ) &= q( x + \langle x , c \rangle c) \\
&= q(x) + \langle x , c \rangle q(c) + \langle x , c \rangle^2 \\
&= q(x) + \langle x , c \rangle ( q(c) + 1).
\end{aligned}
\end{equation*}
Thus $s_c$ preserves $q$ if and only if $q(c) = 1$. Now if $c = (a , b_{ij} , 0)$, we have $q(c) = q_W( b_{ij} ) = 1$, so $q$ is preserved by these transformations.
\end{proof}

\begin{lemma}\label{lemb1}
Let $B : \Lambda_\Sigma[2] \to \Lambda_\Sigma[2]$ be symmetric, i.e. $\langle Bx,y \rangle = \langle x , By \rangle$. If $l$ is even we further assume $B$ is even, i.e. $\langle Bx, x \rangle = 0$ for all $x$. Then the matrix
\begin{equation}\label{equb1}
\left[ \begin{matrix} I_{2g} & 0 & B \\ 0 & I & 0 \\ 0 & 0 & I_{2g} \end{matrix} \right],
\end{equation}
is realised by products of the $A_{ij}^x$ matrices.
\end{lemma}
\begin{proof}
Let $S^2( \Lambda_\Sigma[2])$ denote the space of symmetric bilinear endomorphisms of $\Lambda_\Sigma[2]$ and $S^{2,ev}( \Lambda_\Sigma[2])$ the space of symmetric, even endomorphisms. For any $x,y \in \Lambda_\Sigma[2]$, define symmetric endomorphisms $B^{x,y}$ and $B^x$ by:
\begin{equation*}
\begin{aligned}
B^{x,y}(a) = \langle x , a \rangle y + \langle y , a \rangle x, & & B^x(a) = \langle x , a \rangle x.
\end{aligned}
\end{equation*}
Note that $S^{2,ev}(\Lambda_\Sigma[2])$ is spanned by the $B^{x,y}$ and $S^2( \Lambda_\Sigma[2] )$ is spanned by the $B^{x,y}$ and $B^x$. Next, define endomorphisms $M^{x,y},N^x$ by:
\begin{equation*}
\begin{aligned}
M^{x,y} = \left[ \begin{matrix} I_{2g} & 0 & B^{x,y} \\ 0 & I & 0 \\ 0 & 0 & I_{2g} \end{matrix} \right], & & N^{x} = \left[ \begin{matrix} I_{2g} & 0 & B^x \\ 0 & I & 0 \\ 0 & 0 & I_{2g} \end{matrix} \right].
\end{aligned}
\end{equation*}
By (\ref{equaijx}), we find that $M^{x,y} = A_{ij}^x A_{ij}^y A_{ij}^{x+y}$, for any $i \neq j$. This proves the result in the case that $l$ is even. Now suppose that $l$ is odd and for any $x$, consider $\hat{N}^x = A_{12}^x A_{34}^x \dots A_{2l-1, 2l}^x$. Since $b_{12} + b_{34} + \dots + b_{2l-1, 2l} = 0$, it is not hard to see that $\hat{N}^x = N^x$ and this proves the result in the case that $l$ is odd.
\end{proof}

\begin{remark}
Note that in the case that $l$ is even, we have $\hat{N}^x = I$.
\end{remark}

\begin{proposition}
Let $G \subseteq GL( \Lambda_S[2] )$ be the group generated by the monodromy action of $\rho$ on $\Lambda_S[2]$. Then $G$ is isomorphic to a semi-direct product $G = S_{2l} \ltimes H$ of the symmetric group $S_{2l}$, generated by the elements $\{ s_{ij} \; | \; i < j \; \}$ given in (\ref{equtransposition}) and the group $H$ generated by the transformations $\{ A_{ij}^x \; | \; i < j, \; x \in \Lambda_\Sigma[2] \; \}$ given in (\ref{equaijx}). The action of $\omega \in S_{2l}$ on $H$ is given by $\omega( A_{ij}^x ) = A_{\omega(i) \omega(j)}^x$.
\end{proposition}
\begin{proof}
The monodromy action is generated by $\sigma$ together with the Picard-Lefschetz transformations of the form $s_{(x,b_{ij},0)}$. Thus $G$ is generated by the $s_{ij}$, the $A_{ij}^x$ and $\sigma$. Clearly the $s_{ij}$ generate the symmetric group $S_{2l}$. If $\sigma_{ij} \in S_{2l}$ denotes the transposition of $i$ and $j$ then it is clear that $s_{ij} A_{kl}^x s_{ij}^{-1} = A_{k'l'}^x$, where $k' = \sigma_{ij}(k)$, $l' = \sigma_{ij}(l)$. The proposition will follow if we can show that the action of $\sigma$ can be expressed as a product of $A_{ij}^x$ terms. Recall that $\sigma( a,b,c ) = (a+c,b,c)$. The result now follows from Lemma \ref{lemb1}, since the identity $I$ is symmetric and even.
\end{proof}

\begin{theorem}\label{thmmonogp}
Let $K$ be the subgroup of elements of $GL( \Lambda_S[2])$ of the form:
\begin{equation}\label{equsubgroup1}
\left[ \begin{matrix} I_{2g} & A & B \\ 0 & I & A^t \\ 0 & 0 & I_{2g} \end{matrix} \right],
\end{equation}
where $A : W \to \Lambda_\Sigma[2]$, $B : \Lambda_\Sigma[2] \to \Lambda_\Sigma[2]$,  and $A^t : \Lambda_\Sigma[2] \to W$ is the adjoint of $A$, so $\langle Ab , c \rangle = (( b , A^tc ))$. Recall that $H$ is the subgroup of $GL(\Lambda_S[2])$ generated by the $A_{ij}^x$. We have:
\begin{enumerate}
\item{If $l$ is odd then $H$ is the subgroup of $K$ preserving the intersection form $\langle \; , \; \rangle$, or equivalently, the elements of $K$ satisfying: 
\begin{equation*}
\langle Bc , c' \rangle + \langle Bc' , c \rangle + \langle A^tc , A^t c' \rangle = 0.
\end{equation*}
}
\item{If $l$ is even then $H$ is the subgroup of $K$ preserving the quadratic refinement $q$ of $\langle \; , \; \rangle$, or equivalently, the elements of $K$ satisfying:
\begin{equation*}
\langle Bc , c \rangle + q_W( A^t c ) = 0.
\end{equation*}
}
\end{enumerate}
\end{theorem}
\begin{proof}
First, note that $H$ is clearly a subgroup of $K$ preserving the intersection form $\langle \; , \; \rangle$, as well as the quadratic refinement $q$ if $l$ is even. Thus it only remains to show that every such element of $K$ is in $H$. By Lemma \ref{lemb1} it is enough to show that for any endomorphism $A : W \to \Lambda_\Sigma[2]$, there is an endomorphism $B : \Lambda_\Sigma[2] \to \Lambda_\Sigma[2]$ for which the corresponding element of $K$ belongs to $H$. But it is easy to see that any such $A$ can be written as a sum of terms of the form $L_{ij}^x$, as in (\ref{equaijx}). Taking the corresponding product of $A_{ij}^x$ terms, we obtain the desired element of $H$.
\end{proof}

\begin{remark}
The structure of the group $H$ generated by the $A_{ij}^x$ may be described as follows. As in Lemma \ref{lemb1}, we have relations: 
\begin{equation*}
\begin{aligned}
A_{ij}^x A_{ij}^y A_{ij}^{x+y} &= M^{x,y}, & & A_{12}^x A_{34}^x \dots A_{2l-1, 2l}^x = \begin{cases} I & \text{if } l \text{ is even}, \\  
N^x & \text{if } l \text{ is odd}. \end{cases}
\end{aligned}
\end{equation*}
In addition, we have commutation relations:
\begin{equation*}
[ A_{ij}^x , A_{kl}^y ] = \begin{cases} I & \text{if } ((b_{ij} , b_{kl} )) = 0, \\  
M^{x,y} & \text{if } ((b_{ij},b_{kl})) = 1. \end{cases}
\end{equation*}
In particular, this shows that $H$ is a central extension of $( Hom( W , \Lambda_\Sigma[2] ) , + )$ by $S^{2,ev}(\Lambda_\Sigma[2])$ when $l$ is even and by $S^2(\Lambda_\Sigma[2])$ when $l$ is odd.
\end{remark}

\begin{corollary}\label{cormono1}
Let $G_P \subseteq GL( \Lambda_P[2] )$ be the group generated by the monodromy action of $\check{\rho}$ on $\Lambda_P[2]$. Then $G_P$ is the set of matrices of the form
\begin{equation}\label{equmonoprym}
M = \left[ \begin{matrix} I_{2g} & A \\ 0 & \omega \end{matrix} \right],
\end{equation}
where $\omega \in S_{2l}$ is any permutation and $A$ is any endomorphism $A : W \to \Lambda_\Sigma[2]$.
\end{corollary}

Corollary \ref{cormono1} was originally proven in \cite{sch1}. We can likewise describe the monodromy representation $\hat{\rho}$ on the dual $(\Lambda_P[2])^* \simeq \Lambda_S[2]/ \pi^*( \Lambda_\Sigma[2]) \simeq W \oplus \Lambda_\Sigma[2]$ as follows:

\begin{corollary}
Let $G_{P^*} \subseteq GL( (\Lambda_P[2])^* )$ be the group generated by the monodromy action of $\hat{\rho}$ on $(\Lambda_P[2])^*$. Then $G_{P^*}$ is the set of matrices of the form
\begin{equation}\label{equmonoprym}
M = \left[ \begin{matrix} \omega & C \\ 0 & I_{2g} \end{matrix} \right],
\end{equation}
where $\omega \in S_{2l}$ is any permutation and $C$ is any endomorphism $C : \Lambda_\Sigma[2] \to W$.
\end{corollary}

%%%%%%%%%%%%%%%%%%%%%%%%%%%%%%%%%%%%%%%%%%%%
\subsection{Monodromy action on $\widetilde{\Lambda}_P[2]$}\label{secmonoact2}
%%%%%%%%%%%%%%%%%%%%%%%%%%%%%%%%%%%%%%%%%%%%

For later applications we need to consider a certain $\mathbb{Z}_2$-extension of $\Lambda_P^k[2]$ and the corresponding lift $\check{\tilde{\beta}}$ of $\check{\beta}$. Let $A$ be a degree $k$ line bundle on $\Sigma$. We define $\widetilde{\Lambda}_P^k[2]$ to be the covering space of $\mathcal{A}_{\rm reg}^0$ whose fibre over the spectral curve $S$ is the set of pairs $(M,\tilde{\sigma})$, where $M \in Jac_k(S)$ satisfies $M^2 = \pi^*(A)$ and $\tilde{\sigma} : M \to M$ is an involution covering $\sigma$, so in particular $M \simeq \sigma(M)$. Then $\widetilde{\Lambda}_P[2]$ is a bundle of groups which is a $\mathbb{Z}_2$-extension of $\Lambda_P[2]$ and $\widetilde{\Lambda}_P^k[2]$ is a bundle of $\widetilde{\Lambda}_P[2]$-torsors.\\

We now determine the monodromy action on $\widetilde{\Lambda}_P[2]$. Recall as in Section \ref{secmonoact}, we have a natural map $\epsilon : \widetilde{\Lambda}_P[2] \to (\mathbb{Z}_2 B)^{\rm ev}$ which sends an equivariant line bundle $(M , \tilde{\sigma})$ to $\epsilon_1 b_1 + \dots + \epsilon_{2l} b_{2l}$, where $\tilde{\sigma}$ acts on $M_{u_i}$ by $(-1)^{\epsilon_i}$. Similar to Proposition \ref{propprymexact1}, we have a short exact sequence:
\begin{equation}\label{equshortexactprym4}
\xymatrix{
0 \ar[r] & \Lambda_\Sigma[2] \ar[r]^-{\pi^*} & \widetilde{\Lambda}_P[2] \ar[r]^-{\epsilon} &  (\mathbb{Z}_2 B)^{\rm ev} \ar[r] & 0.
}
\end{equation}
Choose a splitting of (\ref{equshortexactprym4}) which we may assume is compatible with our previously chosen splitting of (\ref{equshortexactprym1}). Thus we have an identification $\widetilde{\Lambda}_P[2] = \Lambda_\Sigma[2] \oplus (\mathbb{Z}_2 B)^{\rm ev}$. This allows us to identify $\Lambda_P[2]$ with the quotient $\widetilde{\Lambda}_P[2] / (b_o)$. The natural action of $S_{2l}$ on the set of branch points $B$ extends by linearity to $\mathbb{Z}_2 B$. Then:

\begin{proposition}
The image of the monodromy group in $GL( \widetilde{\Lambda}_P[2])$ is the set of matrices of the form
\begin{equation}
M = \left[ \begin{matrix} I_{2g} & A \\ 0 & \omega \end{matrix} \right],
\end{equation}
where $\omega \in S_{2l}$ is any permutation and $A$ is any endomorphism $A : (\mathbb{Z}_2 B)^{\rm ev} \to \Lambda_\Sigma[2]$ for which $A(b_o) = 0$.
\end{proposition}
\begin{proof}
This is a straightforward extension of Corollary \ref{cormono1}. All that needs to be checked is that the monodromy action arising from a swap of branch points $b_i$, $b_j$ acts on $(\mathbb{Z}_2 B)^{\rm ev}$ as the transposition of $i$ and $j$. But this is clearly seen to be the case by thinking of elements of $\widetilde{\Lambda}_P[2]$ as $\sigma$-equivariant line bundles on $S$.
\end{proof}

%%%%%%%%%%%%%%%%%%%%%%%%%%%%%%%%%%%%%%%%%%%%
\subsection{Affine monodromy representations}\label{secmonoact3}
%%%%%%%%%%%%%%%%%%%%%%%%%%%%%%%%%%%%%%%%%%%%

Having determined the monodromy actions on $\Lambda_P[2],$ $\Lambda_S[2],$ $\widetilde{\Lambda}_P[2]$, we now turn to their affine counterparts. Let $(\mathbb{Z}_2 B)^{\rm odd}$ be the complement $\mathbb{Z}_2 B - (\mathbb{Z}_B)^{\rm ev}$ and for any $k \in \mathbb{Z}$, let $(\mathbb{Z}_2 B)^k$ equal $(\mathbb{Z}_2 B)^{\rm ev}$ or $(\mathbb{Z}_2 B)^{\rm odd}$ according to whether $k$ is even or odd. Elements of $\Lambda_P^k[2]$ are line bundles $M$ which satisfy $\sigma(M) \simeq M$. Then, as in Section \ref{secmonoact}, there is a map $\epsilon : \Lambda_P^k[2] \to (\mathbb{Z}_2 B)/(b_o)$ defined as follows: choose an involutive lift $\tilde{\sigma} : M \to M$ of $\sigma$ and let $\epsilon(M) = \epsilon_1 b_1 + \dots \epsilon_{2l} b_{2l}$, where $\tilde{\sigma}$ acts on $M_{u_i}$ by $(-1)^{\epsilon_i}$. 

\begin{lemma}\label{lemlift1}
Let $M = \mathcal{O}(u_j)$. Then $\sigma(M) \simeq M$, since $\sigma(u_j) = u_j$. Let $\tilde{\sigma} : M \to M$ be the unique involutive isomorphism covering $\sigma$ and such that $\tilde{\sigma}$ acts as on $M_{u_j}$ as $-1$. Then for any $i \neq j$, we have that $\tilde{\sigma}$ acts on $M_{u_i}$ as the identity. Thus $\epsilon( (M , \tilde{\sigma} ) ) = b_j$.
\end{lemma}
\begin{proof}
Let $s$ be a non-trivial holomorphic section of $M$. The space of holomorphic sections of $M$ is spanned by $s$, so $\tilde{\sigma}(s) = \pm s$. Since $s$ vanishes to first order at $u_j$ it is easy to see that in fact we must have $\tilde{\sigma}(s) = s$. For any $i \neq j$, we have that $s$ is non-vanishing at $u_i$, but $\tilde{\sigma}( s(u_i) ) = s(u_i)$. Thus $\tilde{\sigma}$ must act as the identity on $M_{u_i}$.
\end{proof}

\begin{remark}
Lemma \ref{lemlift1} implies that the map $\epsilon : \Lambda_P^k[2] \to (\mathbb{Z}_2 B)/(b_o)$ actually takes values in $(\mathbb{Z}_2 B)^k/(b_o)$.
\end{remark}

Let $W^1 = (\mathbb{Z}_2 B)^{\rm odd}/(b_o)$, which is an affine space modelled on $(\mathbb{Z}_2 B)^{\rm ev}/(b_o)$. Now choose splittings as in Proposition \ref{propsplit1}, so that we have identifications $\Lambda_P[2] = \Lambda_\Sigma[2] \oplus W$ and $\Lambda_S[2] = \Lambda_\Sigma[2] \oplus W \oplus \Lambda_\Sigma[2]$. By Lemma \ref{lemlift1}, we have $\epsilon(\Theta) = b_1$. We will identify $\Theta$ with the point $(0,b_1,0)$ in $\Lambda_\Sigma[2] \oplus W^1 \oplus \Lambda_\Sigma[2]$ and if $(a,b,c) \in \Lambda_\Sigma[2] \oplus W \oplus \Lambda_\Sigma[2]$ then we identify $(a,b,c)+\Theta$ with $(a,b+b_1,c) \in \Lambda_\Sigma[2] \oplus W^1 \oplus \Lambda_\Sigma[2]$. In this way, we have obtained identifications 
\begin{equation*}
\begin{aligned}
\Lambda_P^1[2] &= \Lambda_\Sigma[2] \oplus W^1, \\
\Lambda_S^1[2] &= \Lambda_\Sigma[2] \oplus W^1 \oplus \Lambda_\Sigma[2].
\end{aligned}
\end{equation*}

Let $b \in (\mathbb{Z}_2 B)^{\rm ev}$ and $b' \in (\mathbb{Z}_2 B)^{\rm odd}/(b_o) = W^1$. Then the pairing $(( b , b' ))$ is well-defined because $b_o$ is orthogonal to $(\mathbb{Z}_2 B)^{\rm ev}$. Similarly if $(a,b,c) \in \Lambda_\Sigma[2] \oplus (\mathbb{Z}_2 B)^{\rm ev} \oplus \Lambda_\Sigma[2]$ and $(a',b',c') \in \Lambda_\Sigma[2] \oplus W^1 \oplus \Lambda_\Sigma[2]$, we set:
\begin{equation*}
\langle (a,b,c) , (a',b',c') \rangle = \langle a , c' \rangle + ((b,b')) + \langle a, c' \rangle.
\end{equation*}

We now turn to the computation of the monodromy action on $\Lambda_S^1[2]$. First recall that $\Lambda_S^1[2] = \{ M \in Jac_1(S) \; | \; M^2 = \pi^*(A) \; \}$, where $A$ is a degree $1$ line bundle on $\Sigma$. As before, we take $A = \mathcal{O}(b_1)$ and set $N = \mathcal{O}(u_1) \in \Lambda_S^1[2]$. Recall from Section \ref{sectcc} that the cocycle $\check{\beta}$ is given by $\check{\beta}(g) = g.\Theta - \Theta$. Any $x \in \Lambda_S^1[2]$ can be written uniquely as $x = a + \Theta$, where $a \in \Lambda_S[2]$. Then the monodromy action of $g$ on $x$ has the form
\begin{equation*}
\begin{aligned}
g.x &= g.a + g.\Theta \\
&= g.a + \check{\beta}(g) + \Theta.
\end{aligned}
\end{equation*}
This can be viewed as an affine action $a \mapsto g.a + \check{\beta}(g)$. We determine this action.

\begin{proposition}\label{propaffmono1}
The monodromy action of $\tau$ as in Definition \ref{deftau} acts on $\Lambda_S^1[2]$ by $(a,b,c) \mapsto (a+c,b,c)$. Let $\gamma$ be a path in $\Sigma$ joining $b_i$ to $b_j$. Then the monodromy action of a lift of the swap along $\gamma$ acts on $\Lambda_S^1[2]$ as a Picard-Lefschetz transformation $s_{c_\gamma}$, that is:
\begin{equation*}
s_{c_\gamma}(x) = x + \langle c_\gamma , x \rangle c_\gamma.
\end{equation*}
\end{proposition}
\begin{proof}
Let $\tau$ be the loop generated by the $\mathbb{C}^*$-action. The first statement of the proposition holds because $\check{\beta}(\tau) = 0$. Using Theorems \ref{thmswapmono} and \ref{thmbeta}, the monodromy action associated by a path $\gamma$ is given by:
\begin{equation*}
\begin{aligned}
\gamma( a +\Theta) = \gamma(a) + \check{\beta}(\gamma) + \Theta &= \begin{cases} a + \langle c_\gamma , a \rangle c_\gamma + \Theta & \text{if } ((b_{ij} , b_{1} )) = 0 \\  
a + \langle c_\gamma , a \rangle c_\gamma +c_\gamma + \Theta & \text{if } ((b_{ij},b_{1})) = 1 \end{cases}\\
& = a + \langle a + \Theta , c_\gamma \rangle c_\gamma + \Theta \\
&= s_{c_\gamma}( a+\Theta).
\end{aligned}
\end{equation*}
This shows that the affine monodromy action can be expressed as a Picard-Lefschetz transformation as claimed.
\end{proof}

Recall that we have defined the bundle of groups $\widetilde{\Lambda}_P[2]$ and the $\widetilde{\Lambda}_P[2]$-torsor $\widetilde{\Lambda}_P^1[2]$. We now determine the affine monodromy action on this space. Let $\check{\tilde{\beta}} \in H^1( \mathcal{A}_{\rm reg}^0 , \widetilde{\Lambda}_P[2] )$ be the class corresponding to the torsor $\widetilde{\Lambda}_P^1[2]$. Let $\gamma$ be a path in $\Sigma$ between branch points $b_i,b_j$ and $c_\gamma$ the corresponding class in $H^1(S , \mathbb{Z}_2)$. From Proposition \ref{propjoin}, it follows that we can uniquely lift $c_\gamma$ to an element of $\widetilde{\Lambda}_P[2]$ by requiring $\epsilon( c_\gamma ) = b_i + b_j \in (\mathbb{Z}_2 B)^{\rm ev}$.

\begin{proposition}\label{propbeta2}
Let $\tau$ be the loop in $\mathcal{A}_{\rm reg}^0$ generated by the $\mathbb{C}^*$-action. Then $\check{\tilde{\beta}}(\tau) = 0$. Let $\tilde{s}_\gamma \in \pi_1( \mathcal{A}_{\rm reg}^0 , a_0)$ be a lift of a swap of $b_i,b_j$ along the path $\gamma$. 
\begin{equation*}
\check{\tilde{\beta}}( \tilde{s}_\gamma ) = \begin{cases} 0 & \text{if } 1 \notin \{i,j\}, \\  
c_\gamma & \text{if } 1 \in \{i,j\}. \end{cases}
\end{equation*}
\end{proposition}
\begin{proof}
This is a straightforward refinement of Theorem \ref{thmbeta}. Recall that we have taken $\Theta = \mathcal{O}(u_1)$ as an origin in $\Lambda_P^1[2]$. We lift this to an origin $\tilde{\Theta} = (\Theta , \tilde{\sigma} ) \in \widetilde{\Lambda}_P^1[2]$ by letting $\tilde{\sigma}$ be the lift of $\sigma$ acting as $-1$ on $\Theta_{u_1}$. Then by Lemma \ref{lemlift1}, we find $\epsilon( \tilde{\Theta} ) = b_1$. We then have $\check{\tilde{\beta}}(\tau) = 0$, because $\tilde{\sigma}$ is an involution covering $\sigma$, so $\sigma^*( \Theta , \tilde{\sigma}) \simeq (\Theta , \tilde{\sigma})$.\\

Let $\tilde{s}_\gamma \in \pi_1( \mathcal{A}_{\rm reg}^0 , a_0)$ be the lift of a swap along the path $\gamma$. Consider $\tilde{s}_\gamma$ as a loop in $\mathcal{A}_{\rm reg}^0$ based at $a_0$. Recall that we had defined $q : [0,1] \to \Lambda_P^1[2]$ as the unique lift of $\tilde{s}_\gamma$ to a path in $\Lambda_P^1[2]$ with $q(0) = \Theta$, so $q(1) = \check{\beta}(\gamma) q(0)$. Similarly let $\tilde{q}(t)$ be the unique lift of $q(t)$ to a path in $\widetilde{\Lambda}_P^1[2]$ starting at $\tilde{\Theta}$. Suppose that $\gamma$ is a path from $b_i$ to $b_j$ and recall there were three cases: (i) $1 \notin \{ i,j \}$, (ii) $i=1$ and (iii) $j=1$.\\

In case (i), we had $q(t) = q(0)$, hence we also have $\tilde{q}(t) = \tilde{q}(0)$ and $\check{\tilde{\beta}}(\gamma) = 0$. In case $(ii)$ we had
\begin{equation*}
\check{\beta}(\gamma) = q(1) \otimes q(0)^* = \mathcal{O}( u_j ) \otimes \mathcal{O}(u_1)^* \otimes \pi^*( \Gamma(1)^* ).
\end{equation*}
Correspondingly, we obtain
\begin{equation*}
\check{\tilde{\beta}}(\gamma) = \tilde{q}(1) \otimes \tilde{q}(0)^* = \tilde{\mathcal{O}}( u_j ) \otimes \tilde{\mathcal{O}}(u_1)^* \otimes \pi^*( \Gamma(1)^* ),
\end{equation*}
where $\tilde{\mathcal{O}}(u_j)$ denotes $\mathcal{O}(u_j)$ together with the involutive lift of $\sigma$ which acts as $-1$ over $u_j$. Thus $\epsilon( \tilde{\mathcal{O}}(u_j) ) = b_j$. Note also that the pullback of any line bundle on $\Sigma$ comes with a canonical involutive lift of $\sigma$ (which acts trivially over the fixed points). Therefore $\epsilon( \check{\tilde{\beta}}(\gamma) ) = b_1 + b_j = \epsilon( c_\gamma )$, proving the proposition in this case. Case (iii) is similar.
\end{proof}

\begin{proposition}
The monodromy action of $\tau$ as in Definition \ref{deftau} acts on $\widetilde{\Lambda}_P^1[2]$ trivially. Let $\gamma$ be a path in $\Sigma$ joining $b_i$ to $b_j$. Then the monodromy action of a lift of the swap along $\gamma$ acts on $\widetilde{\Lambda}_P^1[2]$ as a Picard-Lefschetz transformation $s_{c_\gamma}$, that is:
\begin{equation*}
s_{c_\gamma}(x) = x + \langle c_\gamma , x \rangle c_\gamma.
\end{equation*}
\end{proposition}
\begin{proof}
This is proved in exactly the same way as Proposition \ref{propaffmono1}.
\end{proof}

%%%%%%%%%%%%%%%%%%%%%%%%%%%%%%%%%%%%%%%%%%%%
%%%%%%%%%%%%%%%%%%%%%%%%%%%%%%%%%%%%%%%%%%%%
% SECTION
%%%%%%%%%%%%%%%%%%%%%%%%%%%%%%%%%%%%%%%%%%%%
%%%%%%%%%%%%%%%%%%%%%%%%%%%%%%%%%%%%%%%%%%%%
\section{Real twisted Higgs bundles and monodromy}\label{secrthbm}

%%%%%%%%%%%%%%%%%%%%%%%%%%%%%%%%%%%%%%%
\subsection{Real twisted Higgs bundles}
%%%%%%%%%%%%%%%%%%%%%%%%%%%%%%%%%%%%%%%

In Section \ref{secthb}, we defined twisted Higgs bundles moduli spaces $\mathcal{M}(r,d,L)$, $\check{\mathcal{M}}(r,D,L)$, $\hat{\mathcal{M}}(r,d,L)$ corresponding to the complex groups $GL(2,\mathbb{C}), SL(2,\mathbb{C})$ and $PGL(2,\mathbb{C}) = PSL(2,\mathbb{C})$. We now consider real analogues of these moduli spaces. In general, for any real reductive Lie group $G$, one may define $L$-twisted $G$-Higgs bundles and construct a moduli space of polystable $G$-Higgs bundles \cite{ggm0}. Here we recall the definitions in the cases $G = GL(2,\mathbb{R}), SL(2,\mathbb{R}), PGL(2,\mathbb{R})$ and $PSL(2,\mathbb{R})$.

\begin{definition}
\begin{itemize} We have:
\item[(1)]{An {\em $L$-twisted $GL(2,\mathbb{R})$-Higgs bundle} is a pair $(E,\Phi)$, where $E$ is a rank $2$ holomorphic vector bundle with orthogonal structure $\langle \; , \; \rangle : E \otimes E \to \mathbb{C}$ and $\Phi$ is a holomorphic section of $L \otimes End(E)$ which is symmetric, i.e. $\langle \Phi u , v \rangle = \langle u , \Phi v \rangle$.}
\item[(2)]{An {\em $L$-twisted $SL(2,\mathbb{R})$-Higgs bundle} is a triple $(N,\beta,\gamma)$, where $N$ is a holomorphic line bundle, $\beta \in H^0(\Sigma , N^2L)$ and $\gamma \in H^0(\Sigma , N^{-2}L)$.}
\item[(3)]{An {\em $L$-twisted $PGL(2,\mathbb{R})$-Higgs bundle} is an equivalence class of triple $(E,\Phi , A)$, where $E$ is a rank $2$ holomorphic vector bundle equipped with a symmetric, non-degenerate bilinear pairing $\langle \; , \; \rangle : E \otimes E \to A$ valued in a line bundle $A$ and $\Phi$ is a holomorphic section of $L \otimes End(E)$ which is trace-free and symmetric, i.e. $\langle \Phi u , v \rangle = \langle u , \Phi v \rangle$. Two triples $(E,\Phi,A),(E',\Phi',A')$ are considered equivalent if there is a holomorphic line bundle $B$ such that $(E',\Phi',A') = (E \otimes B , \Phi' \otimes Id , A \otimes B^2 )$ with the induced pairing $(E\otimes B) \otimes (E \otimes B) \to A \otimes B^2$.}
\item[(4)]{An {\em $L$-twisted $PSL(2,\mathbb{R})$-Higgs bundle} is an equivalence class of quadruple $(N_1,N_2,\beta,\gamma)$, where $N_1,N_2$ are holomorphic line bundles, $\beta \in H^0(\Sigma , N_1 N_2^* L )$ and $\gamma \in H^0( \Sigma , N_2 N_1^* L)$. Two quadruples $(N_1,N_2,\beta,\gamma),(N_1',N_2',\beta',\gamma')$ are considered equivalent if there is a holomorphic line bundle $B$ such that $(N_1',N_2',\beta',\gamma') = (N_1B , N_2B , \beta , \gamma)$.}
\end{itemize}
\end{definition}

\begin{remark}
We have the following relations between Higgs bundles for various real and complex groups:
\begin{itemize}
\item[(1)]{A $GL(2,\mathbb{R})$-Higgs bundle $(E,\Phi)$ is in a natural way a $GL(2,\mathbb{C})$-Higgs bundle.}
\item[(2)]{An $SL(2,\mathbb{R})$-Higgs bundle $(N,\beta,\gamma)$ determines a $GL(2,\mathbb{R})$-Higgs bundle $( E , \Phi )$, where $E = N \oplus N^*$ equipped with the natural pairing of $N$ and $N^*$, and $\Phi = \left[ \begin{matrix} 0 & \beta \\ \gamma & 0 \end{matrix} \right]$. Note that $(E,\Phi)$ constructed in this manner is trace-free of trivial determinant so can also be thought of as an $SL(2,\mathbb{C})$-Higgs bundle.}
\item[(3)]{Note that $(N,\beta,\gamma)$ and $(N^* , \gamma , \beta)$ define the same underlying $GL(2,\mathbb{R})$-Higgs bundle, but are generally distinct as $SL(2,\mathbb{R})$-Higgs bundles.}
\item[(4)]{A $PGL(2,\mathbb{R})$-Higgs bundle $(E,\Phi, A)$ can be considered as a $PGL(2,\mathbb{C})$-Higgs bundle $(E,\Phi)$.}
\item[(5)]{A $PSL(2,\mathbb{R})$-Higgs bundle $(N_1,N_2,\beta,\gamma)$ determines a $PGL(2,\mathbb{R})$-Higgs bundle $(E , \Phi , A)$, where $A = N_1N_2$, $E = N_1 \oplus N_2$ equipped with the natural $A$-valued pairing of $N_1$ and $N_2$, and $\Phi = \left[ \begin{matrix} 0 & \beta \\ \gamma & 0 \end{matrix} \right]$.}
\item[(6)]{Note that $(N_1,N_2,\beta,\gamma)$ and $(N_2,N_1,\gamma,\beta)$ define the same underlying $PGL(2,\mathbb{R})$-Higgs bundle.}
\end{itemize}
\end{remark}

As in \cite{ggm}, one may introduce notions of stability, semistability and polystability and construct moduli spaces of polystable $L$-twisted Higgs bundles for real reductive groups. We recall these definitions for the relevant groups.
\begin{definition} We have the following definitions:
\begin{itemize}
\item[(1)]{An $L$-twisted $GL(2,\mathbb{R})$-Higgs bundle $(E,\Phi)$ is {\em stable} (resp. {\em semistable}) if for any $\Phi$-invariant isotropic line subbundle $N \subset E$ we have $deg(N) < 0$ (resp. $deg(N) \le 0$). We say $(E,\Phi)$ is {\em polystable} if either (i) $(E,\Phi)$ is stable, or (ii) $\Phi = \alpha . Id$, for some $\alpha \in H^0(\Sigma , K)$ and $E = N \oplus N^*$ for a degree $0$ line bundle $N$, where the orthogonal structure on $E$ is the dual pairing of $N$ and $N^*$.}
\item[(2)]{An $L$-twisted $SL(2,\mathbb{R})$-Higgs bundle $(N,\beta,\gamma)$ is {\em stable} (resp. {\em semistable, polystable}) if the associated $GL(2,\mathbb{R})$-Higgs bundle is stable (resp. semistable, polystable).}
\item[(3)]{An $L$-twisted $PGL(2,\mathbb{R})$-Higgs bundle represented by $(E,\Phi , A)$ is {\em stable} (resp. {\em semistable}) if for any $\Phi$-invariant isotropic line subbundle $N \subset E$ we have $deg(N) < deg(A)/2$ (resp. $deg(N) \le deg(A)/2$). We say $(E,\Phi,A)$ is {\em polystable} if either (i) $(E,\Phi , A)$ is stable, or (ii) $\Phi = 0$ and $E = N_1 \oplus N_2$, where $deg(N_1) = deg(N_2) = deg(A)/2$, $A = N_1N_2$ and the orthogonal structure on $E$ is the $A$-valued pairing of $N_1$ and $N_2$.}
\item[(4)]{An $L$-twisted $PSL(2,\mathbb{R})$-Higgs bundle is {\em stable} (resp. {\em semistable, polystable}) if the associated $PGL(2,\mathbb{R})$-Higgs bundle is stable (resp. semistable, polystable).}
\end{itemize}
\end{definition}

According to these definitions, we can associate to any semistable Higgs bundle an associated polystable Higgs bundle. This defines a notion of $S$-equivalence and allows us to define moduli spaces of $S$-equivalence classes of semistable real Higgs bundles. Equivalently, these may be defined as moduli spaces of polystable real Higgs bundles:
\begin{definition}
We define the following moduli spaces:
\begin{enumerate}
\item{Let ${}^{\mathbb{R}} \mathcal{M}(L)$ denote the moduli space of polystable $L$-twisted $GL(2,\mathbb{R})$-Higgs bundles. We futher let ${}^{\mathbb{R}} \mathcal{M}^0(L)$ denote the moduli space of trace-free polystable $L$-twisted $GL(2,\mathbb{R})$-Higgs bundles.}
\item{Let ${}^{\mathbb{R}} \check{\mathcal{M}}(L)$ denote the moduli space of polystable $L$-twisted $SL(2,\mathbb{R})$-Higgs bundles.}
\item{Let ${}^{\mathbb{R}} \hat{\mathcal{M}}(d,L)$ denote the moduli space of polystable $L$-twisted $PGL(2,\mathbb{R})$-Higgs bundles with fixed value of $d$, where $d \in \mathbb{Z}_2$ is the mod $2$ degree of the associated $PGL(2,\mathbb{C})$-Higgs bundle.}
\item{Let ${}^{\mathbb{R}} \widetilde{\mathcal{M}}(d,L)$ denote the moduli space of polystable $L$-twisted $PSL(2,\mathbb{R})$-Higgs bundles with fixed value of $d$, where $d \in \mathbb{Z}_2$ is the mod $2$ degree of the associated $PGL(2,\mathbb{C})$-Higgs bundle.}
\end{enumerate}
\end{definition}

Under the natural map taking a twisted Higgs bundle for a real group to the corresponding complex group, we see that the conditions of semistability and polystability are preserved. Therefore we have natural maps from the moduli spaces of real Higgs bundles to the corresponding moduli spaces of complex Higgs bundles, namely:
\begin{itemize}
\item[(1)]{${}^{\mathbb{R}} \mathcal{M}(L) \to \mathcal{M}(0,L)$, corresponding to $GL(2,\mathbb{R}) \to GL(2,\mathbb{C})$,}
\item[(2)]{${}^{\mathbb{R}} \mathcal{M}^0(L) \to \mathcal{M}^0(0,L)$, corresponding to $GL(2,\mathbb{R}) \to GL(2,\mathbb{C})$ for trace-free Higgs bundles,}
\item[(3)]{${}^{\mathbb{R}} \check{\mathcal{M}}(L) \to \check{\mathcal{M}}(\mathcal{O} , L)$, corresponding to $SL(2,\mathbb{R}) \to SL(2,\mathbb{C})$,}
\item[(4)]{${}^{\mathbb{R}} \hat{\mathcal{M}}(d,L) \to \hat{\mathcal{M}}(d,L)$, corresponding to $PGL(2,\mathbb{R}) \to PGL(2,\mathbb{C})$,}
\item[(5)]{${}^{\mathbb{R}} \widetilde{\mathcal{M}}(d,L) \to \hat{\mathcal{M}}(d,L)$, corresponding to $PSL(2,\mathbb{R}) \to PGL(2,\mathbb{C})$.}
\end{itemize}
We then define the regular loci ${}^{\mathbb{R}} \mathcal{M}_{\rm reg}(L), {}^{\mathbb{R}} \mathcal{M}^0_{\rm reg}(L), {}^{\mathbb{R}} \check{\mathcal{M}}_{\rm reg}(L), {}^{\mathbb{R}} \hat{\mathcal{M}}_{\rm reg}(d,L)$ and ${}^{\mathbb{R}} \widetilde{\mathcal{M}}_{\rm reg}(d,L)$ to be the open subsets in the real moduli spaces whose underlying complex Higgs bundle maps to $\mathcal{A}_{\rm reg}$ under the Hitchin map.

%%%%%%%%%%%%%%%%%%%%%%%%%%%%%%%%%%%%%%%%%%%%%
\subsection{Spectral data and monodromy for real Higgs bundles}\label{secsdmr}
%%%%%%%%%%%%%%%%%%%%%%%%%%%%%%%%%%%%%%%%%%%%%

In what follows we will assume that $l = deg(L)$ is even. We then fix a choice of a line bundle $L^{1/2}$ on $\Sigma$ whose square is $L$. Let $a_0 \in \mathcal{A}_{\rm reg}^0(L)$ and $\pi : S \to \Sigma$ the corresponding spectral curve. Given a line bundle $M \in Jac_k(S)$, we write $M = M_0 \otimes \pi^*(L^{1/2})$, where $M_0 \in Jac_{k-l}(S)$. As usual the Higgs bundle $(E,\Phi)$ associated to $M$ is given by $E = \pi_*(M) = \pi_*( M_0 \otimes \pi^*(L^{1/2}))$ and $\Phi$ is obtained from the tautological section $\lambda : M \to M \otimes \pi^*(L)$.

\begin{proposition}\label{proprsd}
Under the spectral data construction sending $M_0 \in Pic(S)$ to $E = \pi_*(M \otimes \pi^*(L^{1/2}))$, we have that real Higgs bundles lying over $a_0$ correspond to the following data:
\begin{itemize}
\item[(1)]{For $GL(2,\mathbb{R})$, these are line bundles $M_0 \in Jac(S)$ such that $M_0^2 = \mathcal{O}$, i.e. the space $\Lambda_S[2]$.}
\item[(2)]{For $SL(2,\mathbb{R})$, these are line bundles $M_0 \in Jac(S)$ such that $M_0^2 = \mathcal{O}$, together with an involutive automorphism $\tilde{\sigma} : M_0 \to M_0$ covering $\sigma$, i.e. the space $\widetilde{\Lambda}_P[2]$.}
\item[(3)]{For $PGL(2,\mathbb{R})$, these are line bundles $M_0 \in Pic(S)$ such that $M_0^2 = \pi^*(A)$, for some $A \in Pic(\Sigma)$ modulo $M_0 \mapsto M_0 \otimes \pi^*(B)$, $B \in Pic(\Sigma)$. This space is isomorphic to $\left( \Lambda_S^0[2] \oplus \Lambda_S^1[2] \right) / \pi^*\Lambda_\Sigma[2]$.}
\item[(4)]{For $PSL(2,\mathbb{R})$, these are line bundles $M_0 \in Pic(S)$, together with an involutive automorphism $\tilde{\sigma} : M_0 \to M_0$ covering $\sigma$, modulo $M_0 \mapsto M_0 \otimes \pi^*(B)$, $B \in Pic(\Sigma)$. This space is isomorphic to $\left( \widetilde\Lambda_P^0[2] \oplus \widetilde{\Lambda}_P^1[2] \right) / \pi^* \Lambda_\Sigma[2]$.}
\end{itemize}
\end{proposition}
\begin{proof}
Let $(E,\Phi)$ be a $GL(2,\mathbb{C})$-Higgs bundle associated to the line bundle $M_0$. Thus $E = \pi_*( M_0 \otimes \pi^*(L^{1/2}))$ and $\Phi$ is obtained from $\lambda : M_0 \otimes \pi^*(L^{1/2}) \to M_0\otimes \pi^*( L^{3/2} )$. If $M_0^2  = \mathcal{O}$ then $M_0$ has an orthogonal structure. As in \cite{sch0}, it follows by relative duality that $E$ has an orthogonal structure. Moreover, $\Phi$ is clearly symmetric with this orthogonal structure, so we have obtained a $GL(2,\mathbb{R})$-Higgs bundle. Conversely, if $(E,\Phi)$ is a $GL(2,\mathbb{R})$-Higgs bundle, then the orthogonal structure on $E$ gives an isomorphism $(E , \Phi) \simeq (E^* , \Phi^t)$. In turn this implies an isomorphism $M_0\simeq M_0^*$ of the associated line bundle, since $M_0$ is the line bundle associated to $(E^*,\Phi^t)$.\\

Let $(E,\Phi)$ be an $SL(2,\mathbb{C})$-Higgs bundle associated to the line bundle $M_0$. Thus $\sigma^*(M_0) \simeq M_0^*$. If $M_0^2 = \mathcal{O}$ then we have $\sigma^*(M_0) \simeq M_0$. Let $\tilde{\sigma} : M_0 \to M_0$ be an involution covering $\sigma$. Then $\tilde{\sigma}$ induces an involution $\tilde{\sigma}$ on $E$. Let $E = L_+ \oplus L_-$ be the decomposition of $E$ into $+1$ and $-1$ eigenspaces of $\tilde{\sigma}$. It is easy to see that $L_+,L_-$ are line bundles on $\Sigma$. Moreover, $M_0^2 = \mathcal{O}$, so as in the $GL(2,\mathbb{R})$ case this determines an orthogonal structure on $E$. Now $\Phi$ is symmetric but $\sigma^*(\lambda) = -\lambda$, so it must be that $\tilde{\sigma}$ is skew-symmetric. Hence $L_+,L_-$ are isotropic subbundles and the orthogonal structure on $E$ gives a dual pairing. We set $N = L_+$, then $L_- = N^*$ and $E = N \oplus N^*$. Further, since $\sigma^*(\lambda) = -\lambda$, it follows that $\tilde{\sigma}$ and $\Phi$ anti-commute so that $\Phi$ has the form $\Phi = \left[ \begin{matrix} 0 & \beta \\ \gamma & 0 \end{matrix} \right]$, for sections $\beta,\gamma$ of $N^2L,N^{-2}L$. So the condition $M_0^2 = \mathcal{O}$ together with a choice of involutive lift $\tilde{\sigma}$ of $\sigma$ determines an $SL(2,\mathbb{R})$-Higgs bundle $(N,\beta , \gamma)$. Conversely, given $(N,\beta,\gamma)$ we construct the $SL(2,\mathbb{C})$-Higgs bundle $(E,\Phi)$. Let $M_0$ be the associated line bundle. The orthogonal structure on $E$ gives $M_0^2 = \mathcal{O}$. Let $\tilde{\sigma}$ be the involution on $E = N \oplus N^*$ which acts as $1$ on $N$ and $-1$ on $N^*$. Then $\tilde{\sigma}$ determines an involutive lift of $\tilde{\sigma} : M_0 \to M_0$ of $\sigma$ and hence a pair $(M_0,\tilde{\sigma})$.\\

This completes the proof in the $GL(2,\mathbb{R})$ and $SL(2,\mathbb{R})$ cases. The $PGL(2,\mathbb{R})$ and $PSL(2,\mathbb{R})$ cases are very similar so we omit the details.
\end{proof}

\begin{remark}\label{remfibres}
Choosing splittings of the local systems as in Sections \ref{secmonoact}, \ref{secmonoact2}, \ref{secmonoact3}, we can identify the regular fibres of the various moduli spaces of real Higgs bundles as the following monodromy representations:
\begin{itemize}
\item[(1)]{For ${}^{\mathbb{R}} \mathcal{M}^0(L)$, the representation is $\Lambda_S[2] = \Lambda_\Sigma[2] \oplus (\mathbb{Z}_2 B)^{\rm ev}/(b_o) \oplus \Lambda_\Sigma[2]$.}
\item[(2)]{For ${}^{\mathbb{R}} \check{\mathcal{M}}(L)$, the representation is $\widetilde{\Lambda}_P[2] = \Lambda_\Sigma[2] \oplus (\mathbb{Z}_2 B)^{\rm ev}$.}
\item[(3)]{For ${}^{\mathbb{R}} \hat{\mathcal{M}}(d,L)$, the representation is $\Lambda_S^d[2]/\Lambda_\Sigma[2] = (\mathbb{Z}_2 B)^d/(b_o) \oplus \Lambda_\Sigma[2]$.}
\item[(4)]{For ${}^{\mathbb{R}} \widetilde{\mathcal{M}}(d,L)$, the representation is $\widetilde{\Lambda}_P^d[2]/\Lambda_\Sigma[2] = (\mathbb{Z}_2 B)^d$.}
\end{itemize}
\end{remark}

%%%%%%%%%%%%%%%%%%%%%%%%%%%%%%%%%%%%%%%%%%%%%
\subsection{Topological invariants}\label{secti}
%%%%%%%%%%%%%%%%%%%%%%%%%%%%%%%%%%%%%%%%%%%%%

In this section we continue to assume that the degree of $L$ is even.

\begin{definition}\label{deftopinv}
We define the following topological invariants associated to real Higgs bundles:
\begin{itemize}
\item[(1)]{For a $GL(2,\mathbb{R})$-Higgs bundle $(E,\Phi)$, the orthogonal structure gives $E$ the structure group $O(2,\mathbb{C})$. Reducing to the maximal compact $O(2)$ defines a real rank $2$ orthogonal vector bundle $V$ such that $E = V \otimes \mathbb{C}$. The Stiefel-Whitney classes of $V$ defined invariants $w_1 = w_1(V) \in H^1( \Sigma , \mathbb{Z}_2)$ and $w_2 = w_2(V) \in H^2(\Sigma , \mathbb{Z}_2) \simeq \mathbb{Z}_2$.}
\item[(2)]{For an $SL(2,\mathbb{R})$-Higgs bundle $(N,\beta,\gamma)$, we have an integer-valued invariant $\delta = deg(N)$.}
\item[(3)]{For a $PGL(2,\mathbb{R})$-Higgs bundle represented by $(E,\Phi , A)$ we have two topological invariants $\hat{w}_1,\hat{w}_2$ defined as follows. First note that the line bundle $U = \wedge^2 E \otimes A^*$ is independent of the choice of representative $(E,\Phi , A)$ and that the pairing $E \otimes E \to A$ implies that $U^2 = \mathcal{O}$. Thus $U$ is a well-defined line bundle of order $2$ and defines a class $\hat{w}_1 \in H^1(\Sigma , \mathbb{Z}_2)$. We define $\hat{w}_2 \in \mathbb{Z}_2$ to be the mod $2$ degree of $E$. This is also independent of the choice of representative $(E,\Phi , A)$.}
\item[(4)]{For a $PSL(2,\mathbb{R})$-Higgs bundle represented by $(N_1,N_2,\beta,\gamma)$, we define an integer invariant $\check{\delta} = deg(N_1) - deg(N_2)$. Clearly $\check{\delta}$ is independent of the choice of representative $(N_1,N_2,\beta,\gamma)$.}
\end{itemize}
\end{definition}

The characteristic classes $w_1,w_2$ in the $GL(2,\mathbb{R})$ have a $KO$-theoretical interpretation, as we recall from \cite{hit3}. Suppose that $E$ is a rank $m$ holomorphic vector bundle with orthogonal structure. Choosing a reduction to the maximal compact subgroup $O(m) \subset O(m,\mathbb{C})$ determines a real orthogonal bundle $V$ such that $E = V \otimes \mathbb{C}$. The isomorphism class of $V$ as a real vector bundle is independent of the choice of reduction, so gives a well-defined class $[V] \in KO(\Sigma)$, the real $K$-theory of $\Sigma$. We will abuse notation and write $[E] \in KO(\Sigma)$ for this class.\\

Recall that $\pi : S \to \Sigma$ is the spectral curve corresponding to $a_0 \in \mathcal{A}_{\rm reg}^0(L)$. Let $K_S$ be the canonical bundle of $S$. Suppose that $U$ is a square root of $K_S \otimes \pi^*(K^*)$ on $S$. This can be thought of as a relative spin structure and hence a relative $KO$-orientation for the map $\pi$. It is then possible to define the push-forward map $\pi_! : KO(S) \to KO(\Sigma)$. The map $\pi_!$ has a holomorphic interpretation which is as follows: suppose that $F$ is a holomorphic vector bundle on $S$ with orthogonal structure, so $F$ defines a class $[F] \in KO(S)$. Set $E = \pi_*( F \otimes U)$. As explained in \cite{hit3}, relative duality determines a natural orthogonal structure on $E$, hence we obtain a class $[E] \in KO(\Sigma)$ and we have $[E] = \pi_! [F]$.\\

Suppose $M_0$ is a holomorphic line bundle $S$ of order $2$. Then $M_0$ can be thought of as a rank $1$ holomorphic vector bundle with orthogonal structure. If $(E,\Phi)$ is the associated $GL(2,\mathbb{R})$-Higgs bundle then $E = \pi_*( M_0 \otimes \pi^*(L^{1/2} ) )$. Recall from Section \ref{secschf} that $K_S \pi^*(K^*) = \pi^*(L)$ and hence $\pi^*(L^{1/2})$ gives a relative $KO$-orientation. By the discussion above, we have $[E] = \pi_! [M_0]$. We now consider how the Stiefel-Whitney classes of $E$ are related to the line bundle $M_0$. The case of $w_1(E)$ is straightforward, since as elements of $Jac(\Sigma)[2]$, we have:
\begin{equation*}
w_1(E) = det(E) = Nm(M_0).
\end{equation*}
For $w_2(E)$, we make use of the relation $[E] = \pi_! [M_0]$. Choose a spin structure $K^{1/2}$ on $\Sigma$. Then $\pi^*( K^{1/2} L^{1/2} )$ is a spin structure on $S$ and our choices are compatible with the relative spin structure $\pi^*(L^{1/2})$. The spin structures on $\Sigma$ and $S$ define index maps $\varphi_\Sigma : KO(\Sigma) \to KO^{-2}(pt) = \mathbb{Z}_2$ and $\varphi_S : KO(S) \to KO^{-2}(pt) = \mathbb{Z}_2$. Since we have chosen our spin structures compatibly, we get a commutative diagram:
\begin{equation*}\xymatrix{
KO(S) \ar[drr]^-{\varphi_S} \ar[d]^-{\pi_!} & & \\
KO(\Sigma) \ar[rr]^-{\varphi_\Sigma} & & \mathbb{Z}_2
}
\end{equation*}
We recall from \cite{ati} that the index maps $\varphi_\Sigma,\varphi_S$ have the following holomorphic interpretation. Let $E$ be a holomorphic vector bundle on $\Sigma$ with orthogonal structure. Then $\varphi_\Sigma( [E])$ is the mod $2$ index:
\begin{equation*}
\varphi_\Sigma( [E] ) = dim \left( H^0( \Sigma , E \otimes K^{1/2} ) \right) \; ( {\rm mod} \; 2 )
\end{equation*}
and similarly for $\varphi_S$. As shown in \cite{ati}, the restriction of $\varphi_\Sigma $ to the space $H^1(\Sigma , \mathbb{Z}_2)$ of holomorphic line bundles with orthogonal structure is a quadratic refinement of the Weil pairing $\langle \; , \; \rangle$, that is:
\begin{equation*}
\varphi_\Sigma( N_1 \otimes N_2) = \varphi_\Sigma(N_1) + \varphi_\Sigma(N_2) + \langle N_1 , N_2 \rangle + \varphi_\Sigma(0).
\end{equation*}
Similarly $\varphi_S$ gives a quadratic refinement of the Weil pairing on $H^1(S,\mathbb{Z}_2)$.

\begin{lemma}[\cite{hit3}]\label{lemw2}
Let $E$ be a rank $m$ vector bundle on $\Sigma$ with orthogonal structure. Then
\begin{equation}\label{equw2}
w_2(E) = \varphi_\Sigma([E]) + \varphi_\Sigma( [det(E)] ) +(m-1)\varphi_\Sigma(0).
\end{equation}
\end{lemma}
\begin{proof}
For any such vector bundle $E$ we wish to show that $\delta(E) = 0$, where
\begin{equation*}
\delta(E) = w_2(E) + \varphi_\Sigma([E]) + \varphi_\Sigma( [det(E)] ) +(m-1)\varphi_\Sigma(0).
\end{equation*}
Using the fact that $w_2(E \oplus F) = w_2(E) + w_2(F) + \langle det(E) , det(F) \rangle$ and that fact that $\varphi_\Sigma$ is a quadratic refinement of $\langle \; , \; \rangle$, we see that $\delta(E \oplus F) = \delta(E) + \delta(F)$. Thus $\delta$ descends to a homomorphism $\delta : KO(\Sigma) \to \mathbb{Z}_2$.\\

The additive group of $KO(\Sigma)$ is generated by line bundles and bundles of the form $E = N + N^*$, where $N$ is a complex line bundle and the orthogonal structure on $E$ is the dual pairing. To prove Equation (\ref{equw2}), we just need to check that $\delta(E) = 0$ on these generators. If $E$ is a line bundle then it is trivial to see that $\delta(E) = 0$. Now suppose that $E = N \oplus N^*$, where $n = deg(N)$. Then $w_2(E) = n$ and
\begin{equation*}
\begin{aligned}
\varphi_\Sigma( [E] ) & = \varphi_\Sigma( [N] ) + \varphi_\Sigma( [N^*] ) \\
&=  dim( H^0(\Sigma , N \otimes K^{1/2}) ) + dim( H^0(\Sigma , N^* \otimes K^{1/2}) ) \; ({\rm mod} \; 2) \\
&= n,
\end{aligned}
\end{equation*}
where we have used Riemann-Roch in the last step. Lastly, since $det(E) = N \otimes N^* = \mathcal{O}$, we see that $\delta(N\oplus N^*) = 0$ as required.
\end{proof}

\begin{lemma}\label{lemphi0}
We have $\varphi_S(0) = l/2 \; ( {\rm mod} \; 2)$.
\end{lemma}
\begin{proof}
Since $\pi_* \mathcal{O}_S = \mathcal{O}_\Sigma \oplus L^*$ \cite{bnr}, we have $\pi_!( \mathcal{O}_S) = L^{1/2} \oplus L^{-1/2}$. Then
\begin{equation*}
\begin{aligned}
\varphi_S(0) &= \varphi_\Sigma(  \pi_! \mathcal{O}_S ) \\
&= \varphi_\Sigma( L^{1/2} \oplus L^{-1/2} ) \\
&= \frac{l}{2} \; ({\rm mod} \; 2).
\end{aligned}
\end{equation*}
\end{proof}

\begin{proposition}\label{propqref}
Suppose that square roots $K^{1/2}$ and $L^{1/2}$ have been chosen. There exists a splitting of (\ref{equshortexactprym2}) such that the statement of Proposition \ref{propsplit1} holds and in addition we have:
\begin{equation*}
\varphi_S(x) = q(x) + \varphi_S(0),
\end{equation*}
for all $x \in \Lambda_S[2] \simeq H^1(S , \mathbb{Z}_2)$, where $q$ is the quadratic refinement on $\Lambda_S[2]$ introduced in Section \ref{secmonoact}.
\end{proposition}
\begin{proof}
First suppose that $c \in \Lambda_\Sigma[2]$ and let $C$ be the corresponding line bundle of order $2$. Then $\pi_! ( \pi^*(C) ) = CL^{1/2} \oplus CL^{-1/2}$. Hence $\varphi_S( c,0,0 ) = \varphi_\Sigma( CL^{1/2} \oplus CL^{-1/2} ) = l/2 = \varphi_S( 0,0,0)$.\\

Choose any splittings satisfying Proposition \ref{propsplit1}, so that $\Lambda_S[2] = \Lambda_\Sigma[2] \oplus W \oplus \Lambda_\Sigma[2]$ with the Weil pairing given by Equation (\ref{equintpairing}). Let $\psi : \Lambda_S[2] \to \mathbb{Z}_2$ be given by $\psi(x) = \varphi_S(x) + \varphi_S(0)$. Then $\psi$ is also a quadratic refinement of the Weil paring and clearly satisfies $\psi(0) = 0$. The above calculation also shows that $\psi$ vanishes on $\pi^*(\Lambda_\Sigma[2])$. Next, since $0 \oplus 0 \oplus \Lambda_\Sigma[2]$ is an isotropic subspace, we see that $\psi( (0,0,c))$ is a linear function on $\Lambda_\Sigma[2]$. We will eliminate this linear function using a change of splitting.\\

Let $\iota : \Lambda_\Sigma[2] \to \Lambda_S[2]$ be the given splitting of $(\ref{equshortexactprym2})$. Let $F : \Lambda_\Sigma[2] \to \Lambda_\Sigma[2]$ be a symmetric endomorphism, i.e. $\langle Fx,y\rangle = \langle x , Fy \rangle$. Then we consider a new splitting $c \mapsto \iota(c) + \pi^*(Fc)$. Since $F$ is symmetric we have that the Weil pairing still has the form (\ref{equintpairing}) in the new splitting. However, 
\begin{equation*}
\begin{aligned}
\psi( \iota(c) + \pi^*(Fc) ) &= \psi( (0,0,c) + (Fc,0,0) ) \\
&= \psi(0,0,c) + \psi(Fc,0,0) + \langle Fc,c \rangle \\
&= \psi(0,0,c) + \langle Fc , c \rangle.
\end{aligned}
\end{equation*}
One can easily show that given any linear function $\alpha : \Lambda_\Sigma[2] \to \mathbb{Z}_2$, there is a symmetric endomorphism $F$ such that $\alpha(c) = \langle Fc , c \rangle$. Applying this to $\alpha(c) = \phi(0,0,c)$, we see that we can choose $F$ and hence a splitting, such that $\psi$ vanishes on the image of the splitting.\\

So far we have shown that $\psi(a,0,0) = \psi(0,0,c) = 0$ for all $a,c \in \Lambda_\Sigma[2]$. Then since $\psi$ is a quadratic refinement of the Weil pairing, we have
\begin{equation*}
\psi(a,b,c) = \langle a , c \rangle + \psi(0,b,0).
\end{equation*}
To complete the proposition it remains to show that $\psi(0,b,0) = q_W(b)$ for all $b \in W$. However, we know that $\varphi_S$ and hence $\psi$ are monodromy invariant functions, because the square roots $L^{1/2}$, $K^{1/2}$ are also monodromy invariant. In particular $\psi( 0 , b_{ij} , 0 )$ takes the same value for all $1 \le i < j \le 2l$. However,
\begin{equation*}
\begin{aligned}
\psi( b_{13} ) &= \psi( b_{12} + b_{23} )\\
&= \psi( b_{12} ) + \psi( b_{23} ) + (( b_{12} , b_{23} )) \\
&= \psi( b_{12} ) + \psi( b_{23} ) + 1.
\end{aligned}
\end{equation*}
Therefore we must have $\psi( b_{ij} ) = 1 = q_W( b_{ij} )$ for all $i < j$. Then using the quadratic property we see that $\psi( 0 , b ,0 ) = q_W(b)$ for all $b \in W$.
\end{proof}

\begin{proposition}\label{propspecinv}
Identify the regular fibres of the moduli spaces ${}^{\mathbb{R}} \mathcal{M}^0(L)$, ${}^{\mathbb{R}} \check{\mathcal{M}}(L)$, ${}^{\mathbb{R}} \hat{\mathcal{M}}(d,L)$, ${}^{\mathbb{R}} \widetilde{\mathcal{M}}(d,L)$, with the monodromy representations $\Lambda_S[2]$, $\widetilde{\Lambda}_P[2]$, $\Lambda_S^d[2]/\Lambda_\Sigma[2]$, $\widetilde{\Lambda}_P^d[2]/\Lambda_\Sigma[2]$ as in Remark \ref{remfibres}. Then the topological invariants given in Definition \ref{deftopinv} are as follows:
\begin{itemize}
\item[(1)]{For $GL(2,\mathbb{R})$, we suppose that we have chosen splittings satisfying Proposition \ref{propqref}. Then we have:
\begin{equation*}
\begin{aligned}
w_1(a,b,c) &= c, \\
w_2(a,b,c) &= \varphi_\Sigma(c) + \varphi_S(a,b,c) + \varphi_\Sigma(0) = \varphi_\Sigma(c) + \frac{l}{2} + q(a,b,c) + \varphi_\Sigma(0).
\end{aligned}
\end{equation*}
}
\item[(2)]{For $SL(2,\mathbb{R})$, we have:
\begin{equation*}
\delta(a,b) = \frac{l-m}{2},
\end{equation*}
where $b = b_{i_1} + b_{i_2} + \dots + b_{i_m}$, with $i_1,i_2,\dots,i_m$ distinct.}
\item[(3)]{For $PGL(2,\mathbb{R})$, we have: 
\begin{equation*}
\begin{aligned}
\hat{w}_1(b,c) &= c, \\
\hat{w}_2(b,c) &= ((b,b_o)) = d.
\end{aligned}
\end{equation*}
}
\item[(4)]{For $PSL(2,\mathbb{R})$, we have:
\begin{equation*}
\check{\delta}(b) = (l-m),
\end{equation*}
where $b = b_{i_1} + b_{i_2} + \dots + b_{i_m}$, with $i_1,i_2,\dots,i_m$ distinct.}
\end{itemize}
\end{proposition}
\begin{proof}
For $GL(2,\mathbb{R})$, let $(E,\Phi)$ correspond to $M_0 = (a,b,c) \in \Lambda_S[2]$. Then $w_1(E) = Nm(M_0) = \pi_*(a,b,c) = c$. Next, using Lemma \ref{lemw2}, Lemma \ref{lemphi0} and Proposition \ref{propqref}, we have:
\begin{equation*}
\begin{aligned}
w_2(E) &= \varphi_\Sigma([E]) + \varphi_\Sigma( [det(E)] ) + \varphi_\Sigma(0) \\
&= \varphi_S( [M_0]) + \varphi_\Sigma( c ) + \varphi_\Sigma(0) \\
&= \varphi_S( a,b,c) + \varphi_\Sigma( c ) + \varphi_\Sigma(0) \\
&= q(a,b,c) + \frac{l}{2} + \varphi_\Sigma( c ) + \varphi_\Sigma(0).
\end{aligned}
\end{equation*}
For $SL(2,\mathbb{R})$, let $(M_0 , \tilde{\sigma})$ be the line bundle and lift of $\sigma$ and let $(a,b)$ be the corresponding point in $\widetilde{\Lambda}_P[2]$. The underlying $GL(2,\mathbb{R})$-Higgs bundle is $(E,\Phi)$ where $E = \pi_*( M_0 \otimes \pi^*(L^{1/2}))$. Then $\tilde{\sigma}$ determines an involution on $E$ and we obtain a decomposition $E = L \oplus L^*$, where $L$ is the $+1$-eigenspace and $L^*$ is the $-1$-eigenspace. If $b = b_{i_1} + \dots + b_{i_m}$ with $i_1,\dots , i_m$ distinct, then from the discussion in Section \ref{secmonoact}, it follows that $\tilde{\sigma}$ acts as $-1$ over $m$ ramification points and acts as $+1$ over the remaining $p = 2l-m$ points. As shown in \cite{sch0}, the Lefschetz index theorem \cite{atbo} gives $2\delta(a,b) = (p-m)/2 = l-m$, hence $\delta(a,b) = (l-m)/2$.\\

For a $PGL(2,\mathbb{R})$-Higgs bundle represented by $(E,\Phi,A)$, we have defined $d = deg(E) \; ({\rm mod} \; 2)$. We then clearly have $\hat{w}_2(b,c) = ((b,b_o)) = d$. For the invariant $\hat{w}_1$, we consider separately the cases $d=0$ and $d=1$. When $d=0$ our $PGL(2,\mathbb{R})$-Higgs bundle is represented by a $GL(2,\mathbb{R})$ corresponding to $(a,b,c) \in \Lambda_S[2]$ and in this case it is clear that $\hat{w}_1 = c$. When $d=1$ we can find a representative of the form $(E,\Phi,A)$, where $A = \mathcal{O}(b_1)$. Let $M_0\in Jac_1(S)$ be the corresponding line bundle on $S$, so $E = \pi_*( M_0 \otimes \pi^*( L^{1/2}))$. Let $N = \mathcal{O}(u_1)$, so that $M_0$ can be written in the form $M_0 = (a,b,c) + N$, where $(a,b,c) \in \Lambda_S[2]$. Then $\wedge^2 E = det(E) = Nm( a,b,c) + Nm(N) = c  + \mathcal{O}(b_1)$, hence $\wedge^2 E \otimes A^* = c \in H^1(\Sigma , \mathbb{Z}_2)$. So again $\hat{w}_1 = c$.\\

For $PSL(2,\mathbb{R})$, we again use the Lefschetz index theorem as we did in the $SL(2,\mathbb{R})$ case to obtain $\check{\delta} = deg(N_1) - deg(N_2) = (p-m)/2 = l-m$.
\end{proof}
\begin{remark}\label{remineq}
From Proposition \ref{propspecinv}, we have inequalities
\begin{equation*}
-\frac{l}{2} \le \delta  \le \frac{l}{2}, \; \; \; \; \; \; \; \; -l \le  \check{\delta}  \le l.
\end{equation*}
\end{remark}

%%%%%%%%%%%%%%%%%%%%%%%%%%%%%%%%%%%%%%%%%%%%
%%%%%%%%%%%%%%%%%%%%%%%%%%%%%%%%%%%%%%%%%%%%
% SECTION
%%%%%%%%%%%%%%%%%%%%%%%%%%%%%%%%%%%%%%%%%%%%
%%%%%%%%%%%%%%%%%%%%%%%%%%%%%%%%%%%%%%%%%%%%
\section{Components of real character varieties}\label{seccrcv}

%%%%%%%%%%%%%%%%%%%%%%%%%%%%%%%%%%%
\subsection{Real character varieties}\label{secrcv}
%%%%%%%%%%%%%%%%%%%%%%%%%%%%%%%%%%%

Let $G$ be a real reductive Lie group. A representation $\theta : \pi_1(\Sigma) \to G$ is said to be {\em reductive} if the representation of $\pi_1(\Sigma)$ on the Lie algebra of $G$ obtained by composing $\theta$ with the adjoint representation decomposes into a sum of irreducible representations. Let $Hom^{\rm red}(\pi_1(\Sigma) , G)$ be the space of reductive representations given the compact-open topology. The group $G$ acts on $Hom^{\rm red}(\pi_1(\Sigma) , G)$ by conjugation and it is known that quotient
\begin{equation*}
Rep(G) = Hom^{\rm red}(\pi_1(\Sigma) , G)/G
\end{equation*}
is Hausdorff \cite{ric}. We call $Rep(G)$ the {\em character variety} of reductive representations of $\pi_1(\Sigma)$ in $G$. It can furthermore be shown that $Rep(G)$ has the structure of a real analytic variety which is algebraic if $G$ is algebraic \cite{goldm}.\\

Let $\widetilde{\Sigma}$ be the universal cover of $\Sigma$. Given a representation $\theta : \pi_1(\Sigma) \to G$, we obtain a principal $G$-bundle $P_\theta = \widetilde{\Sigma} \times_\theta G$. In this way we can associate topological invariants to $\theta$ by taking various topological invariants of the associated bundle $P_\theta$. When $\theta$ is a representation into $PSL(2,\mathbb{C})$, we obtain a class $d \in H^2(\Sigma , \mathbb{Z}_2) \simeq \mathbb{Z}_2$ which is the obstruction to lifting $P_\theta$ to a principal $SL(2,\mathbb{C})$-bundle. We write $Rep_d( PSL(2,\mathbb{C}) )$ for the subvariety of $Rep( PSL(2,\mathbb{C}))$ consisting of those representations with fixed value of the invariant $d$. Similarly, we obtain $Rep_d( PGL(2,\mathbb{R}) )$ (resp. $Rep_d( PSL(2,\mathbb{R} ) ) )$ where $d \in \mathbb{Z}_2$ is the obstruction to lifting $P_\theta$ to $GL(2,\mathbb{R})$ (resp. $SL(2,\mathbb{R})$).\\

The non-abelian Hodge theory established by Hitchin \cite{hit1}, Simpson \cite{sim1,sim2}, Donaldson \cite{don} and Corlette \cite{cor} gives a homeomorphism between the moduli space of polystable $G$-Higgs bundles (where $L = K$ is the canonical bundle) and the character variety $Rep(G)$, when $G$ is a complex semisimple Lie group. There is a similar correspondence in the complex reductive case. The particular cases of relevance to us are:

\begin{proposition}
There exists homeomorphisms:
\begin{itemize}
\item[(1)]{$\mathcal{M}(0,K) \simeq Rep( GL(2,\mathbb{C}) )$}
\item[(2)]{$\check{\mathcal{M}}(\mathcal{O},K) \simeq Rep( SL(2,\mathbb{C}) )$}
\item[(3)]{$\hat{\mathcal{M}}(d,K) \simeq Rep_d( PSL(2,\mathbb{C}) )$}
\end{itemize}
\end{proposition}

The non-abelian Hodge correspondence has also been extended to real reductive groups \cite{bgm}, \cite{ggm0}. In particular this gives the following:

\begin{proposition}\label{proprealnah}
There exists homeomorphisms:
\begin{itemize}
\item[(1)]{${}^{\mathbb{R}} \mathcal{M}(K) \simeq Rep( GL(2,\mathbb{R}) )$}
\item[(2)]{${}^{\mathbb{R}} \check{\mathcal{M}}(K) \simeq Rep( SL(2,\mathbb{R}) )$}
\item[(3)]{${}^{\mathbb{R}} \hat{\mathcal{M}}(d,K) \simeq Rep_d( PGL(2,\mathbb{R}))$}
\item[(4)]{${}^{\mathbb{R}} \widetilde{\mathcal{M}}(d,K) \simeq Rep_d( PSL(2,\mathbb{R}))$}
\end{itemize}
\end{proposition}

%%%%%%%%%%%%%%%%%%%%%%%%%%%%%%%%%%%%%
\subsection{Connected components of real character varieties}
%%%%%%%%%%%%%%%%%%%%%%%%%%%%%%%%%%%%%

We continue to assume that $L = K$ or $deg(L) > 2g-2$. We will also assume that $deg(L) = l$ is even.

\begin{proposition}\label{propconnreg}
Let $G = GL(2,\mathbb{R}), SL(2,\mathbb{R}), PGL(2,\mathbb{R})$ or $PSL(2,\mathbb{R})$. Then every connected component of the corresponding moduli spaces ${}^{\mathbb{R}} \mathcal{M}(L), {}^{\mathbb{R}} \check{\mathcal{M}}(L), {}^{\mathbb{R}} \hat{\mathcal{M}}(d,L)$ and ${}^{\mathbb{R}} \widetilde{\mathcal{M}}(d,L)$ meets the regular locus.
\end{proposition}
\begin{proof}
In the case of a polystable twisted $SL(2,\mathbb{R})$-Higgs bundle, the result is a straightforward generalisation of \cite[Proposition 10.2]{sch1}. Next we consider a polystable $GL(2,\mathbb{R})$-Higgs bundle $(E,\Phi)$. Note that it is sufficient to consider the case that $\Phi$ is trace-free. We may assume that $w_1(E) \neq 0$, since otherwise $(E,\Phi)$ comes from a polystable $SL(2,\mathbb{R})$-Higgs bundle and the previous argument applies. Let $p : \Sigma' \to \Sigma$ be the double cover associated to the class $w_1(E) \in H^1(\Sigma , \mathbb{Z}_2)$ and let $\iota : \Sigma' \to \Sigma'$ be the involution swapping the two sheets of the covering $\Sigma' \to \Sigma$. Note that $\Sigma'$ is connected as $w_1(E) \neq 0$. Then $(p^*(E) , p^*(\Phi) )$ is an $SL(2,\mathbb{R})$-Higgs bundle on $\Sigma'$ in the sense that there exists a line bundle $N \to \Sigma'$ for which $p^*(E) = N \oplus N^*$ with orthogonal structure the dual pairing and $p^*(\Phi) = \left[ \begin{matrix} 0 & \beta \\ \gamma & 0 \end{matrix} \right]$. We also have $f^*(N) = N^*$ and $f^*(\beta) = \gamma$. Note that since $f$ is orientation preserving, the condition $f^*(N) = N^*$ implies that $deg(N) = 0$. We also have that $\beta \gamma = \beta f^*(\beta) = p^*( a )$ for some $a \in H^0( \Sigma , L^2)$. Let $a(t)$ be a path joining $a = a(0)$ to an element $a(1) \in H^0(\Sigma , L^2)^{\rm simp}$. It is easy to see that we can lift this to a path $\beta(t) \in H^0(\Sigma' , N^2 \pi^*(L) )$ such that $\beta = \beta(0)$ and $\beta(t) f^*(\beta(t) ) = p^*( a(t) )$. Setting $\gamma(t) = f^*(\beta(t))$ we obtain a path $(E , \Phi(t) )$ joining $(E,\Phi)$ to a point in the regular locus.\\

The $PGL(2,\mathbb{R})$ and $PSL(2,\mathbb{R})$ cases are proved in a manner similar to the $GL(2,\mathbb{R})$ and $SL(2,\mathbb{R})$ cases.
\end{proof}

\begin{remark}
Proposition \ref{propconnreg} implies that the inequalities of Remark \ref{remineq} are valid for all points of the corresponding moduli spaces ${}^{\mathbb{R}} \check{\mathcal{M}}(L)$ and ${}^{\mathbb{R}} \widetilde{\mathcal{M}}(d,L)$ (for either value of $d$). This generalises the Milnor-Wood inequality $-(g-1) \le \delta \le (g-1)$ for $Rep( SL(2,\mathbb{R})) \simeq {}^{\mathbb{R}} \check{\mathcal{M}}(K)$.
\end{remark}

Next, we define a notion of maximal Higgs bundle which corresponds to representation of maximal Toledo invariant:
\begin{definition}
We define {\em maximal} real Higgs bundles as follows:
\begin{itemize}
\item[(1)]{An $SL(2,\mathbb{R})$-Higgs bundle $(N,\beta,\gamma)$ is said to be {\em maximal} if it is polystable and $\delta = deg(N) = \pm l/2$.}
\item[(2)]{A $PSL(2,\mathbb{R})$-Higgs bundle $(N_1,N_2,\beta,\gamma)$ is said to be {\em maximal} if it is polystable and $\check{\delta} = deg(N_1) - deg(N_2) = \pm l$.}
\item[(3)]{We say that a trace-free $GL(2,\mathbb{R})$-Higgs bundle is {\em maximal} if it is the $GL(2,\mathbb{R})$-Higgs bundle associated to a maximal $SL(2,\mathbb{R})$-Higgs bundle. More generally, we say that a $GL(2,\mathbb{R})$-Higgs bundle $(E,\Phi)$ is maximal if the associated trace-free $GL(2,\mathbb{R})$-Higgs bundle $(E , \Phi - \frac{tr(\Phi)}{2} Id )$ is maximal.}
\item[(4)]{We say that a $PGL(2,\mathbb{R})$-Higgs bundle is {\em maximal} if it is the $PGL(2,\mathbb{R})$-Higgs bundle associated to a maximal $PSL(2,\mathbb{R})$-Higgs bundle.}
\end{itemize}
\end{definition}

\begin{proposition}
We have the following classification of maximal Higgs bundles:
\begin{enumerate}
\item{Up to isomorphism, maximal $SL(2,\mathbb{R})$-Higgs bundles are of the form $(N , \beta , 1)$ or $(N^*,1,\beta)$, where $N^2 = L$ and $\beta$ is a holomorphic section of $L^2$.}
\item{Up to isomorphism, maximal $GL(2,\mathbb{R})$-Higgs bundles are of the form $E = N \oplus N^*$, $\Phi = \left[ \begin{matrix} \alpha & \beta \\ 1 & \alpha \end{matrix} \right]$, where $N^2 = L$, $\alpha$ is a holomorphic section of $L$ and $\beta$ is a holomorphic section of $L^2$.}
\item{Up to isomorphism, maximal $PSL(2,\mathbb{R})$-Higgs bundles are of the form $(L,1,\beta,1)$ or \linebreak $(1,L,1,\beta)$, where $\beta$ is a holomorphic section of $L^2$.}
\item{Up to isomorphism, maximal $PGL(2,\mathbb{R})$-Higgs bundles are of the form $E = L \oplus 1$, $\Phi = \left[ \begin{matrix} 0 & \beta \\ 1 & 0 \end{matrix} \right]$, where $\beta$ is a holomorphic section of $L^2$.}
\end{enumerate}
\end{proposition}
\begin{proof}
We give the proof for the $SL(2,\mathbb{R})$ case, the other cases being similar. If $(N,\beta,\gamma)$ is maximal then $deg(N) = \pm l/2$. If $deg(N) = l/2$ then $\gamma$ is a section of $N^{-2}L$ which has degree $0$ and is non-vanishing by polystability. Thus $N^2 \simeq L$ and we can choose the isomorphism of $N^2$ and $L$ so that $\gamma = 1$. Similarly if $deg(N) = -l/2$ then $N^2 = L^{-1}$ and we can take $\beta = 1$.
\end{proof}

\begin{corollary}\label{cormax1}
The number of maximal connected components are as follows:
\begin{enumerate}
\item{$2^{2g}$ for $GL(2,\mathbb{R})$}
\item{$2^{2g+1}$ for $SL(2,\mathbb{R})$}
\item{$1$ for $PGL(2,\mathbb{R})$}
\item{$2$ for $PSL(2,\mathbb{R})$}
\end{enumerate}
\end{corollary}

\begin{theorem}\label{thmcomponents}
The number of connected components of the $L$-twisted real Higgs bundle moduli spaces are as follows:
\begin{enumerate}
\item{$3.2^{2g} + (l-4)/2$ for ${}^{\mathbb{R}} \mathcal{M}(L)$ (and ${}^{\mathbb{R}} \mathcal{M}^0(L)$)}
\item{$2.2^{2g}+(l-1)$ for ${}^{\mathbb{R}} \check{\mathcal{M}}(L)$}
\item{$2^{2g}+l/2$ for ${}^{\mathbb{R}} \hat{\mathcal{M}}(0,L)$ and $2^{2g}+l/2-1$ for ${}^{\mathbb{R}} \hat{\mathcal{M}}(1,L)$}
\item{$l+1$ for ${}^{\mathbb{R}} \widetilde{\mathcal{M}}(0,L)$ and $l$ for ${}^{\mathbb{R}} \widetilde{\mathcal{M}}(1,L)$}
\end{enumerate}
\end{theorem}
\begin{proof}
Our strategy for counting components is as follows: Proposition \ref{propconnreg} ensures that every component meets the regular locus and thus every component meets any fixed choice of non-singular fibre. Next we determine the orbits of the monodromy action on the fibre. We say that an orbit is maximal if the corresponding Higgs bundles are maximal and we say an orbit is non-maximal otherwise. By inspection, we will find that any two distinct non-maximal orbits will have different topological invariants and thus correspond to distinct connected components of the moduli space. It follows that the number of connected components is the number of non-maximal orbits plus the number of maximal components (and this is just the total number of orbits of the monodromy).\\

{\bf Case (1): $GL(2,\mathbb{R})$.} The real points of a fibre is given by $\Lambda_S[2] = \Lambda_\Sigma[2] \oplus (\mathbb{Z}_2 B)^{\rm ev}/(b_o) \oplus \Lambda_\Sigma[2]$. The maximal orbits are those of the form $(a,0,0)$, $a \in \Lambda_\Sigma[2]$. Let $(a,b,c) \in \Lambda_S[2]$. If $c \neq 0$, we can use the monodromy action to eliminate $b$ leaving $(a,0,c)$. We claim that for each fixed $c \neq 0$ there are two orbits corresponding to whether $q(a,0,c) = \langle a , c \rangle $ is $0$ or $1$. Note that we have a monodromy action $(a,0,c) \mapsto (a+Fc,0,c)$ , where $F$ is any even symmetric endomorphism of $\Lambda_\Sigma[2]$.\\

If $\langle a , c \rangle = 0$, choose an element $c' \in \Lambda_\Sigma[2]$ with $\langle c , c' \rangle = 1$ and define $F$ by $Fx = \langle a , x \rangle c' + \langle c' , x \rangle a$. Then $Fc = a$ and $(a+Fc , 0 , c) = (0,0,c)$ and so there is just one such orbit for each $c \neq 0$.\\

If $\langle a , c \rangle = \langle a' , c \rangle = 1$, we will show that there is a symmetric even endomorphism $F$ such that $Fc = a+a'$. Then $(a+Fc,0,c) = (a',0,c)$ and so there is just one orbit of this type. In fact, we can take $F$ to be given by $Fx = \langle a , x \rangle a' + \langle a' , x \rangle a$.\\

Now consider non-maximal orbits of the form $(a,b,0)$. Since $b \neq 0$ we can use monodromy to set $a$ to zero, so we just need to consider elements $(0,b,0)$. Since the monodromy acts on such elements by permutations of $B$, we find that there are exactly $l/2$ such non-maximal orbits. In total we have found $2.(2^{2g}-1) +l/2$ non-maximal orbits and by inspection they are seen to be distinguished their topological invariants. Together with the $2^{2g}$ maximal components this gives a total of $3.2^{2g} +(l-4)/2$ components.\\

{\bf Case (2): $SL(2,\mathbb{R})$.} The real points of a fibre is given by $\widetilde{\Lambda}_P[2] = \Lambda_\Sigma[2] \oplus (\mathbb{Z}_2 B)^{\rm ev}$. The $2.2^{2g}$ maximal orbits are those of the form $(a,0)$ and $(a,b_o)$ for $a \in \Lambda_\Sigma[2]$. The non-maximal orbits have representatives of the form $(0,b)$ and we find there are $(l-1)$ such orbits. Again, we see by inspection that the non-maximal orbits have distinct topological invariants, so the total number of connected components is $2.2^{2g} + l-1$.\\

{\bf Case (3): $PGL(2,\mathbb{R})$.} The real points of a fibre is $W \oplus \Lambda_\Sigma[2]$ for $d=0$ and $W^1 \oplus \Lambda_\Sigma[2]$ for $d=1$. There is a single maximal orbit $(0,0)$. Consider an element of the form $(b,c)$ with $c \neq 0$. By the monodromy action we can replace $b$ by $b + b'$ for any $b' \in (\mathbb{Z}_2 B)^{\rm ev}/(b_o)$. Thus we can assume $b = 0$ (for $d=0$) or $b = b_1$ (for $d=1$). Thus for either value of $d$, there are $2^{2g}-1$ such orbits. The remaining orbits have the form $(b,0)$ for $b \neq 0,b_o$. We find there are $l/2$ such orbits for each value of $d$. Once again, the non-maximal orbits have distinct topological invariants and so the total number of components is $2^{2g}+l/2$ for $d=0$ and $2^{2g}+l/2-1$ for $d=1$.\\

{\bf Case (4): $PSL(2,\mathbb{R})$.} The real points of a fibre are $(\mathbb{Z}_2 B)^{\rm ev}$ for $d=0$ and $(\mathbb{Z}_2 B)^{\rm odd}$ for $d=1$. There are two maximal orbits $0$ and $b_o$. There are a further $l-1$ non-maximal orbits when $d=0$ and $l$ non-maximal orbits when $d=1$. Yet again, the non-maximal orbits are distinguished by topological invariants so the number of connected components is $l+1$ for $d=0$ and $l$ for $d=1$.
\end{proof}

\begin{corollary}
Setting $L = K$, we have the number of connected components of the following real character varieties:
\begin{enumerate}
\item{$3.2^{2g}+g-3$ for $Rep(GL(2,\mathbb{R}))$}
\item{$2.2^{2g}+2g-3$ for $Rep(SL(2,\mathbb{R}))$}
\item{$2^{2g}+g-1$ for $Rep_0(PGL(2,\mathbb{R}))$ and $2^{2g}+g-2$ for $Rep_1(PGL(2,\mathbb{R}))$}
\item{$2g-1$ for $Rep_0(PSL(2,\mathbb{R}))$ and $2g-2$ for $Rep_1(PSL(2,\mathbb{R}))$}
\end{enumerate}
\end{corollary}

\begin{remark}
The number of components $2.2^{2g}+2g-3$ for $Rep(SL(2,\mathbb{R}))$ and $4g-3$ for \linebreak $Rep(PSL(2,\mathbb{R}))$ were shown by Goldman in \cite{goldm2}. Xia \cite{xia1,xia2} showed that the number of components of the space of homomorphisms $Hom( \pi_1(\Sigma) , PSL(2,\mathbb{R}) )$ is $2.2^{2g} + 4g-5$. This number is different to the number $2.2^{2g} + 2g-3$ of components of $Rep( PGL(2,\mathbb{R}) )$ because upon taking the quotient of the conjugation action of $PGL(2,\mathbb{R})$, certain pairs of components are identified.
\end{remark}

%%%%%%%%%%%%%%%%%%%%%%%%%%%%%%%%%%%%%%%%%%%%
\subsection{Components of maximal $Sp(4,\mathbb{R})$ representations}
%%%%%%%%%%%%%%%%%%%%%%%%%%%%%%%%%%%%%%%%%%%%

Let $\theta$ be a representation of $\pi_1(\Sigma)$ into $Sp(4,\mathbb{R})$. Since the maximal compact subgroup of $Sp(4,\mathbb{R})$ is $U(2)$, we can associated to $\theta$ an integer invariant $d$ called the {\em Toledo invariant}, defined as the degree of the $U(2)$-bundle obtained by a reduction of structure of the flat $Sp(4,\mathbb{R})$-bundle associated to $\theta$. Turaev \cite{tur} showed that the Toledo invariant satisfies an inequality, often referred to as a Milnor-Wood inequality:
\begin{equation*}
| d | \le (2g-2).
\end{equation*}
We say that a representation $\theta$ of $\pi_1(\Sigma)$ into $Sp(4,\mathbb{R})$ is {\em maximal} if it satisfies $|d| = (2g-2)$ and we let $Rep_{max}(Sp(4,\mathbb{R}))$ denote the subspace of $Rep( Sp(4,\mathbb{R}))$ consisting of maximal representations. We also write $Rep_{d}( Sp(4,\mathbb{R}))$ for the representations with fixed value of the Toledo invariant. It can easily be shown that $Rep_{d}( Sp(4,\mathbb{R}))$ is homeomorphic to $Rep_{-d}( Sp(4,\mathbb{R}))$ and $Rep_{max}( Sp(4,\mathbb{R})) = Rep_{2g-2}( Sp(4,\mathbb{R})) \cup Rep_{-(2g-2)}( Sp(4,\mathbb{R}))$. Using the Cayley correspondence of \cite{ggm} it can be shown that there is a homeomorphism between $Rep_{2g-2}(Sp(4,\mathbb{R}))$ and ${}^{\mathbb{R}} \mathcal{M}(K^2)$, the moduli space of $K^2$-twisted $GL(2,\mathbb{R})$-Higgs bundles. From Theorem \ref{thmcomponents}, we immediately obtain:

\begin{corollary}\label{cormaxsp}
The number of components of $Rep_{2g-2}(Sp(4,\mathbb{R}))$ is given by $3.2^{2g}+2g-4$.
\end{corollary}

\begin{remark}
Corollary \ref{cormaxsp} was shown by Gothen in \cite[Theorem 5.8]{got}.
\end{remark}

%%%%%%%%%%%%%%%%%%%%%%%%%%%%%
\section{Monodromy for $SO(2,2)$-Higgs bundles}\label{secso22}
%%%%%%%%%%%%%%%%%%%%%%%%%%%%%

In this section we will use our results on the monodromy of rank $2$ Higgs bundle moduli spaces to determine the monodromy for $SO(2,2)$-Higgs bundles. To begin, we let $\mathcal{M}^{SO(4,\mathbb{C})}$ denote the moduli space of semi-stable $SO(4,\mathbb{C})$-Higgs bundles and $h : \mathcal{M}^{SO(4,\mathbb{C})} \to \mathcal{A}^{SO(4,\mathbb{C})}$ the Hitchin fibration, where $\mathcal{A}^{SO(4,\mathbb{C})} := H^0(\Sigma , K^2) \oplus H^0(\Sigma , K^2)$. The moduli space has two connected components $\mathcal{M}^{SO(4,\mathbb{C})}(0)$, $\mathcal{M}^{SO(4,\mathbb{C})}(1)$ corresponding to the value of the second Stiefel-Whitney class $w_2 \in H^2(\Sigma , \mathbb{Z}_2) \simeq \mathbb{Z}_2$ of the underlying $SO(4,\mathbb{C})$-bundle.\\

As we are mainly concerned with the monodromy of the regular locus, we will omit discussion of semi-stability and pass directly to the spectral data description of $SO(4,\mathbb{C})$-Higgs bundles, as detailed in \cite{hit2}. Let $(a_2 , p ) \in \mathcal{A}^{SO(4,\mathbb{C})} = H^0(\Sigma , K^2) \oplus H^0(\Sigma , K^2)$ be a pair of quadratic differentials on $\Sigma$. Associated to the pair $(a_2,p)$ is a characteristic equation of the form:
\begin{equation}\label{equso4}
\lambda^4 + a_2 \lambda^2 + p^2 = 0.
\end{equation}
The curve $S \subset K$ defined by (\ref{equso4}) is always singular, but for generic pairs $(a_2,p)$, the singularities of $S$ are ordinary double points lying over the zeros of $p$. Let $\nu : S^\nu \to S$ be the normalisation of $S$. The involution $\sigma : S \to S$ given by $\lambda \mapsto -\lambda$ lifts to a free involution $\sigma^\nu : S^\nu \to S^\nu$. Let $\overline{S}$ be the quotient of $S^\nu$ by the action of $\sigma^\nu$ and let $\pi^{\nu} : S^\nu \to \overline{S}$ be the projection. Then $\overline{S}$ can also be identified with the quotient of $S$ by $\sigma$. Let $\overline{\pi} : K^2 \to \Sigma$ be the projection from the total space of $K^2$ and $y$ the tautological section of $\overline{\pi}^* K^2$. Then $\overline{S} \subset K^2$ is given by the equation
\begin{equation}\label{equsbar}
y^2 + a_2 y + p^2 = 0.
\end{equation}
In particular, $\overline{S}$ is smooth if and only if the discriminant $\Delta = a_2^2 - 4p^2$ has only simple zeros. The regular locus $\mathcal{A}^{SO(4,\mathbb{C})}_{\rm reg}$ of the base $\mathcal{A}$ is precisely the set of points where $\overline{S}$ is smooth and in this case the fibre of the Hitchin system lying over $(a_2,p) \in \mathcal{A}^{SO(4,\mathbb{C})}_{\rm reg}$ is given by the Prym variety of the cover $\pi^\nu : S^\nu \to \overline{S}$. To be more precise, let us define $Prym(S^\nu , \overline{S})$ by:
\begin{equation*}
Prym(S^\nu , \overline{S} ) = \{ M \in Jac(S^\nu) \; | \; Nm(M) = \mathcal{O} \}.
\end{equation*}
Since the double covering $S^\nu \to \overline{S}$ has no branch points, we have that $Prym(S^\nu , \overline{S} )$ is a complex group having two connected components $Prym_0(S^\nu , \overline{S})$ and $Prym_1(S^\nu , \overline{S})$. The identity component $Prym_0(S^\nu , \overline{S})$ is an abelian variety and $Prym_1(S^\nu , \overline{S})$ has the structure of a $Prym_0(S^\nu , \overline{S})$-torsor. If $(a_2,p) \in \mathcal{A}^{SO(4,\mathbb{C})}_{\rm reg}$ with corresponding smooth curves $\pi^\nu : S^\nu \to \overline{S}$, then the fibre of $\mathcal{M}^{SO(4,\mathbb{C})}(a)$ lying over $(a_2,p)$ can be identified with the component $Prym_a(S^\nu , \overline{S})$ of the Prym variety.\\

Given $SO(2,2)$-the split real form of $SO(4,\mathbb{C})$, we consider  the moduli space $\mathcal{M}^{SO(2,2)}$ of $SO(2,2)$-Higgs bundles. We have a naturally defined Hitchin map $\mathcal{M}^{SO(2,2)} \to \mathcal{A}^{SO(4,\mathbb{C})}$ given by the composition of the  map $\mathcal{M}^{SO(2,2)} \to \mathcal{M}^{SO(4,\mathbb{C})}$ with the Hitchin map $\mathcal{M}^{SO(4,\mathbb{C})} \to \mathcal{A}^{SO(4,\mathbb{C})}$. Let $\mathcal{M}^{SO(2,2)}_{\rm reg}$ be the points of $\mathcal{M}^{SO(2,2)}$ lying over $\mathcal{A}^{SO(4,\mathbb{C})}_{\rm reg}$. From \cite[Theorem 4.12]{sch0}, in the case of $SO(2,2)$, spectral data over a point $(a_2,p) \in \mathcal{A}^{SO(4,\mathbb{C})}_{\rm reg}$ consists of a pair $(M,\tilde{\sigma}^\nu)$, where $M$ is a line bundle of order $2$ and $\tilde{\sigma}^\nu$ is an involutive lift of $\sigma^\nu$. Now since $\sigma^\nu$ acts freely, we see that such pairs correspond simply to line bundles on $\overline{S}$ of order $2$, i.e. the space $\Lambda_{\overline{S}}[2]$. We have thus proven the following:
\begin{theorem}
The bundle of groups $\mathcal{M}^{SO(2,2)}_{\rm reg} \to \mathcal{A}^{SO(4,\mathbb{C})}_{\rm reg}$ is the pullback of ${}^{\mathbb{R}} \mathcal{M}(K)_{\rm reg} \to \mathcal{A}_{\rm reg}(K)$ under the map $j : \mathcal{A}^{SO(4,\mathbb{C})}_{\rm reg} \to \mathcal{A}_{\rm reg}(K)$ given by $j(a_2,p) = (a_2 , p^2)$. In particular, the monodromy $\rho^{SO(2,2)} : \pi_1( \mathcal{A}^{SO(4,\mathbb{C})}_{\rm reg} ) \to Aut( \Lambda_{\overline{S}}[2] )$ of the $SO(2,2)$-Hitchin system is determined by the following commutative diagram:
\begin{equation*}\xymatrix{
& & Aut( \Lambda_{\overline{S}}[2] ) \\
\pi_1( \mathcal{A}^{SO(4,\mathbb{C})}_{\rm reg} , (a_2,p) ) \ar[rr]^-{j_*} \ar[urr]^-{\rho^{SO(2,2)}} & & \pi_1( \mathcal{A}_{\rm reg}(K) , (a_2 , p^2) ) \ar[u]^-{\rho}
}
\end{equation*}
\end{theorem}

Recall that the double cover $\overline{\pi} : \overline{S} \to \Sigma$ defined by the pair $(a_2,p)$ is smooth if and only if the discriminant $\Delta = a_2^2 - 4p^2$ has only simple zeros. Let us define $c_1 := a_2 - 2p$, $c_2 := a_2 + 2p$, so that $\Delta = c_1c_2$. Then $a_2 = (c_1+c_2)/2$ and $p = (c_2-c_1)/4$, so the pair $(c_1,c_2) \in H^0(\Sigma , K^2) \oplus H^0(\Sigma , K^2)$ uniquely determines the pair $(a_2,p)$. Moreover, $\overline{S}$ is smooth if and only if $c_1$ and $c_2$ have simple zeros and no zeros in common. Let $B_1 = \{ b_1 , \dots , b_{4g-4} \}$ be the set of zeros of $c_1$ and $B_2 = \{ b_{4g-3} , \dots , b_{8g-8} \}$ the set of zeros of $c_2$. Then $B = B_1 \cup B_2$ is the set of branch points of $\overline{\pi} : \overline{S} \to \Sigma$.\\

We now look for loops in $\mathcal{A}_{\rm reg}(K)$ that can be realised as the image under $j$ of loops in $\mathcal{A}^{SO(4,\mathbb{C})}_{\rm reg}$. Consider a swap of $b_i$ and $b_j$ along a path $\gamma$. Using the lifting procedure as described in Section \ref{secfgc}, this gives a loop $\Delta(t)$ within $H^0(\Sigma , K^4 )^{\rm simp}$. In order for this loop to come from a loop in $\mathcal{A}^{SO(4,\mathbb{C})}_{\rm reg}$, we need to be able to find loops $c_1(t) , c_2(t) \in H^0(\Sigma , K^2)^{\rm simp}$ satisfying $\Delta(t) = c_1(t)c_2(t)$. To do this, it is clearly necessary that $b_i,b_j$ both belong to $B_1$ or both belong to $B_2$. Conversely, suppose that $b_i,b_j$ both belong to $B_1$ (the case of $B_2$ is similar). Then by considering $B_1$ alone, $\gamma$ defines a braid in $\Sigma$ with $4g-4$ strands, the swap of $b_i,b_j$ along $\gamma$. Using our lifting procedure, we obtain a loop $c_1(t) \in H^0(\Sigma , K^2)^{\rm simp}$. If we take $c_2(t)$ to be the constant loop and set $\Delta(t) = c_1(t)c_2(t)$, then we have the desired factorisation. In summary, if $\gamma$ is an embedded path from $b_i$ to $b_j$ and $b_i,b_j$ both belong to $B_1$ or $B_2$, then the lifted swap $\tilde{s}_\gamma$ may be realised as a loop in $\pi_1( \mathcal{A}^{SO(4,\mathbb{C})}_{\rm reg} , (a_2,p) )$.\\

Let us say that an embedded path $\gamma$ joining $b_i$ to $b_j$ is {\em admissible} if $b_i,b_j$ both belong to $B_1$ or to $B_2$. We now proceed to compute the monodromy action on $\Lambda_{\overline{S}}[2]$ exactly as in Section \ref{secmonoact}, except that we only allow for admissible loops. Thus we can choose splittings so that 
\begin{equation*}
\Lambda_{\overline{S}}[2] = \Lambda_\Sigma[2] \oplus W \oplus \Lambda_\Sigma[2],
\end{equation*}
with $W = (\mathbb{Z}_2 B)^{\rm ev}/(b_o)$, $b_o = b_1 + b_2 + \dots + b_{8g-8}$, and the monodromy is generated by transformations $s_{ij}$ and $A_{ij}^x$ as in Equations (\ref{equtransposition})-(\ref{equaijx}). The only difference now is that we must restrict the indices $i,j$ to satisfy $1 \le i < j \le 4g-4$ or $4g-3 \le i < j \le 8g-8$.\\

Recall from Section \ref{secmonoact} that $\Lambda_{\overline{S}}[2]$ is equipped with the Weil pairing $\langle \; , \; \rangle : \Lambda_{\overline{S}}[2] \otimes \Lambda_{\overline{S}}[2] \to \mathbb{Z}_2$ and quadratic refinement $q : \Lambda_{\overline{S}}[2] \to \mathbb{Z}_2$. Then since $K^2$ has even degree, the monodromy action of the $s_{ij}$ and $A_{ij}^x$ must preserve $q$. We are now ready to state the main theorem of this section:

\begin{theorem}
Let $G \subseteq GL( \Lambda_{\overline{S}}[2] )$ be the group generated by the monodromy action of $\rho^{SO(2,2)}$ on $\Lambda_{\overline{S}}[2]$. Then:
\begin{itemize}
\item[(1)]{$G$ is isomorphic to a semi-direct product $G = \left( S_{4g-4} \times S_{4g-4} \right) \ltimes H$ of the product of symmetric groups $S_{4g-4} \times S_{4g-4}$, generated by the elements $\{ s_{ij} \; | \; 1 \le i < j \le 4g-4, \; \text{or } 4g-3 \le i < j \le 8g-8 \; \}$ given in (\ref{equtransposition}) and the group $H$ generated by the transformations $\{ A_{ij}^x \; | \; 1 \le i < j \le 4g-4, \; \text{or } 4g-3 \le i < j \le 8g-8, \; x \in \Lambda_\Sigma[2] \; \}$ given in (\ref{equaijx}).}
\item[(2)]{Let $K$ be the subgroup of elements of $GL( \Lambda_{\overline{S}}[2])$ of the form:
\begin{equation}
\left[ \begin{matrix} I_{2g} & A & B \\ 0 & I & A^t \\ 0 & 0 & I_{2g} \end{matrix} \right],
\end{equation}
where $A : W \to \Lambda_\Sigma[2]$, $B : \Lambda_\Sigma[2] \to \Lambda_\Sigma[2]$,  and $A^t : \Lambda_\Sigma[2] \to W$ is the adjoint of $A$, so $\langle Ab , c \rangle = (( b , A^tc ))$. Then $H$ is the subgroup of $K$ such that $A(b_1 + b_2 + \dots + b_{4g-4}) = 0$ and such that the quadratic refinement $q$ is preserved, i.e:
\begin{equation*}
\langle Bc , c \rangle + q_W( A^t c ) = 0.
\end{equation*}
}
\end{itemize}
\end{theorem}
\begin{proof}
Let $H'$ be the subgroup of $K$ satisfying $A(b_1 + b_2 + \dots + b_{4g-4}) = 0$ and preserving the quadratic refinement $q$. By an argument similar to the proof of Theorem \ref{thmmonogp}, we can easily show that $H' = H$. It remains to show that any element of $GL( \Lambda_{\overline{S}}[2])$ obtained through monodromy belongs to the group $G = \left( S_{4g-4} \times S_{4g-4} \right) \ltimes H'$.\\

Let $T \in GL( \Lambda_{\overline{S}}[2])$ be in the image of the monodromy representation. Then $T$ preserves $\overline{\pi}_*,\overline{\pi}^*$, so must have the form
\begin{equation*}
T = \left[ \begin{matrix} I & T_{12} & T_{13} \\ 0 & T_{22} & T_{23} \\ 0 & 0 & I \end{matrix} \right].
\end{equation*}
From the discussion at the beginning of Section \ref{secmonoact} relating points of order $2$ in the Prym variety with the space $W = (\mathbb{Z}_2 B)^{\rm ev}/(b_o)$, we see that $T_{22}$ must act on $W$ through a permutation of $B$. Moreover this permutation must preserve the zero sets of $c_1,c_2$, so $T_{22}$ belongs to $S_{4g-4} \times S_{4g-4}$. After composing with a product of transpositions $s_{ij}$, we may assume $T_{22}$ is the identity. Moreover, $T$ preserves the Weil pairing, so $T_{23}$ is the adjoint of $T_{12}$. To complete the proof we just need to show that $T_{12}(b_1 + \dots + b_{4g-4} ) = 0$ and that $T$ preserves the quadratic refinement $q$. In fact, the point $( 0 , b_1 + \dots + b_{4g-4} , 0)$ can be shown to correspond to an $SO(2,2)$-Higgs bundle obtained as the tensor product $V_1 \otimes V_2$ of two maximal $SL(2,\mathbb{R})$-Higgs bundles and thus must be preserved by monodromy. To show that $T$ preserves $q$, one can show that the function $q$ is related to a characteristic class of the corresponding $SO(2,2)$-Higgs bundles. We omit the details, as they are very similar to the $GL(2,\mathbb{R})$ case of Proposition \ref{propspecinv}. Thus $T$ must preserve $q$.
\end{proof}

\begin{remark}
The character variety $Rep( SO(2,2) )$ can also be studied through low rank isogenies as done in \cite[Section 4]{steve}, where $SO(2,2)$-Higgs bundles are obtained through fibre product of spectral curves of two $SL(2,\mathbb{R})$-Higgs bundles. Hence, the monodromy representation for the rank $4$ Hitchin system can also be studied through the monodromy for the two rank $2$ Hitchin systems. In particular, one finds from this point of view that every component of $Rep( SO(2,2) )$ meets the regular locus.
\end{remark}

Using an argument similar to the proof of Theorem \ref{thmcomponents}, we can deduce the number of components of the $SO(2,2)$-character variety by counting the number of orbits of the monodromy action on $\Lambda_{\overline{S}}[2]$, giving:

\begin{corollary}
The character variety $Rep( SO(2,2) )$ has $6.2^{2g} + 4g^2 - 6g - 3$ components.
\end{corollary}

%%%%%%%%%%%%%%%%%%%%%%%%%%%%%%%%%%%%%%%%%%%%%%%%%%%%%%%%%%%%%%%%%%%%%%%%%%%%%%%%
%%%%%%%%%%%%%%%%%%%%%%%%%%%%%%%%%%%%%%%%%%%%%%%%%%%%%%%%%%%%%%%%%%%%%%%%%%%%%%%%
%%%%%%%%%%%%%%%%%%%%%%%%%%%%%%%%%%%%%%%%%%%%%%%%%%%%%%%%%%%%%%%%%%%%%%%%%%%%%%%%
%%%%%%%%%%%%%%%%%%%%%%%%%%%%%%%%%%%%%%%%%%%%%%%%%%%%%%%%%%%%%%%%%%%%%%%%%%%%%%%%

\bibliographystyle{amsplain}

\begin{thebibliography}{99}
\bibitem{agzv}V. I. Arnold, S. M. Gusein-Zade, A. N. Varchenko, Singularities of differentiable maps. Vol. II. Monodromy and asymptotics of integrals. Translated from the Russian by Hugh Porteous. Translation revised by the authors and James Montaldi. Monographs in Mathematics, 83. Birkh\"auser Boston, Inc., Boston, MA, 1988. 492 pp.
\bibitem{ati}M. F. Atiyah, Riemann surfaces and spin structures. {\em Ann. Sci. \'Ecole Norm. Sup.} Vol. {\bf 4}, no. 4 (1971) 47-62. 
\bibitem{atbo}M. F. Atiyah, R. Bott, A Lefschetz fixed point formula for elliptic complexes. II. Applications. {\em Ann. of Math.} (2) {\bf 88} (1968) 451-491. 
\bibitem{bar1t}D. Baraglia, Topological T-duality for general circle bundles. {\em Pure Appl. Math. Q.} Vol. {\bf 10}, no. 3, 367-438 (2014).
\bibitem{bar2t}D. Baraglia, Topological T-duality for torus bundles with monodromy. {\em Rev. Math. Phys.} Vol. {\bf 27}, no. 3 (2015) 1550008.
\bibitem{bnr}A. Beauville, M. S. Narasimhan, S. Ramanan, Spectral curves and the generalised theta divisor. {\em J. Reine Angew. Math.} {\bf 398} (1989), 169-179.
\bibitem{bila}C. Birkenhake, H. Lange, Complex abelian varieties. Second edition. Grundlehren der Mathematischen Wissenschaften, {\bf 302.} Springer-Verlag, Berlin, 2004. 635 pp.
\bibitem{bir}J. S. Birman, Braids, links, and mapping class groups. Annals of Mathematics Studies, No. {\bf 82}. Princeton University Press, 1974. 228 pp. 
\bibitem{bgm}S. B. Bradlow, O. Garc\'ia-Prada, I. Mundet i Riera, Relative Hitchin-Kobayashi correspondences for principal pairs. {\em Q. J. Math.} {\bf 54} (2003), no. 2, 171-208.
\bibitem{steve}S. B. Bradlow, L. P. Schaposnik, Higgs bundles and exceptional isogenies, preprint, (2015).
\bibitem{chm}M. A. A. de Cataldo, T. Hausel, L. Migliorini, Topology of Hitchin systems and Hodge theory of character varieties: the case $A_1$. {\em Ann. of Math.} (2) {\bf 175} (2012), no. 3, 1329-1407.
\bibitem{cop0}D. J. Copeland, A special subgroup of the surface braid group. arXiv:0409461 (2004).
\bibitem{cop}D. J. Copeland, Monodromy of the Hitchin map over hyperelliptic curves. {\em Int. Math. Res. Not.} (2005), no. 29, 1743-1785.
\bibitem{cor}K. Corlette, Flat $G$-bundles with canonical metrics. {\em J. Differential Geom.} {\bf 28} (1988), no. 3, 361-382.
\bibitem{doli}I. Dolgachev, A. Libgober, On the fundamental group of the complement to a discriminant variety. In: Algebraic geometry, Lecture Notes in Math. 862, 1-25, Springer, Berlin-New York, 1981. 
\bibitem{don}S. K. Donaldson, Twisted harmonic maps and the self-duality equations. {\em Proc. London Math. Soc.} (3) {\bf 55} (1987), no. 1, 127-131.
\bibitem{ggm0}O. Garc\'ia-Prada, P. B. Gothen, I. Mundet i Riera, The Hitchin-Kobayashi correspondence, Higgs pairs and surface group representations. arXiv:0909.4487v3 (2012).
\bibitem{ggm}O. Garc\'ia-Prada, P. B. Gothen, I. Mundet i Riera, Higgs bundles and surface group representations in the real symplectic group. {\em J. Topol.} {\bf 6} (2013), no. 1, 64-118.
\bibitem{goldm}W. M. Goldman, The symplectic nature of fundamental groups of surfaces. {\em Adv. in Math.} {\bf 54} (1984), no. 2, 200-225.
\bibitem{goldm2}W. M. Goldman, Topological components of spaces of representations. {\em Invent. Math.} {\bf 93} (1988), no. 3, 557-607. 
\bibitem{got}P. B. Gothen, Components of spaces of representations and stable triples. {\em Topology} {\bf 40} (2001), no. 4, 823-850. 
\bibitem{hit1}N. J. Hitchin, The self-duality equations on a Riemann surface. {\em Proc. London Math. Soc.} (3) {\bf 55} (1987), no. 1, 59-126. 
\bibitem{hit2}N. J. Hitchin, Stable bundles and integrable systems. {\em Duke Math. J.} {\bf 54} (1987), no. 1, 91-114. 
\bibitem{hit3}N. J. Hitchin, Higgs bundles and characteristic classes. arXiv:1308.4603, (2013).
\bibitem{nit}N. Nitsure, Moduli space of semistable pairs on a curve. {\em Proc. London Math. Soc.} (3) {\bf 62} (1991), no. 2, 275-300. 
\bibitem{ric}R. W. Richardson, Conjugacy classes of $n$-tuples in Lie algebras and algebraic groups. {\em Duke Math. J.} {\bf 57} (1988), no. 1, 1-35. 
\bibitem{sch0}L. P. Schaposnik, Spectral data for $G$-Higgs bundles. DPhil thesis. arXiv:1301.1981 (2013).
\bibitem{sch1}L. P. Schaposnik, Monodromy of the $SL2$ Hitchin fibration. {\em Internat. J. Math.} {\bf 24} (2013), no. 2, 1350013, 21 pp. 
\bibitem{sim1}C. Simpson, Constructing variations of Hodge structure using Yang-Mills theory and applications to uniformization. {\em J. Amer. Math. Soc.} {\bf 1} (1988), no. 4, 867-918.
\bibitem{sim2}C. Simpson, Higgs bundles and local systems. {\em Inst. Hautes \'Etudes Sci. Publ. Math.} no. {\bf 75} (1992), 5-95.
\bibitem{tur}V. G. Turaev, A cocycle of the symplectic first Chern class and Maslov indices. {\em Funct. Anal. Appl.} {\bf 18} (1984), no. 1, 35-39.
\bibitem{walk}K. C. Walker, Quotient groups of the fundamental groups of certain strata of the moduli space of quadratic differentials. {\em Geom. Topol.} {\bf 14} (2010), no. 2, 1129-1164.
\bibitem{xia1}E. Z. Xia, Components of $Hom(\pi_1, PGL(2,\mathbb{R}))$. {\em Topology} {\bf 36} (1997), no. 2, 481-499.
\bibitem{xia2}E. Z. Xia, The moduli of flat $PGL(2,\mathbb{R})$ connections on Riemann surfaces. {\em Comm. Math. Phys.} {\bf 203} (1999), no. 3, 531-549. 
\end{thebibliography}

\end{document}